\documentclass[a4paper,10pt]{amsart}
\usepackage{amsmath,amsfonts,amssymb,mathrsfs,esint}
\usepackage{enumitem,verbatim}
\usepackage{graphicx,color}
\usepackage{longtable}

\newtheorem{lemma}{Lemma}
\newtheorem{proposition}{Proposition}

\newtheorem{theorem}{Theorem}
\theoremstyle{remark}
\newtheorem{remark}{Remark}
\theoremstyle{definition}
\newtheorem{definition}{Definition}

\newcommand{\trema}{\"}
\newcommand{\R}{\ensuremath{\mathbb{R}}}

\newcommand{\diese}{{\scriptscriptstyle \#}}
\newcommand{\per}{{\rm \scriptscriptstyle per}}

\newcommand{\ext}{\mathrm{ext}}
\newcommand{\dif}{\mathrm{d}}

\newcommand{\loc}{{\rm \scriptscriptstyle loc}}
\newcommand{\transpose}{{\rm t}}

\DeclareMathOperator*{\esssup}{ess \, sup}
\DeclareMathOperator*{\essinf}{ess \, inf}
\DeclareMathOperator{\Div}{div}

\DeclareMathOperator{\supp}{supp}
\DeclareMathOperator*{\cO}{\mathit{O}}
\DeclareMathOperator*{\co}{\mathit{o}}

\begin{document}
\title[Front propagation by an Arrhenius kinetic]{On the propagation of a periodic flame front by an Arrhenius kinetic}
\author{Natha{\"e}l Alibaud}
\author{Gawtum Namah}
\address[Natha\"{e}l Alibaud and Gawtum Namah]{UMR CNRS 6623, Universit\'e de Franche-Comt\'e\\
16 route de Gray\\
25 030 Besan\c{c}on cedex, France}
\email{nathael.alibaud@ens2m.fr, gawtum.namah@ens2m.fr}
\thanks{During this research, the first author was supported by the ``French ANR
project CoToCoLa, no. ANR-11-JS01-006-01.''} 
\date{\today}
\subjclass[2010]{Primary 
35R35, 80A25, 35C07, 35B10; Secondary 35B27, 80M35.}
\keywords{Free boundary problems, front propagation, combustion, Arrhenius law, travelling wave solutions, periodic solutions, homogenization, curvature effects, asymptotic analysis}

\begin{abstract}
We consider the propagation of a flame front in a solid medium with a periodic structure. The model is governed by a free boundary system for the pair ``temperature-front.'' The front's normal velocity depends on the temperature via a (degenerate) Arrhenius kinetic. It also depends on 
the front's mean curvature. We show the existence of travelling wave solutions for the full system and consider their homogenization as the period tends to zero. We analyze the curvature effects on the homogenization and obtain a continuum of limiting waves parametrized by the limiting ratio ``curvature coefficient/period.'' This analysis provides valuable information on the heterogeneous propagation as well.
\end{abstract}
\maketitle
\tableofcontents

\section{Introduction}

We investigate the propagation of a flame front in a solid heterogeneous
medium $\mathbb{R}_{x}\times\mathbb{R}_{y}$. Throughout $Y>0$ is a fixed
period and we consider a solid with horizontal striations which are
$Y$-periodic in $y$. The fresh region is assumed to be a hypograph~$\{x<\xi
(y,t)\}$ with a temperature $T=T(x,y,t)$. The flame front $\{x=\xi(y,t)\}$ is
assumed to propagate to the left. A typical representation is given in Figure \ref{fig-striations}, which shows the propagation through a medium consisting of a periodic superposition of two layers.

\begin{figure}[ht!]
\begin{center}
\scalebox{0.40}{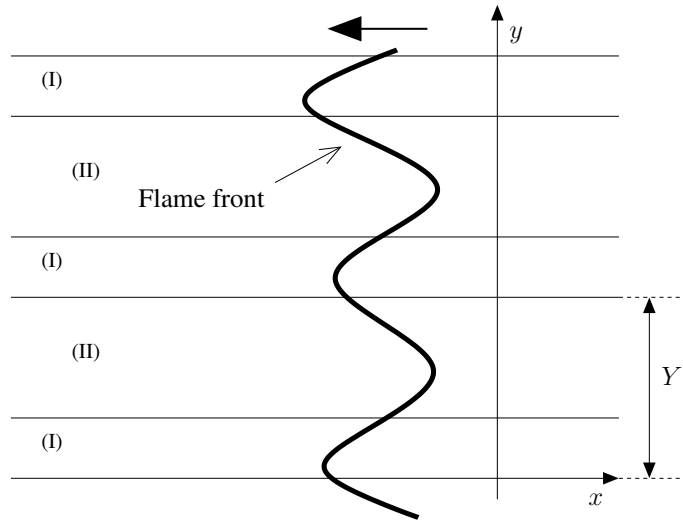}
\caption{Periodic superposition of two materials (I) and (II).}
\label{fig-striations}
\end{center}
\end{figure}

We consider the model where the evolution of
$(\xi,T)$ is governed by the free boundary system
\begin{equation}%
\begin{cases}
b\,T_{t}-\Div(a\,\nabla T)=0, & \quad x<\xi(y,t),~t>0,\\[1ex]
\xi_{t}+R(y,T)\sqrt{1+\xi_{y}^{2}}=\mu\,\frac{\xi_{yy}}{1+\xi_{y}^{2}}, &
\quad x=\xi(y,t),~t>0,
\end{cases}
\label{free-bdry-sys}%
\end{equation}
subject to the boundary conditions
\begin{equation}%
\begin{cases}
a\,\frac{\partial T}{\partial\nu}=g\,V_{n}, & \quad x=\xi(y,t),\\[1ex]
T(x,y,t)\rightarrow0, & \quad \mbox{as~$x\rightarrow -\infty$},
\end{cases}
\label{bdry-cond}%
\end{equation}
where $\nu=\frac{(1,-\xi_{y})}{\sqrt{1+\xi_{y}^{2}}}$ is the outward unit
normal and $V_{n}$ is the normal velocity of the front. Throughout ``$\Div$,'' ``$\nabla$'' and ``$\frac{\partial}{\partial \nu}$'' denote the divergence, gradient and normal derivative operators, respectively. In \eqref{bdry-cond}, $T \to 0$ as $x \to -\infty$ because the steady state temperature of the fresh region is normalized to zero. The second
equation of \eqref{free-bdry-sys} just states that the front propagates with a
normal velocity $V_{n}$ given by
\[
V_{n}=-R(y,T)-\mu\,\kappa,
\]
where $\kappa$ is the mean curvature and $\mu$ is a positive {\it curvature coefficient.} This latter coefficient is related to the {\it surface tension}\footnote{Rigourously speaking, as here we have a solid/gas interface, we must rather talk of {\it interface energy} instead of {\it surface tension} which is usually used for liquid interfaces.} effects.
The propagation is not just a
geometric one (as it was the case for example in \cite{ChNa97,NaRo97,ChNa02}) because here the {\it combustion rate} $R=R(y,T)$ depends also on the temperature at the front. This dependence is typically given by an {\it Arrhenius kinetic} of the form:
\begin{equation}\label{combustion-type}
R=A \, e^{-\frac{E}{T}}, 
\end{equation}
where $A$ is a {\it prefactor} and $E$ is related to the {\it activation energy}. These parameters can depend on the layers which means that $A=A(y)$ and $E=E(y)$. 
As far as the other data are concerned, $a=a(y)$ represents the {\it thermal diffusivity,} $b=b(y)$ the {\it heat capacity}, and $g=g(y)$ represents a fraction of the total {\it heat release} and which serves to heat the solid thus making the combustion self-sustained. For more details about this model see
\cite{BrFiNaSc93,ChNa97,NaRo97,ChNa02} and the references therein.

\medskip

From now we assume that each parameter $f=a,b,g$ satisfies: 
\begin{enumerate}[label={{\rm (A\arabic*)}}]
\item \label{hyp-param-1} The function $f:\R \to \R$ is measurable and $Y$-periodic.
\smallskip
\item \label{hyp-param-2} There are constants $f_M \geq f_m >0$ such that 
$$
f_{M}\geq f(y)\geq f_{m} \quad \mbox{for almost each $y \in \R$}.
$$
\end{enumerate}
The parameter $R$ is assumed to be a function from $\R \times \R^+$ into $\R^+$ such that:

\begin{enumerate}[label={{\rm (A\arabic*)}}]
\setcounter{enumi}{2}
\item\label{hyp-arrh-1} For almost each $y \in \R$, $T \mapsto R(y,T)$ is continuous and nondecreasing.
\smallskip
\item\label{hyp-arrh-2} For each $T >0$, $y \to R(y,T)$ is measurable and $Y$-periodic.
\smallskip
\item\label{hyp-arrh-3} There is a constant $R_M>0$ such that 
$$
R_M \geq R(y,T) >0 \quad \mbox{for almost each $y \in \R$ and all $T>0$.}
$$ 
\item \label{hyp-arrh-4} For any $T >0$, we have: 
$$
\essinf_{y \in \R} R(y,T)>0.
$$
\end{enumerate}
In the above, $\mathbb{R}^{+}$ denotes the interval $(0,+\infty)$. The function $R$ is in particular of Carath\'eodory's type by \ref{hyp-arrh-1}--\ref{hyp-arrh-2}. Note that \ref{hyp-arrh-1}--\ref{hyp-arrh-4} are satisfied by the typical combustion rate in \eqref{combustion-type} whenever $A$ and $E$ satisfy \ref{hyp-param-1}--\ref{hyp-param-2}. This is exactly what occurs for the striated solid medium we consider where all these parameters are constant on the layers. They therefore take just two values for the case presented in Figure \ref{fig-striations}. Let us point out
at this stage that 
for reasons associated to mathematical analysis, 
the existing papers usually impose a 
``positive
lower bound to the front's speed''---that is to $R$ in our case---, see for instance \cite{BrFiNaSc93,ChNa97,NaRo97,ChNa02,LoCh09,ChJi13}. In our setting, this might be not satisfactory because the typical $R$ in \eqref{combustion-type} decays towards zero as $T \downarrow 0$. Actually, we will be brought to occasionally prescribe a ``slower decay'' to $R$ at the neighborhood of $T=0$, but we will always treat situations
where it may go to zero. Moreover a large part of our results work for the  Arrhenius law given in~\eqref{combustion-type}.

Let us continue with some general comments on the literature. Models of similar type are considered in \cite{BrFiNaSc93,ChNa97,NaRo97,ChNa02,LoCh09} but with slightly different assumptions. In \cite{BrFiNaSc93}, the striations are vertical and the front's profile is a straight line. As a consequence its equation reduces to an ODE. In \cite{ChNa97,NaRo97,ChNa02,LoCh09}, the propagation is purely geometric (that is to say $R=R(y)$ only) and the analysis concerns the sole front's equation. In the more recent paper \cite{ChJi13}, a full free boundary system somewhat similar to ours has been
studied in the context of solidification process in crystal growth. The front's propagation is governed by a Hamilton-Jacobi equation (where the surface tension effects are neglected). The main purpose of \cite{ChJi13} concerns the global
in time existence of a solution of the Cauchy problem via the study of the
regularity of the expanding front. 
The front's velocity considered 
is relaxed
regularitywise: It is just H\"older continuous as compared to the natural
Lipschitz regularity used for Hamilton Jacobi equations. In all these references, 
the speed of propagation is assumed to have a positive lower bound. 


In this paper, we focus on the study of travelling wave solutions
to~\eqref{free-bdry-sys}--\eqref{bdry-cond}. Our first purpose is to show the existence of such solutions. This comes to looking for fronts
and temperatures of the form%
$$
\xi(y,t)=-c\,t+v(y)\quad\mbox{and}\quad T(x,y,t)=u(x+c\,t,y),
$$
where $c>0$ will be the speed of the wave and $v$ its profile. It is
convenient to fix the front through the change of variable~$x+c\,t\mapsto x$.
This leads to the problem of finding a triplet~$(c,v,u)$ such that%
\begin{equation}%
\begin{cases}
c\,b\,u_{x}-\Div(a\,\nabla u)=0, & \quad x<v(y),\\[1ex]
a\,\frac{\partial u}{\partial\nu}=\frac{c\,g}{\sqrt{1+v_{y}^{2}}}, & \quad
x=v(y),\\[1ex]
u(x,y)\rightarrow0, & \quad \mbox{as~$x\rightarrow -\infty$,}
\end{cases}
\label{T-eqn}%
\end{equation}
and%
\begin{equation}
-c+R(y,u) \sqrt{1+v_{y}^{2}}=\mu\,\frac{v_{yy}}{1+v_{y}^{2}},\quad
x=v(y).
\label{front-eqn}%
\end{equation}
Note that if there exists a travelling
wave, the profile $v$ will be defined up to an additive constant. For simplification and
without loss of generality, we will be brought in the course of the analysis
to fix this constant.

Let us recall that Equation \eqref{front-eqn} has been the object of a thorough study in
\cite{ChNa97} but only for $R=R(y)$.
The existence of 
%
a travelling wave $(c,v)$ is proved
with some characterization of the speed $c$ with respect to the curvature
coefficient $\mu$. In \cite{NaRo97,ChNa02},
this analysis is extended to oblique striations and in \cite{LoCh09} to almost periodic
media. In this paper, we will establish the existence of a nontrivial travelling wave
solution for the whole system \eqref{free-bdry-sys}--\eqref{bdry-cond}.
If we consider a general $R$ satisfying \ref{hyp-arrh-1}--\ref{hyp-arrh-4}, we will need to assume the 
period (or more precisely the ratio period/$\mu$)
to be small enough. This includes in particular the combustion rate in \eqref{combustion-type}. For general 
periods 
we will need to consider a more restrictive class of (degenerate) $R$, see \eqref{hyp-lim}. 

We unfortunately do not know whether these travelling wave solutions are unique. Nevertheless we can give some precise characterization by considering their homogenization as the period tends to zero. This is the second purpose of this paper. Let us first mention that a similar homogenization problem has been considered in \cite{BrFiNaSc93} for media with vertical striations. A remarkable difference with our setting is the absence of surface tension effects (recall that the front's profile is a straight line in that case). Here we propose an homogenization analysis that will provide information on the curvature effects too. For that we allow the curvature coefficient to depend on the period. This amounts to consider a family of triplets $(c^\varepsilon,v^\varepsilon,u^\varepsilon)$ satisfying: 
\begin{equation}%
\begin{cases}
c^{\varepsilon}\,b\left(\frac{y}{\varepsilon}\right) u_{x}^{\varepsilon}-\Div\left(
a\left(\frac{y}{\varepsilon}\right)\nabla u^{\varepsilon}\right)  =0, & \quad x<v^{\varepsilon
}(y),\\[1ex]
a\left(\frac{y}{\varepsilon}\right)\frac{\partial u^{\varepsilon}}{\partial\nu}=\frac
{c^{\varepsilon}\,g\left(\frac{y}{\varepsilon}\right)}{\sqrt{1+(v_{y}^{\varepsilon})^{2}}}, &
\quad x=v^{\varepsilon}(y),\\[1ex]
u^{\varepsilon}(x,y)\rightarrow0, & \quad\mbox{as~$x\rightarrow -\infty$,}
\end{cases}
\label{eps-T-eqn}%
\end{equation}
and%
\begin{equation}
-c^{\varepsilon}+R\left(\frac{y}{\varepsilon},u^{\varepsilon}\right)\sqrt{1+(v_{y}%
^{\varepsilon})^{2}}=\mu(\varepsilon) \,\frac{v_{yy}^{\varepsilon}}{1+(v_{y}^{\varepsilon
})^{2}},\quad x=v^{\varepsilon}(y),
\label{eps-front-eqn}%
\end{equation}
for some given $1$-periodic parameters $a$, $b$, etc., so that $\varepsilon$ is the new period of the medium.
We will show that the triplet $(c^\varepsilon,v^\varepsilon,u^\varepsilon)$ converges towards some $(c^0,v^0,u^0)$, as $\varepsilon \downarrow 0$, if and only if the ratio $\frac{\mu(\varepsilon)}{\varepsilon}$ converges towards some $\lambda \in [0,+\infty]$. This will define a continuum of homogenized waves parametrized by $\lambda$ for which several characteristics will be easier to compute such as the speed, the temperature at the front, etc. 

As far as the speed is concerned, we will show that the map 
$$
\lambda \mapsto c^0=c^0(\lambda)
$$ 
is (smooth and) decreasing. This means that the surface tension effects slow down the propagation. Also, the minimal and maximal speeds are achieved (only) at the regimes $\lambda=+\infty$ and $0$, respectively. For both of these regimes we will derive an explicit formula for $c^0$. Surprisingly for $\lambda=0$, $c^0$ is not given by some mean of the combustion rate. We shall see for instance that, for the typical model \eqref{combustion-type}, it reads
$$
c^0(0)=\esssup_{z} A(z) \, e^{-E(z) \, \frac{\overline{b}}{\overline{g}}},
$$
where $\overline{b}$ and $\overline{g}$ denote the mean values of $b$ and $g$, respectively.
This suggests that the respective width of each layer has no influence in that regime
as compared of course to the other intrinsic parameters of the material. This original feature was already observed in \cite{ChNa97} 
for a pure geometric propagation. Here we somewhat extend this observation for the full system ``front-temperature.'' Note finally that once the speed is known, the temperature $u^0$ will be an explicit exponential entirely determined by $c^0$.

As far as the front's profile is concerned, we will need a more careful analysis to get some interesting information. It indeed turns out that the homogenized profile is always a straight line (normalized to zero). We will then analyze the microscopic oscillations of the front's profile by establishing an asymptotic expansion of the form: 
\begin{equation*}
v^\varepsilon(y)=\varepsilon \, w \left(\frac{y}{\varepsilon}\right)+\co (\varepsilon),
\end{equation*}
where $w$ will be some corrector. We will show that this corrector is entirely determined by $\lambda$ and satisfies a pure geometric equation of the form considered in \cite{ChNa97,NaRo97,ChNa02,LoCh09}. As a byproduct, our analysis thus provides some relationship between these works and the whole system ``front-temperature.'' In the limiting regime $\lambda=+\infty$, we will show that the corrector equals zero everywhere so that the front is ``almost a straight line'' at the microscopic level too. To get more information in that case, we will then establish a second-order expansion of the form:
\begin{equation*}
v^\varepsilon(y)=\varepsilon^2 \, Q \left(\frac{y}{\varepsilon}\right)+\co (\varepsilon^2),
\end{equation*}
where we will identify the profile $Q$ too.
The identification of $w$ in the other limiting regime $\lambda=0$ is more difficult and remains open. A more detailed discussion will be done in the concluding remarks, see Section  \ref{conclusion}. 

Technically, the proof of the monotonicity of $\lambda \mapsto c^0(\lambda)$ will be the most difficult. It will rely on the implicit function theorem. The existence of a travelling wave solution will rely on the Schauder's fixed point theorem as well as on careful lower bounds on the temperature at the front. For the homogenization, we will not need sophisticated tools such as the two-scale convergence, etc., thanks to general arguments showing that $v^0$ and $u^0$ can a priori not depend on $y$ anymore. The identification of $c^0$ will call for different
technical properties such as, for example, the H\"older regularity of the temperature $u^\varepsilon$. 

Let us finally give some more references on related topics. Following the inhomogeneity considered, one can as well
obtain pulsating travelling fronts, that is where the speed of the travelling
wave is no longer a constant but is periodic in time. The preceding references \cite{ChNa02,NaRo97,LoCh09} are closely related to that subject. For a rather complete study in the framework of
reaction advection diffusion equations, see \cite{BeHa02}. Finally, for a survey on travelling waves, be
it in homogeneous, periodic or heterogeneous random media, see \cite{Xin00}.

The rest of this paper is organized as follows. Section \ref{sec-prelim} is devoted to some preliminaries on the equations of the front and temperature considered separately. Section \ref{sec-existence} is devoted to the existence of a travelling wave solution to 
\eqref{free-bdry-sys}--\eqref{bdry-cond}, see Theorems \ref{existence-tw} and \ref{existence-small-Y}. The homogenization analysis starts in Section \ref{sec-homogenization}, see Theorems \ref{Theorem-homogenization} and \ref{Theorem-homogenization-lambda}. 
The qualitative analysis of the speed is done in Section \ref{monotonicity}, see Theorem \ref{Theorem-C-monotone}. The asymptotic expansions of the front's profile are given in Section \ref{sec-expansion}, see Theorems \ref{thm-expansion}, \ref{continuum-corrector} and \ref{thm-expansion-bis}.
A synthesis and some open questions are proposed in Section \ref{conclusion}. For the sake of clarity, the technical or standard proofs are postponed in appendices together with a list of the main specific notations (see Appendix \ref{app-notations}).

\section{Preliminaries: Notations and first results}\label{sec-prelim}

The existence of a travelling wave solution will be proved with the help of
the Schauder's fixed point theorem by successively freezing the temperature
and the front. In this section we focus on the frozen problems. All along this section, $Y>0$ is a given fixed period and $\mu>0$ a given fixed parameter. 

\subsection{Front's well-posedness}

We first consider \eqref{front-eqn} for a fixed temperature. This comes therefore to finding $(c,v)$ which solves
\begin{equation}
\label{frozen-front-eqn}%
\begin{cases}
-c+H(y) \sqrt{1+v_{y}^{2}}=\mu\, \frac{v_{yy}}{1+v_{y}^{2}},\\[1ex]
v(y+Y)=v(y),
\end{cases}
\end{equation}
for (almost) every $y \in \R$, where $H$ is an arbitrary given function.
Here is a result from~\cite{ChNa97}.

\begin{theorem}
\label{front-well-posed} Let $H:\mathbb{R} \to \R$ be measurable and $Y$-periodic with
$$
H_m \leq H \leq H_M \quad \mbox{almost everywhere on} \quad \R,
$$
for some positive constants $H_m$ and $H_M$. Then there exists $(c,v) \in\mathbb{R}
\times W^{2,\infty}(\mathbb{R} )$ which satisfies~\eqref{frozen-front-eqn}
almost everywhere. The speed $c$ is unique and the profile $v$ is unique up to an additive constant. Moreover, 
\begin{equation}
\label{front-esti}H_{m} \leq c \leq H_{M} \quad\mbox{and} \quad\|v_{y}%
\|_{\infty} \leq\sqrt{\frac{H_{M}^{2}}{H_{m}^{2}}-1}.
\end{equation}
\end{theorem}


\subsection{Temperature's well-posedness}

\label{Sobolev-preliminaries}

Now we suppose that the profile $v$ of the front is given and we introduce the
notations
\begin{equation*}
\Omega:=\left\{ (x,y) \in \R^2:x<v(y) \right\} , \quad\Gamma:=\left\{ (x,y) \in \R^2:x=v(y)
\right\} ,
\end{equation*}
as well as
\begin{equation*}
\Omega_{\diese} := \Omega\cap\{0<y<Y\} \quad\mbox{and}
\quad\Gamma_{\diese} := \Gamma\cap\{0<y<Y\}.
\end{equation*}
For simplicity, we do not specify the dependence on $v$.

\subsubsection*{Periodic Sobolev's spaces}

We proceed by defining the functional framework that we will need. We use the
subscript ``$_{{\diese}}$'' for spaces of functions which are
$Y$-periodic in~$y$. Hereafter we consider the Hilbert spaces
\begin{align*}
L_{\diese}^{2}(\Omega):=  & \left\{  u\in L_{\loc}%
^{1}(\overline{\Omega}):\mbox{$u$ is~$Y$-periodic in~$y$
and~$\int_{\Omega_\diese} u^2<+\infty$}\right\}  ,\\
H_{\diese}^{1}(\Omega):=  & \left\{  u\in
L_{\diese}^{2}(\Omega):\nabla u\in\left(
L_{\diese}^{2}(\Omega)\right)  ^{2}\right\},
\end{align*}
endowed with the following norms and semi-norm:
\begin{equation*}
\| u\|_{L_{\diese}^{2}(\Omega)}:=\left(  \frac{1}{Y}%
\int_{\Omega_{\diese}}u^{2}\right)  ^{\frac{1}{2}}%
,\quad|u|_{H_{\diese}^{1}(\Omega)}:=\left(  \frac{1}{Y}%
\int_{\Omega_{\diese}}|\nabla u|^{2}\right)  ^{\frac{1}{2}%
}
\end{equation*}
and $\| u\|_{H_{\diese}^{1}(\Omega)}:=\left(  \|
u\|_{L_{\diese}^{2}(\Omega)}^{2}%
+|u|_{H_{\diese}^{1}(\Omega)}^{2}\right)  ^{\frac{1}{2}}.$

\subsubsection*{Extension operator}

We next define the extension of $u$ to $\mathbb{R}^{2}$, which we will use
throughout, by 
``reflection'' as follows:
\begin{equation}
\mathrm{ext}(u)(x,y):=%
\begin{cases}
u(x,y), & \quad x<v(y),\\
u(2\,v(y)-x,y), & \quad x>v(y).
\end{cases}
\label{extension}%
\end{equation}
Recall that $\mathrm{ext}:H_{\diese}^{1}(\Omega)\rightarrow
H_{\diese}^{1}(\mathbb{R}^{2})$ is linear and bounded with
\begin{equation}
\|\mathrm{ext}(u)\|_{H_{\diese}^{1}(\mathbb{R}^{2})}\leq
C\left(  \| v_{y}\|_{\infty}\right)  \| u\|
_{H_{\diese}^{1}(\Omega)};\label{ext-esti}%
\end{equation}
see \cite{Ada75}.\footnote{Note that $C \left( \|v_y\|_\infty\right)$ does not depend on the period $Y$, which is 
immediate by direct computations.} For brevity, $\mathrm{ext}(u)$ will be simply denoted by $u
$.

\subsubsection*{Trace operator}

The front $\Gamma$ will be endowed with its superficial measure, that is for any
$\varphi\in C_{c}(\Gamma)$,
\begin{equation*}
\int_{\Gamma} \varphi:=\int_{\mathbb{R} } \varphi(v(y),y) \sqrt{1+v_{y}%
^{2}(y)} \, \mathrm{d} y.
\end{equation*}
For brevity, we will denote by $u_{|_\Gamma}$ or simply $u$ the trace of $u \in H^1_\diese(\Omega)$ on $\Gamma$, see \cite{Ada75}. Recall that $u \in H^1_\diese(\Omega) \mapsto u_{|_\Gamma} \in H^\frac{1}{2}_\diese(\Gamma)$ is well-defined linear and bounded, so that the function
$$
w:y \in \R \mapsto u(v(y),y) \in \R
$$
is well-defined (up to some negligible set) and belongs to $H^\frac{1}{2}_\diese(\R)$ with
\begin{equation*}
\|w\|_{H_{ \diese}^{\frac
{1}{2}}(\R)} \leq C\left( Y,\|v_{y}\|_{\infty} \right)  
\|u\|_{H_{\diese}^{1}(\Omega)}.
\end{equation*}

\subsubsection*{Well-posedness}

We can now state the well-posedness of the temperature.

\begin{definition}[Variational solutions] Assume \ref{hyp-param-1}--\ref{hyp-param-2} and let us consider $c \in \R$ and $v
\in W_{\diese}^{1,\infty}(\mathbb{R} )$. We
say that $u$ is a variational solution to~\eqref{T-eqn} if
\begin{equation}
\label{var-eqn}%
\begin{cases}
u \in H^{1}_{\diese}(\Omega),\\
\int_{\Omega_{\diese}} \left( c \, b \, u_{x} \, w + a \,
\nabla u \, \nabla w \right) =\int_{\Gamma_{\diese}} \frac{c \,
g}{\sqrt{1+v_{y}^{2}}} \, w, \quad\forall w \in H^{1}_{\diese}(\Omega).
\end{cases}
\end{equation}

\end{definition}

\begin{remark}
\label{var-equiv-dis} Note that $u$ is a variational solution to \eqref{T-eqn}
if and only if
\begin{equation}
\label{dis-eqn}%
\begin{cases}
u \in H^{1}_{\diese}(\Omega),\\
\int_{\Omega} \left( c \, b \, u_{x} \, \varphi+ a \, \nabla u \,
\nabla\varphi\right) =\int_{\Gamma} \frac{c \, g}{\sqrt{1+v_{y}^{2}}} \,
\varphi, \quad\forall\varphi\in C^1_{c}(\mathbb{R} ^{2}).
\end{cases}
\end{equation}
\end{remark}

The proof of the above remark is standard and postponed in Appendix
\ref{tech-app} for completeness.

\begin{theorem}
\label{T-well-posed} Assume~\ref{hyp-param-1}--\ref{hyp-param-2} and let~$c>0$ and~$v\in
W_{\diese}^{1,\infty}(\mathbb{R})$. Then there exists a unique
variational solution~$u\in H_{{\diese}}^{1}(\Omega)$
to~\eqref{T-eqn}. Moreover, this solution satisfies
\begin{equation}
|u|_{H_{\diese}^{1}(\Omega)}\leq\frac{2\,c\,g_{M}^{2}}%
{a_{m}\,b_{m}},\label{H1-esti}%
\end{equation}
and for almost every $(x,y)\in\Omega$,
\begin{equation}
\frac{g_{m}\,a_{m}}{a_{M}\,b_{M}}\,e^{c\,\frac{b_{M}}{a_{m}}\left(  x-\|
v\|_{\infty}\right)  }\leq u(x,y)\leq\frac{g_{M}\,a_{M}}{a_{m}\,b_{m}%
}\,e^{c\,\frac{b_{m}}{a_{M}}\left(  x+\| v\|_{\infty}\right)
}.\label{Linfty-esti}%
\end{equation}

\end{theorem}


The proof of the above theorem is standard and postponed in Appendix
\ref{app-T-well-posed} for completeness.

\subsection{Stability}

This section
is devoted to further properties that will be needed later. It deals with the passage to the limit in a sequence of problems of the form
\begin{equation}
\label{T-eqn-n}%
\begin{cases}
c_{n} \, b \, (u_{n})_{x}-\Div (a \, \nabla u_{n}) = 0, & \quad x<v_{n}(y),\\[1ex]
a \, \frac{\partial u_{n}}{\partial\nu} =\frac{c_{n} \, g}{\sqrt{1+(v_{n}%
)_{y}^{2}}}, & \quad x=v_{n}(y),\\[1ex]
u_{n}(x,y) \rightarrow0, & \quad\mbox{as~$x\rightarrow -\infty$,}
\end{cases}
\end{equation}
\begin{equation}
\label{front-eqn-n}-c_{n}+H_{n}(y) \sqrt{1+(v_{n})_{y}^{2}}=\mu\, \frac
{(v_{n})_{yy}}{1+(v_{n})_{y}^{2}}, \quad 
y \in\mathbb{R} .
\end{equation}

For a technical reason, we will need to fix the front's profile by assuming for instance that it has a zero mean value. This will be done without loss of generality, since the
solution of \eqref{front-eqn-n} is unique up to an additive constant. For brevity, we will denote throughout by $\overline{f}:=\frac{1}{Y} \int_0^Y f$ the mean value of any $Y$-periodic function $f:\R \to \R$.

Let us start with a compacity result.

\begin{lemma}
[Compactness]\label{compactness} Let us assume that for each $n \in \mathbb{N}$,
\begin{equation*}%
\begin{cases}
H_{n} \in L_{\diese}^{\infty}(\mathbb{R} ),\\
\mbox{$c_n>0$, $v_n \in W_\diese^{2,\infty}(\R)$,}\\
u_{n} \in H^{1}_{\diese}(\Omega_{n}),\\
\mbox{$v_n$ satisfies \eqref{front-eqn-n} almost everywhere with $\overline{v}_n=0$,}\\
\mbox{$u_n$ is a variational solution to \eqref{T-eqn-n},}
\end{cases}
\end{equation*}
where $\Omega_{n}:=\{x<v_{n}(y)\}$. Then, if
\begin{equation}
\label{sufficient-cond}0<\inf_{n} \essinf_{y} H_{n}(y) \leq\sup_{n}
\esssup_{y} H_{n}(y) <+\infty,
\end{equation}
there exists $(H,c,v,u) \in L^{\infty}_{\diese}(\mathbb{R} )
\times\mathbb{R} _{+} \times W^{2,\infty}_{\diese}(\mathbb{R} )
\times H_{\diese}^{1}(\mathbb{R} ^{2})$ such that
\begin{equation}
\label{weak-conv}%
\begin{cases}
\mbox{$H_n \rightharpoonup H$ in $L^\infty(\R)$ weak-$\star$,}\\
\mbox{$c_n \to c$, $v_n \rightharpoonup v$ in $W^{2,\infty}(\R)$
weak-$\star$,}\\
\mbox{$u_n \rightharpoonup u$ weakly in $H^1_\diese(\R^2)$},
\end{cases}
\end{equation}
up to some subsequence.
\end{lemma}

\begin{remark}
\begin{enumerate}[label={{\rm (\roman*)}}]
\item The limit in \eqref{weak-conv} has to be understood for $u_{n}$ extended
to $\mathbb{R} ^{2}$ by reflection, see \eqref{extension}.
\item The convergence in \eqref{weak-conv} implies that $v_{n} \to v$ and
$u_{n} \to u$ strongly in $W^{1,\infty}(\mathbb{R} )$ and $L^{2}%
_{\diese}(\mathbb{R} ^{2}) $, respectively.
\end{enumerate}
\end{remark}

\begin{proof}
In this proof, the letter $C$ denotes various constants independent of $n$. By
\eqref{sufficient-cond} and Theorem \ref{front-well-posed}\eqref{front-esti},
$\{c_{n}\}_{n}$ is bounded and $\|(v_{n})_{y}\|_{\infty}\leq C$. This leads to $\|v_n\|_\infty \leq C$ because $\overline{v}_n=0$. Moreover, $\|(v_{n})_{yy}\|_{\infty}\leq C$ by \eqref{front-eqn-n}. By Theorem~\ref{T-well-posed}\eqref{H1-esti} and \eqref{Linfty-esti}, we deduce that$\|u_{n}%
\|_{H^{1}_{\diese}(\Omega_{n})} \leq C$ and, after extending
$u_{n}$ to $\mathbb{R} ^{2}$, $\|u_{n}\|_{H^{1}_{\diese}(\mathbb{R} ^{2})} \leq C $ by \eqref{ext-esti}. The proof is complete by
standard weak compactness theorems.
\end{proof}


Let us now identify the limiting problem.

\begin{lemma}[Stability]\label{stability} Let $(H,c,v,u)$ be given by Lemma
\ref{compactness}. Then $(c,v,u)$ is a solution of \eqref{T-eqn} and
\eqref{frozen-front-eqn}, that is
\begin{equation*}%
\begin{cases}
\mbox{$v$ satisfies \eqref{frozen-front-eqn} almost everywhere,}\\
\mbox{and $u_{|_{\Omega}}$ is a variational solution to \eqref{T-eqn}.}
\end{cases}
\end{equation*}

\end{lemma}

\begin{proof}
Given $\varphi \in C^1_{c}(\mathbb{R} ^{2})$, we have
\begin{equation*}
\int_{\Omega_{n}} \left( c_{n} \, b \, (u_{n})_{x} \, \varphi+ a \, \nabla
u_{n} \, \nabla\varphi\right) =\int_{\mathbb{R}} c_{n} \, g(y) \,
\varphi(v_{n}(y),y) \, \mathrm{d} y,
\end{equation*}
that is
\begin{equation}
\label{dis-eqn-n}\int_{\mathbb{R} ^{2}} \left( c_{n} \, (u_{n})_{x} \, b \,
\varphi\, \mathbf{1}_{\Omega_{n}} + a \, \mathbf{1}_{\Omega_{n}} \nabla
\varphi\, \nabla u_{n} \right) =\int_{\mathbb{R}} c_{n} \, g(y) \,
\varphi(v_{n}(y),y) \, \mathrm{d} y.
\end{equation}
Let us pass to the limit in \eqref{dis-eqn-n} (along the subsequence given by Lemma \ref{compactness}). 

We claim that $\mathbf{1}_{\Omega_{n}} \to\mathbf{1}_{\Omega}$ almost everywhere on
$\mathbb{R} ^{2}$. Indeed, for all~$x \neq v(y)$, we have $x \neq v_{n}(y)$
whenever~$n$ is sufficiently large, since~$v_{n}(y) \to v(y)$. This shows that
the convergence holds for all~$(x,y) \notin\Gamma=\left\{ x=v(y)\right\} $.
This proves our claim, since the two-dimensional Lebesgue measure of this
Lipschitz graph is zero.

We deduce that $b \, \varphi\, \mathbf{1}_{\Omega_{n}} \to b \, \varphi\,
\mathbf{1}_{\Omega}$ and~$a \, \mathbf{1}_{\Omega_{n}} \nabla\varphi\to a \,
\mathbf{1}_{\Omega} \nabla\varphi$ strongly in~$L^{2}(\mathbb{R} ^{2})$, by
the dominated convergence theorem. Moreover, $\nabla u_{n} \rightharpoonup
\nabla u$ weakly in $\left(L_{\diese}^{2}(\mathbb{R} ^{2})\right) ^{2}$ since $u_{n}
\rightharpoonup u$ weakly in $H_{\diese}^{1}(\mathbb{R} ^{2})$.
It is then standard to pass to the limit in \eqref{dis-eqn-n} and deduce
that
\begin{equation*}
\int_{\Omega} \left( c \, b \, u_{x} \, \varphi+ a \, \nabla u \,
\nabla\varphi\right) =\int_{\Gamma} \frac{c \, g}{\sqrt{1+v_{y}^{2}}} \,
\varphi,
\end{equation*}
for all $\varphi\in C^1_{c}(\mathbb{R} ^{2})$. To pass to the limit in
the boundary term, we have simply used the uniform convergence of $v_{n}$
towards $v $. This proves that $u_{|_{\Omega}}$ is a variational solution to \eqref{T-eqn}.

To pass to the limit in \eqref{front-eqn-n}, we consider $\varphi\in
C_{c}(\mathbb{R} )$ and write that
\begin{equation*}
\int_{\mathbb{R}} \left( -c_{n} +H_{n} \sqrt{1+(v_{n})_{y}^{2}}\right)
\varphi=\mu\int_{\mathbb{R}} (v_{n})_{yy} \, \frac{\varphi}{1+(v_{n})_{y}^{2}%
}.
\end{equation*}
Since $(v_{n})_{y} \to v_{y}$ uniformly, the $L^{\infty}(\mathbb{R} )$
weak-$\star$ convergences of $H_n$ and $(v_{n})_{yy}$ are sufficient to
pass to the limit. We get
\begin{equation*}
\int_{\mathbb{R}} \left( -c +H \sqrt{1+v_{y}^{2}} \right)  \varphi=\mu
\int_{\mathbb{R}} \frac{v_{yy}}{1+v_{y}^{2}} \, \varphi
\end{equation*}
for all $\varphi\in C_{c}(\mathbb{R} )$, which completes the proof.
\end{proof}

Here is a last stability result for the temperature at the front.

\begin{lemma}
[Strong convergence of the traces]\label{trace-conv} Let $(H,c,v,u)$ be as in the preceding lemmas. Then 
$
u_{n}(v_{n}(y),y) \to u(v(y),y) 
$
for almost every
$y \in\mathbb{R} $ (up to the subsequence considered in the preceding lemmas or one of its subsubsequences). 
\end{lemma}


\begin{proof}
Let $w_{n}(y):=u_{n}(v_{n}(y),y)$ and recall that
\begin{equation*}
\| w_{n}\|_{H_{\diese}^{\frac{1}{2}}(\mathbb{R})}\leq
C\left(  Y,\|(v_{n})_{y}\|_{\infty}\right)  \,\| u_{n}\|
_{H_{\diese}^{1}(\Omega_{n})},
\end{equation*}
as noticed in Section \ref{Sobolev-preliminaries}. By the bounds of
the proof of Lemma \ref{compactness}, it follows that $\{w_{n}\}_{n}$ is bounded
in $H_{{\diese}}^{\frac{1}{2}}(\mathbb{R})$. Hence, the compact
embedding of $H_{\diese}^{\frac{1}{2}}(\mathbb{R})$ into
$L_{{\diese}}^{2}(\mathbb{R})$ implies that $w_{n}$ converges to
some $\tilde{w}$ strongly in $L_{{\diese}}^{2}(\mathbb{R})$
and almost everywhere, up to some subsequence (chosen as stated in the lemma). It remains to show that $\tilde{w}(y)=u(v(y),y)$ almost everywhere.  
By the Gauss-Green formula, we have for any $\varphi\in C_{c}^{1}(\mathbb{R}^{2})$,
\begin{multline*}
\int_{\mathbb{R}}w_{n}(y)\,\varphi(v_{n}(y),y)\,\mathrm{d}y=\int_{\Omega_{n}%
}(u_{n}\,\varphi)_{x}\,\mathrm{d}x\,\mathrm{d}y\\
=\int_{\mathbb{R}^{2}}\left(  \mathbf{1}_{\Omega_{n}}\,\varphi\,(u_{n}%
)_{x}+\mathbf{1}_{\Omega_{n}}\,\varphi_{x}\,u_{n}\right)  \mathrm{d}%
x\,\mathrm{d}y.
\end{multline*}
Arguing as in the proof of Lemma \ref{stability}, $\mathbf{1}_{\Omega_{n}%
}\rightarrow\mathbf{1}_{\Omega}$ almost everywhere and passing to the limit in the above equation implies that 
\begin{multline*}
\int_{\mathbb{R}} \tilde{w}(y)\,\varphi(v(y),y)\,\mathrm{d}y=\int_{\mathbb{R}^2} \left(  \mathbf{1}_{\Omega}\,\varphi\,u_{x}+\mathbf{1}_{\Omega}\,\varphi_{x}\,u\right)\, \dif x \, \mathrm{d}y\\
=\int_{\Omega} (u \, \varphi)_x \, \dif x \, \dif y=\int
_{\mathbb{R}}u(v(y),y)\,\varphi(v(y),y)\,\mathrm{d}y.
\end{multline*}
We complete the proof by taking test functions of the form
$\varphi(x,y)=\theta(x)\,\psi(y)$, with $\theta(x)=1$ when $|x| \leq \| v\|_{\infty}$.
\end{proof}

\section{Existence of a travelling wave solution}\label{sec-existence}

Let us now look for a solution to \eqref{T-eqn}--\eqref{front-eqn}. We continue to use the notations of the preceding section. In particular, we still do not specify the dependence in $v$ of the sets $\Omega=\{x<v(y)\}$ and $\Gamma
=\{x=v(y)\}$ to simplify. Moreover $Y>0$ is a given fixed period and $\mu>0$ a given fixed coefficient all along this section too.  

\begin{definition}[Travelling wave solution] Assume
\ref{hyp-param-1}--\ref{hyp-arrh-4}. A triplet $(c,v,u)$ is said to be a
travelling wave solution to \eqref{free-bdry-sys}--\eqref{bdry-cond} if

\begin{enumerate}[label={{\rm (\roman*)}}]
\item $c\in\mathbb{R}$, $v\in W_{\diese}^{2,\infty}(\mathbb{R})$,

\item $u\in H_{{\diese}}^{1}(\Omega)$, $u\geq0$,

\item $u$ is a variational solution to \eqref{T-eqn} and $v$ satisfies
\eqref{front-eqn} almost everyhwere.
\end{enumerate}
\end{definition}

We first deal with the case where in addition to \ref{hyp-arrh-1}--\ref{hyp-arrh-4} the
parameter $R$ is also bounded from below by some positive constant $R_m$, that is to say:
\begin{equation}
R(y,T) \geq R_{m}>0\quad \mbox{for almost each $y \in \R$ and all $T>0$.}\label{hyp-bounded-below}%
\end{equation}
We will next discuss the more general case where $R$ can go to zero.

\subsection{The case of nondegenerate $R$}

\begin{theorem}
\label{K-bounded-below} Assume~\ref{hyp-param-1}--\ref{hyp-arrh-4} and
\eqref{hyp-bounded-below}. Then there exists a travelling wave solution
$(c,v,u)$ to \eqref{free-bdry-sys}--\eqref{bdry-cond} with a positive speed $c
$.
\end{theorem}

\begin{proof}
The idea is to look for a fixed point of some appropriate function
$\Phi:\mathcal{C} \to\mathcal{C}$. 

Let us first choose the set
\begin{equation*}
\mathcal{C}:= \left\{ H \in L^{\infty}_{\diese}(\mathbb{R}
):R_{m} \leq H \leq R_{M} \right\} .
\end{equation*}
We will use the Schauder-Tikhonov's fixed point theorem thus needing this set to be convex and compact. It is clearly convex and to get the compacity we simply consider the $L^{\infty}(\mathbb{R} )$ weak-$\star$
topology on $\mathcal{C}$. 

Let us now choose $\Phi$. Given~$H \in\mathcal{C}$, we can apply successively
Theorems~\ref{front-well-posed} and~\ref{T-well-posed}. We find that there
exists a triplet~$(c,v,u)$ solution of~\eqref{T-eqn}
and~\eqref{frozen-front-eqn}. This triplet, which of course depends on $H$, is
unique under the additional condition that
$
\overline{v}=0.
$
We can then define the function
\begin{equation*}
\begin{array}[c]{rccl}
\Phi: & \mathcal{C} & \to & \mathcal{C}\\
& H & \mapsto & \Phi(H):y \mapsto R(y,u(v(y),y)).
\end{array}
\end{equation*}
Now it only remains to
show that $\Phi$ is continuous. Note that, rigourously, we should also verify that $\Phi$ is well-defined. This means that two arbitrary almost everywhere representatives of $H$ should give us two almost everywhere equal measurable functions $\Phi(H)$. This is in fact quite standard because $R$ is a Carath\'eodory function by \ref{hyp-arrh-1}--\ref{hyp-arrh-2}. The detailed verification of the well-definition of $\Phi$ is thus left to the reader.  

Let us then continue by proving that $\Phi$ is continuous for the $L^\infty(\R)$ weak-$\star$ topology. We can argue
with sequences because this topology is
metrizable on bounded sets (such as $\mathcal{C}$). Let thus $\mathcal{C} \ni H_{n} \rightharpoonup H$ in~$L^{\infty}(\mathbb{R} )$
weak-$\star$. By the construction above, 
\begin{equation*}
0<R_{m} \leq\inf_{n} \essinf H_{n}(y) \leq\sup_{n} \esssup
H_{n}(y) \leq R_{M} <+\infty.
\end{equation*}
We can then apply Lemma \ref{compactness}. This compacity result implies that $(H_{n},c_{n},v_{n},u_{n})$ converges to some $(\tilde{H},c,v,u)$ in the sense of \eqref{weak-conv} (and up to some subsequence). Note that 
$(c_{n},v_{n},u_{n})$ is the unique triplet associated to $H_{n}$ as
above.
Note also that $\tilde{H}=H$ by uniqueness of the limit of $H_n$. Moreover, the triplet $(c,v,u)$
satisfies \eqref{T-eqn} and \eqref{frozen-front-eqn} with our given $H$, thanks to the stability result of Lemma \ref{stability}. Applying then successively Lemma \ref{trace-conv}, \ref{hyp-arrh-1}, and the relative compactness of $\Phi(\mathcal{C}) \subseteq \mathcal{C}$, we first deduce that $\Phi(H_{n})
\to\Phi(H)$ almost everywhere and next in $L^\infty(\R)$ weak-$\star$ (up to  another subsequence if necessary). We have thus proved that $\Phi(H_{n}) \rightharpoonup\Phi(H)$ in $L^{\infty
}(\mathbb{R} )$ weak-$\star$ up to some subsequence. To conclude the
convergence of the whole sequence, we apply this reasoning by starting from any arbitrary subsequence of
$H_{n}$. We deduce that $\Phi$ is continuous, since the limit of the obtained converging subsubsequence is always the
same, that is $\Phi(H)$.

Finally the Schauder-Tikhonov's theorem gives us a fixed point $\Phi(H)=H$,
whose associated triplet $(c,v,u)$ is a travelling wave solution. Since $R$ is
assumed bounded from below by $R_{m}>0$, the positivity of $c$ is ensured by
Theorem \ref{front-well-posed}.
\end{proof}

\subsection{More general $R$}

Now we want to deal with the case where $R$ may go to
zero at $T=0$. For technical reasons, we will restrict to parameters with the
following prescribed behavior at zero:
\begin{equation}
\label{hyp-lim}\lim_{T \downarrow0}  \left\{|\ln T| \, \essinf_{y \in \R} R(y,T)\right\}=+\infty.
\end{equation}
This is for instance the case if $\essinf_{y} R(y,T) \sim \frac{C}{|\ln T|^\alpha}$ as $T \downarrow 0$ for some $\alpha \in (0,1)$ and some $C>0$. 

We start by giving a result which establishes an a priori positive lower bound
for the temperature at the front.

\begin{lemma}
\label{lower-bound} Assume \ref{hyp-param-1}--\ref{hyp-arrh-4} and let
$(c,v,u)$ be a travelling wave solution to
\eqref{free-bdry-sys}--\eqref{bdry-cond} with a positive speed $c$. If in addition \eqref{hyp-lim} holds true, then
\begin{equation*}
u \geq \min\{T > 0: \ln (T)  \essinf_y R(y,T) \geq-C\, (1+Y)\}>0
\end{equation*}
almost everywhere on $\Gamma$, for some constant $C=C(a_{m},g_{m},a_{M},b_{M},R_{M}) \geq0$. 
\end{lemma}

\begin{proof}
It is sufficient to consider the case where $\overline{v}=0$ (otherwise one can
always consider another travelling wave solution of the same problem with the
triplet $(c,\tilde{v},\tilde{u})$, where $\tilde{v}(y):=v(y)-\overline{v}$ and
$\tilde{u}(x,y):=u(x+\overline{v},y))$. 

Set $u_{m}:=\essinf_{\Gamma}u$ . By \ref{hyp-arrh-1}, we
have
\begin{equation*}
\essinf_{y} R(y,u(v(y),y)) \geq \essinf_y R(y,u_{m}).
\end{equation*}
Note that, at this stage, we could have $u_{m}=0.$ But, as $c>0$, Theorem
\ref{T-well-posed} implies that
\begin{equation*}
u_{m}\geq\frac{g_{m}\,a_{m}}{a_{M}\,b_{M}}\,e^{-2\,c\,\frac{b_{M}}{a_{m}%
}\,\| v\|_{\infty}}.
\end{equation*}
The boundedness of $v$ then ensures that $u_{m}>0$ and a fortiori so is
$\essinf_y R(y,u_{m})$ by \ref{hyp-arrh-4}. Now we can apply Theorem \ref{front-well-posed} to get
\begin{equation*}
c\leq R_{M}\quad\text{and}\quad\| v_{y}\|_{\infty}\leq\sqrt
{\frac{R_{M}^{2}}{\essinf_y R(y,u_{m})^2}-1}.
\end{equation*}
%
Since $\overline{v}=0$, $\| v\|_{\infty}\leq Y\,\| v_{y}\|_{\infty}$ and
\begin{equation*}
u_{m}\geq\frac{g_{m}\,a_{m}}{a_{M}\,b_{M}}\,e^{-2\,c\,\frac{b_{M}}{a_{m}%
}\,Y\sqrt{\frac{R_{M}^{2}}{\essinf_y R(y,u_{m})^2}-1}}\geq
\frac{g_{m}\,a_{m}}{a_{M}\,b_{M}}\,e^{-2\,R_M \,\frac{b_{M}}{a_{m}}\,Y\,\frac{R_{M}}{\essinf_y R(y,u_{m})}}.
\end{equation*}
Taking the logarithm,
$
\ln u_{m}\geq-C-\frac{C \, Y}{\essinf_y R(y,u_{m})},
$
for some constant $C$ having the dependence stated in the lemma. Multiplying by $\essinf_y R(y,u_{m})$ and using the fact that $0<\essinf_y R(y,u_{m})\leq
R_{M}$, we deduce that
\begin{equation*}
 \ln (u_{m}) \essinf_y R(y,u_{m})\geq-C\,R_{M}-C \, Y. \qedhere
\end{equation*}
\end{proof}

We are now ready to give the analogous of Theorem \ref{K-bounded-below} under
the more general assumption \eqref{hyp-lim}.

\begin{theorem}
\label{existence-tw} Assume~\ref{hyp-param-1}--\ref{hyp-arrh-4} and
\eqref{hyp-lim}. There then exists a travelling wave solution $(c,v,u)$ to
\eqref{free-bdry-sys}--\eqref{bdry-cond} with a positive speed $c$.
\end{theorem}

\begin{proof}
Let us consider $R_{n}:=\max\left\{  R,\frac{1}{n}\right\}  $. We can apply
Theorem \ref{K-bounded-below} to get the existence of some nontrivial
travelling wave solution $(c_{n},v_{n},u_{n})$ with this truncated parameter. By Lemma \ref{lower-bound} and \eqref{hyp-lim}, $\essinf_{\Gamma_{n}}u_{n}\geq \gamma_n$ for
$$
\gamma_n := \min\{T > 0:\ln(T) \essinf_y R_n(y,T) \geq-C\, (1+Y)\}>0,
$$
where $C$ is independent of $n$. But as $R_{n}\geq R,$ we have 
$$
\gamma_n \geq \gamma:=\min\{T > 0:\ln(T) \essinf_y R(y,T) \geq-C\, (1+Y)\}>0,
$$
for all $n$. In particular,
\begin{equation*}
\inf_n \essinf_{y}R(y,u_{n}(v_{n}(y),y))\geq \essinf_{y}R(y,\gamma)>0
\end{equation*}
thanks to \ref{hyp-arrh-1} and \ref{hyp-arrh-4}. Defining $H_{n}(y):=R(y,u_{n}
(v_{n}(y),y))$, we can thus apply the compactness result of Lemma
\ref{compactness}---by assuming that $\overline{v}_{n}=0$ without loss of generality. Arguing similarly as before, we get the existence of some $(H,c,v,u)$ limit (of some
subsequence) of $(H_{n},c_{n},v_{n},u_{n})$ where $(c,v,u)$ is a solution of \eqref{T-eqn} and \eqref{frozen-front-eqn} for the
above $H$, thanks to Lemma \ref{stability}. It thus only remains to identify $H(y)$ with $R(y,u(v(y),y))$.

Applying Lemma \ref{trace-conv}, \ref{hyp-arrh-1}, and recalling that $R_{n}=\max\left\{
R,\frac{1}{n}\right\} $, we infer that 
\begin{equation*}
H_{n}(y) =R_n(y,u_{n}(v_{n}(y),y)) \geq R(y,u_{n}(v_{n}(y),y))
\to R(y,u(v(y),y))
\end{equation*}
for almost every $y \in \R$ (up to some subsubsequence). On the other hand, we also have for any fixed $n_{0}$ and $n \geq n_{0}%
$,
\begin{equation*}
H_{n}(y) =R_n(y,u_{n}(v_{n}(y),y)) \leq R_{n_{0}}(y,u_{n}%
(v_{n}(y),y))
\to R_{n_{0}}(y,u(v(y),y)).
\end{equation*}
Letting $n_{0} \to+\infty$, we deduce that $H_{n}(y) \to R(y,u(v(y),y))$
for almost every $y \in \R$. This is sufficient to identify this limit with the weak-$\star$ limit $H$ of
$H_{n}$.
\end{proof}

\begin{remark}
The proof suggests that the assumption \eqref{hyp-lim} can be slightly relaxed. The key result was indeed the lower bound of Lemma \ref{lower-bound}. Assumption  \eqref{hyp-lim} has only been used to imply that
$$
\min\{T > 0: \ln (T)  \essinf_y R(y,T) \geq-C\, (1+Y)\}>0,
$$
where $C=C(a_{m},g_{m},a_{M},b_{M},R_{M})$. It thus suffices to directly assume that this minimum is positive. This would be for instance the case if 
$$
\essinf_y R(y,T) \sim \frac{\tilde{C}}{|\ln T|} \quad \mbox{as} \quad T \downarrow 0,
$$ 
for some $\tilde{C}>C\, (1+Y)$. 
\end{remark}

\subsection{The case of small periods}

We finally consider the case of small $Y$ or more precisely of large values of the
ratio $\frac{\mu}{Y}$. For that case, we will 
show the existence of a
nontrivial travelling wave solution without the preceding assumption
\eqref{hyp-lim}. It includes in particular the typical model \eqref{combustion-type}. Let us start with an estimate on $v_{y}$. 

\begin{lemma}
\label{bound-front} Let $H\in L^{\infty}(\mathbb{R})$ be $Y$-periodic,
nonnegative, and assume that the pair $(c,v)\in\mathbb{R}^{+}\times
W_{\diese}^{2,\infty}(\mathbb{R})$ satisfies Equation
\eqref{frozen-front-eqn} almost everywhere. Then we have the estimate: $\|\arctan(v_{y})\|_{\infty}\leq
\frac{2 \, c\,Y}{\mu}.$
\end{lemma}

\begin{proof}
Let us define $f(y):=H(y)\sqrt{1+v_{y}^{2}(y)}$ and $F(y):=\mu \arctan
(v_{y}(y))$. These functions are $Y$-periodic and satisfy
\begin{equation*}
F'(y)=f(y)-c,
\end{equation*}
thanks to Equation \eqref{frozen-front-eqn}. Integrating over one
period, we first deduce that $c=\overline{f}$. Integrating then
over one arbitrary interval $(y_{\ast},y)$, such that $F(y_{\ast})=0$, we
deduce that
\begin{equation*}
\| F\|_{\infty}\leq\int_{0}^{Y}|F'|\leq2 \, c\,Y
\end{equation*}
(since $f \geq 0$ and $\int_0^Yf=c \, Y$ by what precedes).
The proof is complete by the definition of $F$. 
\end{proof}

Let us now give a new lower bound for the temperature at the front.

\begin{lemma}
\label{new-lower-bound}
Assume \ref{hyp-param-1}--\ref{hyp-arrh-4} and let
$(c,v,u)$ be a travelling wave solution to
\eqref{free-bdry-sys}--\eqref{bdry-cond} with a positive speed $c$. If in addition $\frac{\mu}{Y}>\frac{4 \, R_M}{\pi}$, then
\begin{equation*}
u\geq \frac{g_{m}\,a_{m}}{a_{M}\,b_{M}}\,e^{-\frac{2\,R_{M}\,b_{M} \,
Y}{a_{M}} \tan\left(  \frac{2 \, R_{M}\,Y}{\mu}\right)}>0 \quad
\mbox{almost everywhere on}\quad\Gamma.
\end{equation*}
\end{lemma}

\begin{proof}
By the estimate \eqref{Linfty-esti} of Theorem \ref{T-well-posed} and the
upper bound $c \leq R_{M}$ given by Theorem \ref{front-well-posed}, 
\begin{equation*}
\essinf_{\Gamma} u\geq\frac{g_{m}\,a_{m}}{a_{M}\,b_{M}%
}\,e^{-\frac{2\,R_{M}\,b_M \, Y}{a_{m}} \|v_{y}%
\|_{\infty}}.
\end{equation*}
The proof is complete by applying the previous
lemma.
\end{proof}

Here is finally our existence result for large ratio
$\frac{\mu}{Y}$.

\begin{theorem}
\label{existence-small-Y} Let us assume that
\ref{hyp-param-1}--\ref{hyp-arrh-4} hold together with the following
condition:
\begin{equation}
\frac{\mu}{Y} > \frac{4 \, R_{M}}{\pi}.\label{hyp-small-Y}%
\end{equation}
Then there exists a solution $(c,v,u)$ to \eqref{T-eqn} and \eqref{front-eqn}
with a positive speed $c$.
\end{theorem}

The proof is exactly the same as for Theorem \ref{existence-tw}, but this time we use Lemma \ref{new-lower-bound} instead of Lemma \ref{lower-bound} to bound the temperature at the front from below. 

\subsection{H\"older regularity of the temperature}

At this stage, we only know that $u$ is $H^{\frac{1}{2}}$ at the front as the
trace of a $H^{1}$ function. During the homogenization, we will require a stronger regularity notably that $u$ be H\"{o}lder
continuous. Since the parameters will be brought to rapidly oscillate with period $Y = \varepsilon$, it is important to clarify their
influence (if any) on this regularity. The main difficulty will be to control the influence of $\frac{\mu}{Y}$ in \eqref{hyp-small-Y} or of $R$ in \eqref{hyp-lim}. This will be done by assuming the existence of a fixed ratio $\lambda_0>0$ and a fixed degenerate combustion rate $R_0:\R^+ \to \R^+$ such that 
\begin{equation}\label{hyp-both}
\begin{split}
\bullet & \quad \mbox{either}  \quad \Big[\frac{\mu}{Y} \geq \lambda_0 > \frac{4 \, R_M}{\pi}\Big]\\
\bullet & \quad \mbox{or} \quad \Big[\essinf_y R(y,\cdot) \geq R_0(\cdot) \quad \mbox{and} \quad \lim_{T \downarrow 0} R_0(T) \, |\ln T|=+\infty \Big].
\end{split}
\end{equation} 

Here is the precise result.

\begin{theorem}\label{Holder-regularity}
Assume that $(c,v,u)$ is a travelling solution to \eqref{free-bdry-sys}--\eqref{bdry-cond} with a positive speed $c$ and such that \ref{hyp-param-1}--\ref{hyp-arrh-4} hold. Assume in addition that there are $Y_0>0$, $\lambda_0>0$ and $R_0:\R^+ \to \R^+$ such that $0<Y\leq Y_0$ and \eqref{hyp-both} holds. Then (the extension of) the temperature $u$ is H\trema{o}lder continuous with
\begin{equation*}
|u(x,y)-u(\tilde{x},\tilde{y})| \leq C \left( |x-\tilde{x}|^\alpha+|y-\tilde{y}|^\alpha \right) \quad \forall (x,y),(\tilde{x},\tilde{y})\in \R^2,
\end{equation*}
for some positive constants $C$ and $\alpha$ depending only on $Y_0$, $\lambda_0$, $R_0$ and the bounds $a_{m}$, $b_{m}$, $g_{m}$, $a_{M}$, $b_{M}$, $g_{M}$ and $R_{M}$.
\end{theorem}

\begin{remark}
\begin{enumerate}[label={{\rm (\roman*)}}]
\item The assumption \eqref{hyp-both} ensures the existence of $(c,v,u)$ by Theorems \ref{existence-tw} and \ref{existence-small-Y}. 
\item The constants $C$ and $\alpha$ do not depend on $\mu$ (provided that \eqref{hyp-both} holds). 
\end{enumerate}
\end{remark}

To prove this result we need to collect all the previous a priori estimates and apply standard arguments from the regularity theory of elliptic PDEs, see for instance \cite{GiTr01,Nit11}. The details are postponed in Appendix \ref{app-Holder}. 

\section{Homogenization}\label{sec-homogenization}

We will now be interested in the $\varepsilon$-dependent free
boundary problem
\begin{equation}%
\begin{cases}
c^{\varepsilon}\,b\left(\frac{y}{\varepsilon}\right) u_{x}^{\varepsilon}-\Div \left(
a\left(\frac{y}{\varepsilon}\right) \nabla u^{\varepsilon}\right)  =0, & \quad x<v^{\varepsilon
}(y),\\[1ex]
a\left(\frac{y}{\varepsilon}\right)\frac{\partial u^{\varepsilon}}{\partial\nu}=\frac
{c^{\varepsilon}\,g\left(\frac{y}{\varepsilon}\right)}{\sqrt{1+(v_{y}^{\varepsilon})^{2}}}, &
\quad x=v^{\varepsilon}(y),\\[1ex]
u^{\varepsilon}(x,y)\rightarrow0, & \quad\mbox{as~$x\rightarrow -\infty$},
\end{cases}
\label{eps-T-eqn-bis}%
\end{equation}
and%
\begin{equation}
-c^{\varepsilon}+R\left(\frac{y}{\varepsilon},u^{\varepsilon}\right)\sqrt{1+(v_{y}%
^{\varepsilon})^{2}}=\mu \, \frac{v_{yy}^{\varepsilon}}{1+(v_{y}^{\varepsilon
})^{2}},\quad x=v^{\varepsilon}(y),
\label{eps-front-eqn-bis}%
\end{equation}
where $\mu>0$ is a fixed curvature coefficient and $a,b,g,R$, are fixed $1$-periodic parameters assumed to satisfy \ref{hyp-param-1}--\ref{hyp-arrh-4} (thus with $Y=1$). The new parameter~ $\varepsilon$ is the period of the medium. Note that
the normal $\nu$ depends on~$\varepsilon$ too (which is not precised in \eqref{eps-T-eqn-bis}--\eqref{eps-front-eqn-bis} for simplicity). The purpose of this section is to find the limit of $(c^{\varepsilon
},v^{\varepsilon},u^{\varepsilon})$ as $\varepsilon\downarrow0$. 


To avoid confusion with the preceding notations, the new fresh region, front, etc., will be denoted differently. More precisely, we will denote by
\begin{equation*}
\Omega^{\varepsilon}:=\left\{  (x,y):x<v^{\varepsilon}(y)\right\}
\quad\mbox{and}\quad\Gamma^{\varepsilon}:=\left\{  (x,y):x=v^{\varepsilon
}(y)\right\}  ,
\end{equation*}
respectively the fresh region and the flame front, and by
\begin{equation*}
\Omega_{\per}^{\varepsilon}:=\Omega^{\varepsilon}\cap\{0<y<\varepsilon
\}\quad\mbox{and}\quad\Gamma_{\per}^{\varepsilon}:=\Gamma
^{\varepsilon}\cap\{0<y<\varepsilon\},
\end{equation*}
the corresponding restrictions to one period, just as in Section \ref{Sobolev-preliminaries}.
Likewise the $\varepsilon$-periodic (in $y$) Sobolev's spaces will be denoted
by $L_{\per}^{2}(\Omega^{\varepsilon})$ and $H_{\per}%
^{1}(\Omega^{\varepsilon})$. 


Our main convergence results are stated in the subsection below and their proofs are postponed in the next subsection. 

\subsection{Main convergence results}

We start by recalling the
definition of travelling wave solutions in this new setting.

\begin{definition}
\label{defn-eps-existence} 
Let $\varepsilon,\mu>0$ and assume \ref{hyp-param-1}--\ref{hyp-arrh-4} with $Y=1$. Then the triplet $(c^{\varepsilon},v^{\varepsilon},u^{\varepsilon})$ is a
solution to $(\ref{eps-T-eqn-bis})$--$(\ref{eps-front-eqn-bis})$ if
\begin{enumerate}[label={{\rm (\roman*)}}]
\item $c^{\varepsilon} \in \R$, $v^{\varepsilon}\in W_{\per}^{1,\infty}(\mathbb{R}),$
\item $u^{\varepsilon}\in H_{\per}^{1}(\Omega^{\varepsilon})$,
$u^{\varepsilon}\geq0$,
\item $v^{\varepsilon}$ satisfies \eqref{eps-front-eqn-bis} almost everywhere and
\begin{equation*}
\int_{\Omega_{\per}^{\varepsilon}}\left(  c^{\varepsilon
}\,b^{\varepsilon}\,u_{x}^{\varepsilon}\,w+a^{\varepsilon}\,\nabla
u^{\varepsilon}\,\nabla w\right)  =\int_{\Gamma_{\per}^{\varepsilon}%
}\frac{c^{\varepsilon}\,g^{\varepsilon}}{\sqrt{1+v_{y}^{\varepsilon}}%
}\,w\quad\forall w\in H_{\per}^{1}(\Omega^{\varepsilon})
\end{equation*}
(where $f^\varepsilon(y)=f(y/\varepsilon)$ for $f=a,b,g$).
\end{enumerate}
\end{definition}

Here is our first result. 

\begin{theorem}
\label{Theorem-homogenization} 
Let $\mu>0$ and assume \ref{hyp-param-1}--\ref{hyp-arrh-4} with $Y=1$. Let us then consider a family of solutions to
\eqref{eps-T-eqn-bis}--\eqref{eps-front-eqn-bis} of the form $\{(c^{\varepsilon},v^{\varepsilon
},u^{\varepsilon})\}_{\varepsilon \in (0,\varepsilon_0]}$ and such that
\begin{equation}
\overline{v}^\varepsilon=0\quad\forall\varepsilon \in (0,\varepsilon_0],
\label{front-fixed}
\end{equation}
for some $\varepsilon_0>0$.
Then: 
\begin{equation*}
\begin{cases}
\lim_{\varepsilon\downarrow0}c^{\varepsilon}=c^{0} & \quad\text{in}%
\quad\mathbb{R},\\
\lim_{\varepsilon\downarrow0}v^{\varepsilon}=v^{0} & \quad\text{uniformly on}\quad
\mathbb{R},\\
\lim_{\varepsilon\downarrow0}u^{\varepsilon}\,\mathbf{1}_{\Omega^{\varepsilon
}}=u^{0}\,\mathbf{1}_{x<0} & \quad\text{in}\quad L_{\loc}%
^{p}(\mathbb{R}_{y};L^{p}(\mathbb{R}_{x})), \quad \forall p \in [1,+\infty),
\end{cases}
\end{equation*}
where 
$$
c^{0}=\int_0^1 R \left(z,\frac{\overline{g}}{\overline{b}}\right) \dif z
$$ 
and $(v^{0},u^{0})$ are given by:
\begin{equation}
v^{0}=0\quad\mbox{and}\quad u^{0}(x)=\frac{\overline{g}}{\overline{b}%
} \,\exp \left(\frac{c^{0}\,\overline{b}}{\overline{a}}
\,x\right) \quad \mbox{for $x<0$}.\label{homogenized-triplet}
\end{equation}
Moreover, if we consider the extensions to $\R^2$, then $u^\varepsilon$ converges to $u^0$ for $p=+\infty$ too. More precisely, we have:
\begin{equation*}
\lim_{\varepsilon \downarrow 0} \ext (u^\varepsilon)=\ext(u^0) \quad \mbox{uniformly on} \quad \R^2,
\end{equation*}
where 
\begin{equation}\label{eps-extension}
\ext(u^\varepsilon)(x,y):=
\begin{cases}
u^\varepsilon(x,y), & \quad x<v^\varepsilon(y),\\
u^\varepsilon(2\,v^\varepsilon(y)-x,y), & \quad x>v^\varepsilon(y),
\end{cases}
\end{equation}
and $\ext(u^0)(x):=\frac{\overline{g}}{\overline{b}%
} \,\exp \left(-\frac{c^{0}\,\overline{b}}{\overline{a}}
\, |x|\right)$.
\end{theorem}

\begin{remark}
\begin{enumerate}[label={{\rm (\roman*)}}]
\item Let us recall that $\overline{v}^\varepsilon=\frac{1}{\varepsilon} \int_0^\varepsilon v^\varepsilon$, $\overline{g}=\int_0^1 g$, etc.  
\item Such a family $\left\{(c^\varepsilon,v^\varepsilon,u^\varepsilon)\right\}_{\varepsilon \in (0,\varepsilon_0]}$ always exists provided that $\varepsilon_0$ is small enough, thanks to Theorem \ref{existence-small-Y}. 
\item The extensions above are defined just as in \eqref{extension}, but with respect to the respective fresh regions $\{x<v^\varepsilon(y)\}$ and $\{x<0\}$. As before, we will simply use the letters $u^\varepsilon$ and $u^0$ to denote these extensions. This means that throughout $u^0(x)=\frac{\overline{g}}{\overline{b}%
} \,\exp \left(-\frac{c^{0}\,\overline{b}}{\overline{a}}
\, |x|\right)$ for all $x \in \R$.
\item In the above the convergence holds for the ``whole family'' of triplets and not only for some particular subsequence.
\item The convergence of the temperatures is actually strong in $H^1$ too. The proof is rather standard once having the result above. A short proof is provided in Appendix \ref{app-strong-conv} for the reader's interest (this appendix is independent of the rest). 
\end{enumerate}
\end{remark}

\begin{remark}[The homogenized system] After homogenization, the front's profile, which becomes planar,
reduces to
\begin{equation*}
\Gamma^{0}=\{(x,y):x=v^{0}=0\}
\end{equation*}
and the fresh region to the half plane
\begin{equation*}
\Omega^{0}=\{(x,y):x<0\}
\end{equation*}
(here \eqref{front-fixed} allows to fix the front's profile and get $\Gamma^0$ at the limit).
The temperature $u^{0}$ of the fresh region, which becomes independent of $y,
$ is given by the solution of the one dimensional problem
\begin{equation}%
\begin{cases}
c^{0}\,\overline{b}\,u_{x}^{0}-\overline{a}\,u_{xx}^{0}=0, & \quad x<0,\\[1ex]
\overline{a}\,u_{x}^{0}=c^{0}\,\overline{g}, & \quad x=0,\\[1ex]
u^{0}(x)\rightarrow0, & \quad\mbox{as~$x\rightarrow -\infty$},
\end{cases}
\label{homogenized-T-eqn}%
\end{equation}
which is the homogenized version of \eqref{eps-T-eqn}. Finally the speed
$c^{0}$ is given by the equation
\begin{equation*}
-c^{0}+\int_0^1 R \left(z,u^{0}(0)\right) \dif z=0,
\end{equation*}
which is the homogenized version of \eqref{eps-front-eqn}.
\end{remark}

Let us now analyze the effects of the curvature on the propagation. For that we allow $\mu=\mu(\varepsilon)$ to depend on $\varepsilon$. Next we assume that the limit 
\begin{equation}
\lambda=\lim_{\varepsilon\rightarrow0}\frac{\mu(\varepsilon)}{\varepsilon
}\label{eqn-lambda}%
\end{equation}
exists and we propose to run the analysis
following the different values of $\lambda.$ We will technically need to assume in addition that
\begin{equation}\label{H}
\mbox{either} \quad \lambda>\frac{4 \, R_{M}}{\pi} \quad \mbox{or} \quad \lim_{T\downarrow0} \left\{|\ln T| \, \essinf_z R(z,T)\right\}=+\infty.
\end{equation}
This means that for a small curvature regime $\lambda$, the combustion rate $R$ will be allowed to degenerate but not too much (that is we will use the second assumption). We will also need to assume that:
\begin{equation}\label{hyp-param-5}
\mbox{The function $T \in \R^+ \mapsto \esssup_{z} R(z,T)$ is continuous at $T=\frac{\overline{g}}{\overline{b}}$}. 
\end{equation}
Note that this function is continuous for all $T$ if considering the typical combustion rate in \eqref{combustion-type} (with $A$ and $E$ bounded). During the proof, the continuity at $T=\frac{\overline{g}}{\overline{b}}$ only will be sufficient. This particular value will correspond to the constant value of the homogenized temperature at the front. 

Here is our second and last convergence result.

\begin{theorem}
\label{Theorem-homogenization-lambda} 
For each $\varepsilon>0$, let $\mu(\varepsilon)>0$ be such that the limit $\lambda$ in \eqref{eqn-lambda} exists in $[0,+\infty]$. Assume next that \ref{hyp-param-1}--\ref{hyp-arrh-4} and \eqref{H}--\eqref{hyp-param-5} hold with $Y=1$.  
Let us then consider a family of solutions to \eqref{eps-T-eqn-bis}--\eqref{eps-front-eqn-bis} with $\mu=\mu(\varepsilon)$ and of the form $\{(c^{\varepsilon
},v^{\varepsilon},u^{\varepsilon})\}_{\varepsilon \in (0,\varepsilon_0]}$ with
$$
\overline{v}^{\varepsilon
}=0 \quad \forall \varepsilon \in (0,\varepsilon_0],
$$
for some $\varepsilon_0>0$.
Then:
\begin{enumerate}[label={{\rm (\roman*)}}]
\item There exists $c^0=c^0(\lambda)>0$ such that 
$$
\lim_{\varepsilon \downarrow 0} (c^\varepsilon,v^\varepsilon,u^\varepsilon)=(c^0,v^0,u^0),
$$
where $(v^0,u^0)$ is defined as in Theorem \ref{Theorem-homogenization} and the limits are taken in the same sense.
\item This speed $c^0$ is uniquely determined as follows: 
\begin{enumerate}[label={{\rm (\roman*)}}]
\item If $\lambda=+\infty$ then $c^0=\int_0^1 R\left(z,\frac{\overline{g}}{\overline{b}}\right) \dif z$.
\item\label{corrector} If $\lambda \in (0,+\infty)$ then $c^0$ is the unique real such that the equation
\begin{equation*}
-c^{0}+R \left(z,\frac{\overline{g}}{\overline{b}}\right)\sqrt{1+w_z^{2}}=\lambda
\, \frac{w_{zz}}{1+w_{z}^{2}}
\end{equation*}
admits an almost everywhere $1$-periodic solution $w \in W^{2,\infty}(\R_z)$. 
\item If $\lambda=0$ then $c^0=\esssup_z R\left(z,\frac{\overline{g}}{\overline{b}}\right)$.
\end{enumerate}
\end{enumerate}
\end{theorem}

\begin{remark} 
\begin{enumerate}[label={{\rm (\roman*)}}]
\item Such a family $\left\{(c^\varepsilon,v^\varepsilon,u^\varepsilon)\right\}_{\varepsilon \in (0,\varepsilon_0]}$ always exists provided that~$\varepsilon_0$ is small enough, thanks to \eqref{H} and Theorems \ref{existence-tw} or \ref{existence-small-Y}.
\item Here also the whole family of triplets converges and this is in fact equivalent to the convergence of the whole family of ratios, see Remark \ref{rem-adherence-value}.  
\item Theorem \ref{Theorem-homogenization} is a particular case of Theorem \ref{Theorem-homogenization-lambda} with $\lambda=+\infty$. Actually in that case, we will see during the proof that the convergence of $v^\varepsilon$ towards $v^0$ holds in fact in $W^{1,\infty}$.
\item During the proof, we will technically need the assumption \eqref{hyp-param-5} only in the regime where $\lambda = 0$. 
\end{enumerate}
\end{remark}
\begin{remark}[The corrector]
The well-definition of $c^0$ in the item \ref{corrector} is an easy consequence of Theorem \ref{front-well-posed}. We also know that $w$ is unique up to an additive constant. The $w$ such that $\overline{w}=0$ will roughly speaking be a corrector in relation with the ansatz $v^\varepsilon(y) \approx \varepsilon \, w \left(\frac{y}{\varepsilon} \right)$. 
\end{remark}

\begin{remark}[The homogenized speed]
Here again $v^0=0$ and $u^0$ is deduced from $c^0$ through the formula in \eqref{homogenized-triplet}. 
The knowledge of the mapping 
$$
\lambda \mapsto c^0=c^0(\lambda)
$$
then entirely determines the homogenized triplets parametrized by $\lambda$. A qualitative analysis of this mapping will be done in the next section.
\end{remark}

\subsection{Proofs of the homogenization results}

Let us prove Theorems \ref{Theorem-homogenization} and \ref{Theorem-homogenization-lambda}. We proceed by giving a few lemmas and we start with a compacity result.
\begin{lemma}
\label{homogenization-compacity} 
Let us assume the hypotheses of Theorem \ref{Theorem-homogenization-lambda}. Then there exist $c_0 >0$, $v^0 \in C_b(\R)$ and $u^0 \in C_b(\R^2)$, 
such that 
\begin{equation*}
\begin{cases}
\mbox{$c^{\varepsilon} \to c^{0}$ in $\mathbb{R}$,}\\
\mbox{$v^{\varepsilon} \to v^{0}$ uniformly on $\mathbb{R}$,}\\
\mbox{$u^{\varepsilon} \to u^{0}$ uniformly on $\R^2$,}
\end{cases}
\end{equation*} 
as $\varepsilon \downarrow 0$ and where $u^\varepsilon$ is extended to $\R^2$ by reflection. This convergence holds more precisely at least along some sequence $\{\varepsilon_n\}_n$ converging to zero.
\end{lemma}

\begin{proof}
By the compactness theorems of Bolzano-Weierstrass and Ascoli-Arz\'ela, it  suffices to prove that for some $\varepsilon_0$ small enough:
\begin{equation}\label{claim-complicate}
\begin{cases}
\mbox{The family $\{c^\varepsilon\}_{\varepsilon \in (0,\varepsilon_0]}$ is positively bounded from below and above.}\\
\mbox{The family $\{v^\varepsilon\}_{\varepsilon \in (0,\varepsilon_0]}$ is bounded and equicontinuous on $\R$.}\\
\mbox{The family $\{u^\varepsilon\}_{\varepsilon \in (0,\varepsilon_0]}$ is bounded and equicontinuous on $\R^2$.}
\end{cases}
\end{equation}

Let us consider two cases depending on which condition holds in \eqref{H}.

\medskip

{\bf 1.} {\it Proof of \eqref{claim-complicate} if $\lambda>\frac{4 \, R_M}{\pi}$.} Let us apply Lemma \ref{estimates-tw} (in the appendix) and Theorem \ref{Holder-regularity} to each travelling wave solution $(c^\varepsilon,v^\varepsilon,u^\varepsilon)$. That is to say let us apply these results with: 
\begin{equation}\label{redaction-compliquee}
\begin{cases}
\mbox{the parameters $y \mapsto a \left(\frac{y}{\varepsilon}\right)$, $y \mapsto b \left(\frac{y}{\varepsilon}\right)$ and $y \mapsto g \left(\frac{y}{\varepsilon}\right)$},\\
\mbox{the combustion rate $(y,T) \mapsto R \left(\frac{y}{\varepsilon},T\right)$},\\
\mbox{the period $Y=\varepsilon$},\\
\mbox{and the curvature coefficient $\mu=\mu(\varepsilon)$.}
\end{cases}
\end{equation}
This gives us estimates depending only on the bounds in \ref{hyp-param-1}--\ref{hyp-arrh-4} which can be taken independent of $\varepsilon$, thanks to the special form of the parameters in \eqref{redaction-compliquee}. These estimates also depend on some $Y_0$ and $\lambda_0$. If we can choose them to be independent of $\varepsilon$ too, then the proof of \eqref{claim-complicate} will be complete and so will be the proof of the lemma. 

The $Y_0$ needs to satisfy $0<Y \leq Y_0$ and the choice $Y_0=\varepsilon_0$ clearly works. It now only remains to choose $\lambda_0$ such that \eqref{hyp-both} holds. Let us take some $\lambda_0$ such that 
$
\lambda > \lambda_0>\frac{4 \, R_M}{\pi}.
$ 
Taking a smaller $\varepsilon_0$ if necessary, we infer that
$
\frac{\mu(\varepsilon)}{\varepsilon} \geq \lambda_0>\frac{4 \, R_M}{\pi}
$ 
for all $\varepsilon \leq \varepsilon_0$. This is exactly the first condition of \eqref{hyp-both} which completes the proof.  

\medskip

{\bf 2.} {\it Proof of \eqref{claim-complicate} if $\lim_{T \downarrow 0} \left\{|\ln T| \, \essinf_z R(z,T) \right\}=+\infty$.} We argue as before, but now we choose $R_0$ to have the second condition of \eqref{hyp-both} (thus with the combustion rate in \eqref{redaction-compliquee}). Clearly   
$
R_0(T):=\essinf_z R(z,T)
$
works for each $\varepsilon$ and
is independent of $\varepsilon$. This completes the proof.
\end{proof}


We proceed by giving some properties on $v^0$ and $u^0$.

\begin{lemma}
\label{homogenization-limit-de-v} 
Let $(v^0,u^0) \in C_b(\R) \times C_b(\R^2)$ be given by the
preceding lemma. Then $v^0 \equiv 0$ and $u^0=u^0(x)$ does not depend on $y$.
\end{lemma}

\begin{proof}
Let $a<b$ and let us first show that $u^0(x,a)=u^0(x,b)$ for any $x \in \R$.
We have%
\begin{equation*}
b=a+\frac{b-a}{\varepsilon} \, \varepsilon \quad \text{so that} \quad \underbrace{a+\llcorner\frac
{b-a}{\varepsilon}\lrcorner \, \varepsilon}_{=: b_\varepsilon} \leq b \quad \forall \varepsilon>0,
\end{equation*}
where throughout the symbol ``$\llcorner \cdot \lrcorner$'' is used for the
lower integer part. By periodicity
$
u^{\varepsilon} (x,b_\varepsilon)=u^\varepsilon(x,a)$ and we also know that
$b_\varepsilon$ converges to $b$, as $\varepsilon \downarrow 0$. Let us pass to the limit in the last equality by using the uniform convergence of $u^\varepsilon$ towards $u^0$ (holding at least along the sequence given in Lemma \ref{homogenization-compacity}). We get that
$$
u^0(x,b)=\lim_{\varepsilon \downarrow 0} u^{\varepsilon} (x,b_\varepsilon)=\lim_{\varepsilon \downarrow 0} u^\varepsilon(x,a)=u^0(x,a),
$$
which completes the proof that $u^0=u^0(x)$. We proceed in the same way to show that $v^0$ does not depend on $y$. In particular, $v^0$ is a constant which is necessarily zero, since $\overline{v}^\varepsilon=0$ (see the assumptions of Theorem \ref{Theorem-homogenization-lambda}).
\end{proof}

\begin{lemma}
\label{convergence-properties-u0} 
Let $u^0 \in C_b(\R)$, $u^0=u^0(x)$, be given by the
preceding lemmas. Then $u^0$ is even, tends to zero at infinity, and $u^\varepsilon \to u^0$ in $L^p_{\loc}(\R_y,L^p(\R_x))$ for all $p \geq 1$ and as $\varepsilon \downarrow 0$ (at least along the sequence given by Lemma \ref{homogenization-compacity}).
\end{lemma}

\begin{proof}
Any uniform limit $u^0$ of $u^\varepsilon$ on $\R^2$ is clearly even in $x$, because of the choice of the extension \eqref{eps-extension}. To show the other properties, it suffices to verify that 
\begin{equation}\label{claim-integrability}
|u^\varepsilon(x,y)| \leq C \, e^{-C \, |x|} \quad \forall x,y \in \R^2,
\end{equation} 
for some positive constant $C$ independent of $\varepsilon \leq \varepsilon_0$. Indeed, the zero limit at infinity would be preserved when $\varepsilon \downarrow 0$ and the $L^p$ convergence would follow from the dominated convergence theorem. Let us thus show \eqref{claim-integrability}. First recall that the parameters $y \mapsto a \left(\frac{y}{\varepsilon} \right)$, $y \mapsto b \left(\frac{y}{\varepsilon} \right)$, etc., in \eqref{eps-T-eqn-bis}--\eqref{eps-front-eqn-bis} satisfy \ref{hyp-param-1}--\ref{hyp-arrh-4} with the same bounds $a_m$, $b_m$, etc., independent of $\varepsilon$. Recall also that the ($\varepsilon$-periodic) front $v^\varepsilon$ is bounded on $\R$ independently of $\varepsilon \leq \varepsilon_0$ (thanks to the bound on $v^\varepsilon_y$ in \eqref{claim-complicate} and the fact that $\overline{v}^\varepsilon=0$). Recall finally that $c^\varepsilon \geq c_m >0$ for some $c_m$ independent of $\varepsilon \leq \varepsilon_0$ (again by \eqref{claim-complicate}). With all these facts in hands, it is clear from Theorem \ref{T-well-posed}\eqref{Linfty-esti} that 
$$
0 \leq u^\varepsilon(x,y) \leq \tilde{C} \,e^{\tilde{C} \, x} \quad \mbox{on} \quad \{x \leq v^\varepsilon(y)\},
$$ 
for some $\tilde{C}>0$ independent of $\varepsilon \leq \varepsilon_0$. The rest of the proof of \eqref{claim-integrability} is easy from the choice of the extension \eqref{eps-extension}.
\end{proof}

We can now identify $u^0$.
\begin{lemma}
\label{Lemme-donnant-u0} Let $c^0>0$ and $u^0 \in C_b(\R)$, $u_0=u_0(x)$, be given by the preceding lemmas. Then for all $x \in \R$,
\begin{equation*}
u^{0}(x)=\frac{\overline{g}}{\overline{b}} \exp \left(-\frac{c^{0} \, \overline{b}
\, |x|}{\overline{a}}\right),
\end{equation*}
and where the speed $c^0$ will be identified later. 
\end{lemma}

\begin{proof}
Let us pass to the limit in the equation satisfied by $u^{\varepsilon}$ (along the sequence given by Lemma \ref{homogenization-compacity}). Recall that we have 
\begin{equation*}
\int_{0}^{\varepsilon} \int_{-\infty}^{v^{\varepsilon}(y)} \left(  c^{\varepsilon
}\,b^{\varepsilon}\,u_{x}^{\varepsilon}\,w+a^{\varepsilon}\,\nabla
u^{\varepsilon}\,\nabla w\right) \dif x \, \dif y =\int_{x=v^{\varepsilon}%
(y),0<y<\varepsilon}\frac{c^{\varepsilon}\,g^{\varepsilon}}{\sqrt
{1+v_{y}^{\varepsilon}}}\,w,
\end{equation*}
for all $w\in H_{\per}^{1}(\Omega^{\varepsilon}),$ see Definition
\ref{defn-eps-existence}. By the periodicity of $u^{\varepsilon}$ (and $w$), we can rewrite this equality as
\begin{multline*}
\int_{0}^{l_{\varepsilon}} \int_{-\infty}^{v^{\varepsilon}(y)}c^{\varepsilon
}\,b^{\varepsilon}\,u_{x}^{\varepsilon}\,w \, \dif x \, \dif y 
+\int_{0}^{l_{\varepsilon}}%
\int_{-\infty}^{v^{\varepsilon}(y)}a^{\varepsilon}\,\nabla u^{\varepsilon
}\,\nabla w \, \dif x \, \dif y \\
=\int_{x=v^{\varepsilon}(y),0<y<l_{\varepsilon}}\frac
{c^{\varepsilon}\,g^{\varepsilon}}{\sqrt{1+v_{y}^{\varepsilon}}}%
\,w,
\end{multline*}
where $l_{\varepsilon}=\llcorner\frac{1}{\varepsilon}\lrcorner \, \varepsilon$. Note that
$
l_{\varepsilon} \to 1$ as $\varepsilon\downarrow 0.$ 
It is in this latter equation that we are going to pass to the limit by proceeding term by term. Note that we can restrict ourselves to $w$ of the form $w(x,y)=\varphi(x)$ with an arbitrary
$\varphi\in C_{c}^{\infty}(\mathbb{R})$. This is motivated by the fact that the limit $u^0$ is already known to be independent of $y$, see Lemma \ref{homogenization-limit-de-v}. We thus have to pass to the limit in the equation:
\begin{equation}
\label{homog-eqn-uepsbis}%
\begin{split}
0 & =  \int_{0}^{l_{\varepsilon}} \int_{-\infty}^{v^{\varepsilon}(y)}c^{\varepsilon
}\,b\left(\frac{y}{\varepsilon}\right) u_{x}^{\varepsilon}\, \varphi \, \dif x \, \dif y \\
& \quad +\int_{0}^{l_{\varepsilon}}%
\int_{-\infty}^{v^{\varepsilon}(y)} a\left(\frac{y}{\varepsilon}\right) u_x^{\varepsilon
}  \, \varphi_x \, \dif x \, \dif y\\
&  \quad -\int_{0}^{l_{\varepsilon}} c^{\varepsilon}\,g\left(\frac{y}{\varepsilon}\right) \varphi(v^\varepsilon(y)) \, \dif y\\
& =: I_1+I_2+I_3.
\end{split}
\end{equation}

{\bf 1.} \textit{The first term.} We have
\begin{equation}\label{form-I1}
\begin{split}
I_{1}= & -\int_{\R^2} c^{\varepsilon}\,b\left(\frac{y}{\varepsilon}\right) u^{\varepsilon}\,\varphi_{x} \, \mathbf{1}_{x<v^\varepsilon(y),0<y<l_\varepsilon} \, \dif x \, \dif y\\
 & \quad +\int_{\R} c^{\varepsilon}\,b\left(\frac{y}{\varepsilon}\right) u^{\varepsilon}(v^{\varepsilon}(y),y) \, \varphi(v^{\varepsilon}(y)) \, \mathbf{1}_{0<y<l_\varepsilon}\, \dif y,
\end{split}
\end{equation}
thanks to an integration by parts in $x$ and a rewriting of the obtained integrals with indicator functions. From classical results and the preceding lemmas, we have:
\begin{equation*}
\begin{cases}
\mbox{$b\left(\frac{\cdot}{\varepsilon} \right) \rightharpoonup \overline{b}$ in $L^{\infty}(\R)$ weak-$\star$,}\\
\mbox{$c^{\varepsilon} \to c^0$, $v^\varepsilon \to 0$ uniformly on $\R$,}\\
\mbox{and $u^\varepsilon \to u^0$ uniformly on $\R^2$},
\end{cases}
\end{equation*}
as $\varepsilon \downarrow 0$. In particular, we also have:
$$
\begin{cases}
\mbox{$\mathbf{1}_{x<v^\varepsilon(y),0<y<l_\varepsilon} \rightarrow \mathbf{1}_{x<0,0<y<1}$ in $L_\loc^1(\mathbb{R}_x \times \R_y)$},\\
\mbox{$\mathbf{1}_{0<y<l_\varepsilon} \to \mathbf{1}_{0<y<1}$ in $L^1_{\loc}(\R_y)$,}\\
\mbox{and $u^{\varepsilon}(v^{\varepsilon}(\cdot),\cdot) \to u^0(0)$ uniformly on $\R$.} 
\end{cases}
$$ 
To show the above convergence for the indicator functions, it suffices to use again the arguments of the proof of Lemma \ref{stability}. Let us now pass to the limit in $I_1$, which is possible by weak-strong convergence arguments. We get:
\begin{equation*}
\begin{split}
\lim_{\varepsilon \downarrow 0} I_1
& = -\int_{-\infty}^{0} \int_{0}^1 c^{0}\,\overline{b}
\, u^{0}(x) \, \varphi_{x}(x) \, \dif y \, \dif x +\int_{0}^1  c^{0}\,\overline{b}\,\,u^{0}%
(0) \, \varphi(0) \, \dif y\\
& = -\int_{-\infty}^{0} c^{0}\,\overline{b}
\, u^{0}(x) \, \varphi_{x}(x) \, \dif x +c^{0}\,\overline{b}\,u^{0}(0) \, \varphi(0).
\end{split}
\end{equation*}

{\bf 2.} \textit{The second and third terms.} The term $I_2$ and $I_3$ are of the same form as the integrals in \eqref{form-I1}. One can verify that the same reasoning leads to  
\begin{equation*}
\lim_{\varepsilon \downarrow 0} I_{2} =  -\int_{-\infty}^{0} \overline{a} \, u^{0}(x) \, \varphi_{xx}(x) \, \dif x+\overline{a} \,u^{0}(0) \, \varphi_x(0)
\end{equation*}
and
\begin{equation*}
\lim_{\varepsilon \downarrow 0} I_3=- \int_0^1 c^{0} \, \overline{g} \, \varphi(0) \, \dif y=-c^{0} \, \overline{g} \, \varphi(0).
\end{equation*}

{\bf 3.} {\it Conclusion.} Finally in the
limit $\varepsilon \downarrow 0,$ Equation $(\ref{homog-eqn-uepsbis})$
becomes
\begin{equation}
\int_{-\infty}^{0} \left(c^{0} \, \overline{b} \, u^{0} \, \varphi_x+\overline{a} \,
u^{0} \, \varphi_{xx} \right) \dif x=c^{0} \left(\overline{b} \, u^0(0)- \overline{g} \right) \varphi(0)+\overline{a} \, u^0(0) \, \varphi_x(0)\label{PV-homog}
\end{equation}
for all $\varphi\in C_{c}^{\infty}(\mathbb{R})$.
We recognize the weak formulation of Problem \eqref{homogenized-T-eqn} and, from that, it is quite standard to identify $u^0$. Let us give some details for completeness.

By taking $\varphi$ with compact support in $\{x<0\}$ in \eqref{PV-homog}, we obtain that
\begin{equation}
c^{0} \, \overline{b} \, u_x^{0}-\overline{a} \, u_{xx}^{0}=0 \quad \text{for} \quad x<0
\label{EqhomD'}
\end{equation}
(in the distribution sense).
Thus
$
u^0(x)=C \exp \left( \frac{c^{0} \, \overline{b} \,
x}{\overline{a}}\right)+\tilde{C}
$
for all $x<0$.
Using that $u^0$ is continuous on $\R$, even, and tends to zero at infinity, see Lemma \ref{convergence-properties-u0}, we infer that $\tilde{C}=0$ and
$$
u^0(x)=C \exp \left( \frac{c^{0} \, \overline{b} \,
|x|}{\overline{a}}\right) \quad \forall x \in \R.
$$
Injecting this formula in \eqref{PV-homog} and integrating by parts, 
\begin{equation}\label{eq-for-boundary}
\int_{-\infty}^{0} \underbrace{\left(c^{0} \, \overline{b} \, u_{x}^{0}-\overline{a}
\, u_{xx}^{0} \right)}_{=0} \varphi \, \dif x+\overline{a} \underbrace{u_{x}^{0}(0)}_{=C  \, \frac{c^{0} \, \overline{b}}{\overline{a}}} \varphi
(0)=c^{0} \, \overline{g} \, \varphi(0),
\end{equation}
for all $\varphi \in C_c^\infty(\R)$. This implies that $C=\frac{\overline{g}}{\overline{b}}$ and completes the proof.
\end{proof}

\begin{remark}
As noted previously, $u^0$ satisfies the homogenized problem 
$$
\begin{cases}
c^{0} \, \overline{b} \, u_x^{0}-\overline{a} \, u_{xx}^{0}=0, \quad x<0,\\
\overline{a} \, u^0_x(0)=c^0 \, \overline{g},
\end{cases}
$$
derived in \eqref{EqhomD'} and \eqref{eq-for-boundary}.
\end{remark}

To identify $c^0$, we will roughly speaking homogenize the front's equation \eqref{eps-front-eqn-bis} (with $\mu=\mu(\varepsilon)$ for Theorem \ref{Theorem-homogenization-lambda}). To do so, it will be convenient to rewrite \eqref{eps-front-eqn-bis} as 
\begin{equation}
-c^{\varepsilon}+R\left(z,u^{\varepsilon}(v^{\varepsilon} (\varepsilon
\, z),\varepsilon \, z)\right) \sqrt{1+(w_{z}^{\varepsilon})^{2}}=\frac{\mu(\varepsilon
)}{\varepsilon} \, \frac{w_{zz}^{\varepsilon}}{1+(w_{z}^{\varepsilon})^{2}%
},\label{Homog-eqn-weps}%
\end{equation}
thanks to the change of variables $z=\frac{y}{\varepsilon}$ and
\begin{equation}\label{ansatz}
w^{\varepsilon}(z):=\frac{v^{\varepsilon}(\varepsilon \, z)}{\varepsilon}.
\end{equation}
We shall see later that the limit $w$ of $w^\varepsilon$ is such that $v^\varepsilon(y) \approx \varepsilon \, w \left(\frac{y}{\varepsilon}\right)$. Here are some technical properties that will be needed to identify $c^0$. 

\begin{lemma}\label{lem-corrector}
Let $\{\varepsilon_n\}_n$ be the sequence given by Lemma \ref{homogenization-compacity}. Then:
\begin{enumerate}[label={{\rm (\roman*)}}]
\item\label{conv-combustion-corrector} The function $R \left(\cdot,u^{\varepsilon_n}(v^{\varepsilon_n} (\varepsilon_n
\, \cdot),\varepsilon_n \, \cdot)\right) \to R \left( \cdot,\frac{\overline{g}}{\overline{b}}\right)$ in $L^1(0,1)$ as $n \to +\infty$.
\item\label{bound-corrector} The sequence $\{w^{\varepsilon_n}\}_{n}$ is bounded in $W^{1,\infty}(\R)$.
\item\label{bound-lambda-finite} The latter sequence is bounded in $W^{2,\infty}(\R)$ if $\lambda =\lim_{\varepsilon \downarrow 0} \frac{\mu(\varepsilon)}{\varepsilon} > 0$. 
\end{enumerate}
\end{lemma}

\begin{proof}
Let us start with \ref{conv-combustion-corrector}. By Lemma \ref{homogenization-compacity}, 
$$
u^{\varepsilon_n}(v^{\varepsilon_n}(\varepsilon \, z),\varepsilon \, z) \to u^0(0)
$$ 
uniformly in $z$; recall indeed that $v^0 \equiv 0$ and $u^0=u^0(x)$ by Lemma \ref{homogenization-limit-de-v}. By the identification of $u^0$ in Lemma \ref{Lemme-donnant-u0}, we also know that $u^0(0)=\frac{\overline{g}}{\overline{b}}$. Thus by \ref{hyp-arrh-1},
$$
R \left(z,u^{\varepsilon_n}(v^{\varepsilon_n}(\varepsilon_n \, z),\varepsilon_n \, z)\right)\to
R \left(z,\frac{\overline{g}}{\overline{b}}\right)
$$ 
for almost every $z \in \R$. Using in addition that $R$ is bounded by \ref{hyp-arrh-3}, the dominated convergence theorem implies the desired convergence in \ref{conv-combustion-corrector}. 

For the second item, we will apply Theorem \ref{front-well-posed} to Equation \eqref{Homog-eqn-weps}. Due to the uniform convergence of $u^{\varepsilon_n}(v^{\varepsilon_n}(\varepsilon_n
\, z),z)$ to $u^0(0)$, we know that for all $z$ and sufficiently large values of $n$, we have for example 
$$
u^{\varepsilon_n}(v^{\varepsilon_n}(\varepsilon_n
\, z),z) \geq \frac{u^0(0)}{2} =\frac{\overline{g}}{2 \, \overline{b}}>0.
$$
The assumptions \ref{hyp-arrh-1} and \ref{hyp-arrh-4} then imply that
\begin{equation*}
\essinf_z R\left(z,u^{\varepsilon_n}(v^{\varepsilon_n}(\varepsilon_n
\, z),\varepsilon_n \, z)\right) \geq \essinf_z R\left(z,\frac{\overline{g}}{2 \, \overline{b}}\right)>0.
\end{equation*}
Hence, using also the other bound 
$$
\essinf_z R\left(z,u^{\varepsilon_n}(v^{\varepsilon_n}(\varepsilon_n
\, z),\varepsilon_n \, z)\right) \leq R_M
$$
from \ref{hyp-arrh-3}, Estimate \eqref{front-esti} implies that the sequence of derivatives $\{w_z^{\varepsilon_n}\}_n$ is bounded by some $C$ in $L^\infty(\R)$. To get some bound on the antiderivatives, we note that each $w^{\varepsilon_n}$ is $1$-periodic with $\int_0^1 w^{\varepsilon_n}=\frac{1}{\varepsilon_n^2} \int_0^{\varepsilon_n} v^{\varepsilon_n}=0$ thanks to \eqref{ansatz}. This implies that $\|w^{\varepsilon_n}\|_\infty \leq \|w^{\varepsilon_n}_z\|_\infty \leq C$ and the proof of the item \ref{bound-corrector} is complete.

The item \ref{bound-lambda-finite} immediately follows from the latter item and Equation \eqref{Homog-eqn-weps}. 
\end{proof}

We can now prove Theorems \ref{Theorem-homogenization} and \ref{Theorem-homogenization-lambda}. We will denote again the sequence $\{\varepsilon_n\}_n$ given by Lemma \ref{homogenization-compacity}, or any of its subsequences, simply by $\{\varepsilon\}$. 

\begin{proof}[Proof of Theorem \ref{Theorem-homogenization}]
By Lemmas \ref{homogenization-compacity}--\ref{Lemme-donnant-u0}, we have already proved the convergence 
of $v^{\varepsilon}$ and $u^{\varepsilon}$ towards the desired limits (at least along the above sequence). It
therefore remains to identify $c^{0}$. 

For this sake, let us integrate Equation \eqref{Homog-eqn-weps}
between $0$ and $1$ (thus with $\mu$ fixed here). We get
\begin{equation*}
-c^{\varepsilon} +\int_{0}^{1}R \left(z,u^{\varepsilon}(v^{\varepsilon
}(\varepsilon \, z),\varepsilon \, z)\right) \sqrt{1+(w_{z}^{\varepsilon})^{2}} \, \dif y=\frac{\mu}{\varepsilon} \left[ \arctan(w_{z}%
^{\varepsilon})\right]_{0}^{1}.
\end{equation*}
As the right-hand side vanishes due to the $1$-periodicity of $w^\varepsilon$, we have
\begin{equation}
c^{\varepsilon} =\int_{0}^{1} R \left(z,u^{\varepsilon}(v^{\varepsilon
}(\varepsilon \, z),\varepsilon \, z)\right) \sqrt{1+(w_z^{\varepsilon})^{2}} \, \dif z.\label{Efrontintegral}%
\end{equation}
Note then that $\|w_z^\varepsilon\|_\infty=\|v_y^\varepsilon\|_\infty$ and recall from Lemma \ref{bound-front} that 
$$
\|v^{\varepsilon}_{y}\|_{\infty}\leq\tan\left(  \frac{2 \, c^{\varepsilon}\,\varepsilon}{\mu}\right).
$$
This implies that $w_{z}^{\varepsilon} \to 0$ uniformly as $\varepsilon
\downarrow 0$. Therefore in the limit $\varepsilon \downarrow 0$, \eqref{Efrontintegral} and Lemma \ref{lem-corrector}\ref{conv-combustion-corrector} lead to
$$
c^\varepsilon \to \int_{0}^{1} R \left(z,\frac{\overline{g}}{\overline{b}} \right) \dif z.
$$
This identifies the limiting speed $c^0$. 

To conclude, we have established the convergence of the triplet $(c^\varepsilon,v^\varepsilon,u^\varepsilon)$ towards
$$
c^0= \int_0^1 R\left( z,\frac{\overline{g}}{\overline{b}}\right) \dif z, \quad v^0 \equiv 0 \quad \mbox{and} \quad u^0=u^0(x)=\frac{\overline{g}}{\overline{b}} \, \exp \left(-\frac{c^{0} \, \overline{b}
\, |x|}{\overline{a}}\right), 
$$
along some sequence $\{\varepsilon_n\}_n$ converging to zero. But as the limiting triplet is uniquely determined, the convergence holds for the whole family $\varepsilon \downarrow 0$. This implies the desired result and completes the proof.  
\end{proof}

\begin{proof}[Proof of Theorem \ref{Theorem-homogenization-lambda}]
Here too it remains to identify $c^{0}$. The new difficulty is that we may no longer have the strong convergence of $w_z^\varepsilon$ towards zero. This will complicate the passage to the limit in \eqref{Efrontintegral}. Note also that, as above, the identification of $c^0$ will automatically gives us the convergence of the whole family as $\varepsilon \downarrow 0$. We thus continue to argue along some particular sequence (simply denoted by $\{\varepsilon\}$).

\medskip

{\bf 1.} {\it The case $\lambda=+\infty$.} This is the only case for which $w^\varepsilon_z$ still strongly converges towards zero. Indeed by \eqref{ansatz} and the estimate of Lemma
\ref{bound-front}, we have
\begin{equation}\label{error-estimate}
\|w^\varepsilon_z\|_\infty =\| v_{y}^{\varepsilon}\|_{\infty}\leq\tan\left(  \frac{2 \, c^{\varepsilon}\,\varepsilon}{\mu(\varepsilon)}\right),
\end{equation}
and as $\frac{\varepsilon}{\mu(\varepsilon)}$ goes to zero as $\varepsilon\downarrow 0$, the preceding arguments still apply to give $c^{0}=\int_0^1 R\left(z,\frac{\overline{g}}{\overline{b}}\right) \dif z.$

\medskip

{\bf 2.} {\it The case $\lambda\in (0,+\infty)$.} In that case we need to identify the limit $w$ of $w^\varepsilon$, defined in \eqref{ansatz}, and the equation it satisfies. We have seen in Lemma \ref{lem-corrector}\ref{bound-lambda-finite} that the function $w^\varepsilon$ remains bounded in $W^{2,\infty}(\R)$ as $\varepsilon \downarrow 0$ (at least along the sequence given by Lemma \ref{homogenization-compacity}). By Ascoli-Arz\'ela theorem, it then converges towards some $w$ in $W^{1,\infty}(\R)$  (up to some subsequence). This limit $w$ is necessarily in $W^{2,\infty}(\R)$ and is also $1$-periodic. To identify the equation in $w$, let us rewrite \eqref{Homog-eqn-weps} as
$$
w^\varepsilon_{zz}=\frac{\varepsilon}{\mu(\varepsilon)} \left\{1+(w_{z}^{\varepsilon})^{2}\right\}
\left\{-c^{\varepsilon}+R(z,u^{\varepsilon}(v^{\varepsilon}(\varepsilon
\, z),\varepsilon \, z)) \sqrt{1+(w_{z}^{\varepsilon})^{2}}\right\}.
$$
Using Lemma \ref{lem-corrector}\ref{conv-combustion-corrector}, we infer that $w^{\varepsilon}_{zz}(\cdot)$ converges in $L^1(0,1)$ towards
$$
\frac{1}{\lambda} \left\{1+w_{z}^{2}(\cdot)\right\}
\left\{-c^0+R \left(\cdot,\frac{\overline{g}}{\overline{b}}\right)\sqrt{1+w_{z}^{2}(\cdot)}\right\}.
$$
But this function is necessarily $w_{zz}(\cdot)$ by uniqueness of the distributional limit. Hence $w$ satisfies
\begin{equation*}
-c^{0}+R \left(z,\frac{\overline{g}}{\overline{b}}\right)\sqrt{1+w_{z}^{2}}=\lambda \, \frac{w_{zz}}%
{1+w_{z}^{2}} \quad \mbox{almost everywhere.}
\end{equation*}

We then conclude by Theorem \ref{front-well-posed} which gives the existence of a unique real $c^0$ such that the above equation admits a $1$-periodic solution $w \in W^{2,\infty}(\R)$. Note that the $c^0$ thus identified will depend on $\lambda$.   

\medskip

{\bf 3.} {\it The case $\lambda=0$.} We must show that $c^0=\esssup_z R\left(z,\frac{\overline{g}}{\overline{b}}\right)$. Let us start by applying Theorem \ref{front-well-posed}\eqref{front-esti} to the front's equation \eqref{eps-front-eqn-bis}. We get
$$
c^\varepsilon \leq \esssup_{y \in \R} \Big\{R \left(\frac{y}{\varepsilon},u^\varepsilon(v^\varepsilon(y),y)\right) \Big\}.
$$
Using \ref{hyp-arrh-1}, we infer that
$$
c^\varepsilon \leq \esssup_{z \in \R} R \left(z,T^\varepsilon\right),
$$ 
with $T^\varepsilon:=\max_y u^\varepsilon(v^\varepsilon(y),y)$.
Recalling that $c^\varepsilon \to c^0$ and $u^\varepsilon(v^\varepsilon(\cdot),\cdot) \to u^0(0)=\frac{\overline{g}}{\overline{b}}$ uniformly on $\R$, we have $T^\varepsilon \to \frac{\overline{g}}{\overline{b}}$ and thus
$$
c^0 \leq \limsup_{T \to \frac{\overline{g}}{\overline{b}}} \left\{\esssup_z R \left(z,T\right)\right\}.
$$
We conclude that $c^0 \leq \esssup_z R \left(z,\frac{\overline{g}}{\overline{b}}\right)$ by using \eqref{hyp-param-5}.

To prove the inequality in the other direction, we consider Equation \eqref{Homog-eqn-weps}. Given any nonnegative $\varphi \in C^\infty_c(\R)$, we multiply \eqref{Homog-eqn-weps} by $\varphi$ and integrate the right-hand side by part. We get that
\begin{equation*}
\int_\R \left\{-c^{\varepsilon}+R\left(z,u^{\varepsilon}(v^{\varepsilon}(\varepsilon
\, z),\varepsilon \, z)\right) \sqrt{1+(w_{z}^{\varepsilon})^{2}}\right\} \varphi \, \dif z=-\frac{\mu(\varepsilon
)}{\varepsilon} \int_\R \varphi_z \arctan \left(w^\varepsilon_z\right).
\end{equation*}
From Lemma \ref{lem-corrector}\ref{conv-combustion-corrector}--\ref{bound-corrector} and the fact that $\frac{\mu(\varepsilon
)}{\varepsilon} \to 0$ as $\varepsilon \downarrow 0$, we easily deduce that
$$
c^0 \int_\R \varphi\geq \int_\R R \left(z,\frac{\overline{g}}{\overline{b}}\right) \varphi(z) \, \dif z,
$$
at the limit.
Since the nonnegative test function $\varphi$ is arbitrary, $c^0 \geq \esssup_z R \left(z,\frac{\overline{g}}{\overline{b}}\right)$ and the proof is complete. 
\end{proof}

\section{Monotonicity of the homogenized speed}\label{monotonicity}

In this section, we consider the qualitative analysis of the speed $c^0$ as defined by Theorem \ref{Theorem-homogenization-lambda}. Let us recall that it depends on $\lambda = \lim_{\varepsilon \downarrow 0} \frac{\mu(\varepsilon)}{\varepsilon}$. Our main result will be that 
$
\lambda \mapsto c^0(\lambda)
$ 
is monotonous. 

\subsection{Main result}

For brevity, we will denote $c^0(\lambda)$ simply by $c(\lambda)$ all along this section. Its derivative in $\lambda$ will be denoted by $c'(\lambda)$. The gradient of the corrector $w=w(z)$ given by Theorem \ref{Theorem-homogenization-lambda}\ref{corrector} will be denoted by $h=w_z$. This function satisfies the problem:
\begin{equation}
\begin{cases}
-c+\mathscr{R}(z) \sqrt{1+h^{2}}=\lambda \, \frac{h_{z}}{1+h^{2}},\\[1ex]
h(z)=h(z+1),\\[1ex] 
\int_{0}^{1} h(t) \, \dif t=0,
\end{cases}
\label{Eqn1-Cmonotone}
\end{equation}
for almost every $z \in \R$, where hereafter 
$$
\mathscr{R}(z):=R\left(z, \frac{\overline{g}}{\overline{b}}\right).
$$ 
Under the assumptions of Theorem \ref{Theorem-homogenization-lambda}, we have:
\begin{equation}\label{hyp-geometric}
\mbox{The function $\mathscr{R} \in L^\infty(\R)$ is $1$-periodic and positively lower bounded.}  
\end{equation}
This is the only property that we will need. For any $\lambda \in \R^+$, there then exists a unique $(c,h) \in \R \times W^{1,\infty}(\R)$ satisfying \eqref{Eqn1-Cmonotone}, thanks to Theorem \ref{front-well-posed}. This defines a mapping
\begin{equation}\label{correction-def-c}
\lambda \in \R^+ \mapsto \left(c(\lambda),h(\cdot,\lambda) \right) \in \R \times W_\diese^{1, \infty}(\R)
\end{equation}
(the subscript ``$\diese$'' being used for the $1$-periodicity). Here is the main result of this section.

\begin{theorem}
\label{Theorem-C-monotone} Assume \eqref{hyp-geometric}. Then $\lambda \mapsto c(\lambda)$ (defined as above) is $C^\infty$ and nonincreasing. At the limits, it satisfies:
\begin{equation*}
\lim_{\lambda \downarrow 0} c(\lambda)=\esssup_z \mathscr{R}(z)\quad \mbox{and} \quad \lim_{\lambda \to +\infty} c(\lambda)=\int_0^1 \mathscr{R}(z) \, \dif z.
\end{equation*}
More precisely $c'(\lambda)<0$ for all $\lambda >0$ as soon as $\mathscr{R}$ is not a constant.
\end{theorem}

\begin{remark}\label{rem-adherence-value}
As a corollary of the above theorem, the speed $c^0=c^0(\lambda)$ of Theorem~\ref{Theorem-homogenization-lambda} defines a continuous map of the form:
$$
\lambda \in [0,+\infty] \mapsto c^0(\lambda) \in \left[\esssup_z R \left(z,\frac{\overline{g}}{\overline{b}} \right),\int_0^1 R\left(z,\frac{\overline{g}}{\overline{b}} \right) \dif z \right].
$$ 
This map is smooth in $(0,+\infty)$, and it is decreasing and bijective whenever $R$ is not constant. In that case, it is in particular injective and the convergence stated in Theorem \ref{Theorem-homogenization-lambda} holds if and only if the limit $\lambda=\lim_{\varepsilon \downarrow 0} \frac{\mu(\varepsilon)}{\varepsilon}$ exists. 
\end{remark}

The limit as $\lambda \downarrow 0$ has been established in \cite{ChNa97}. To the best of our knowledge, the other properties are new. The proof of the monotonicity is the most difficult. We do not know how to prove it from a direct differentiation of \eqref{Eqn1-Cmonotone} with respect to $\lambda$. We will call instead the implicit function theorem that will give us an easy-to-use formula of $c'(\lambda)$. The rest of this section is devoted to the proof of Theorem \ref{Theorem-C-monotone}.

\subsection{Proof}

We will start by stating a version of the implicit function theorem in Banach spaces which we will use.

Let then $E_1,$ $E_2$ and $F$ be some given Banach spaces, $\mathcal{O}$ an open subset of $E_1 \times E_2$, and $\phi:\mathcal{O} \subseteq E_1 \times E_2 \to F$ a function. Also recall that:

\begin{definition}
The function $\phi$ is differentiable at $(x_1,x_2)$ if there is a bounded linear map $T:E_1 \times E_2 \to F$ such that for all $(h_1,h_2) \in E_1 \times E_2$,
$$
\phi(x_1+h_1,x_2+h_2)=\phi(x_1,x_2)+T(h_1,h_2)+\co(h_1,h_2)
$$
where  
$
\frac{\|\co(h_1,h_2)\|_F}{\|(h_1,h_2)\|_{E_1 \times E_2}} \to 0
$ 
as $\|(h_1,h_2)\|_{E_1 \times E_2} \to 0$. 
\end{definition}

In that case $T$ is unique and defined as the differential of $\phi$ at $(x_1,x_2)$. The partial differential with respect to $x_1$ and $x_2$ are defined as the maps $T_1:h_1 \mapsto T(h_1,0)$ and $T_2:h_2 \mapsto T(0,h_2)$. Throughout we shall denote them by 
$$
\dif \phi(x_1,x_2):=T \quad \mbox{and}  \quad \dif_{x_i} \phi(x_1,x_2):=T_i,
$$ 
respectively. To avoid confusion between $(x_1,x_2)$ and $(h_1,h_2)$, we shall use the standard notation
$$
\dif \phi(x_1,x_2) \cdot (h_1,h_2) := \dif \phi(x_1,x_2) (h_1,h_2) \in F
$$ 
(with similar notations for the partial differentials). This defines a map 
\begin{equation*}
\dif \phi:E_1 \times E_2 \to \mathcal{L}(E_1 \times E_{2},F),
\end{equation*}
where 
\begin{equation*}
\mathcal{L}(E_1 \times E_{2},F):= \Big\{T:E_1 \times E_2 \to F \mbox{ linear and bounded} \Big\}.
\end{equation*}
Proceeding as before, we can define the second order differential, etc. For further details, see for instance \cite{Car67}. In the implicit function theorem below, we shall use the following notation:
$$
\mbox{Isom} (E_2,F):=\Big\{T \in \mathcal{L}(E_{2},F) \mbox{ such that $T$ is bijective} \Big\}.
$$ 

Here is the theorem.

\begin{theorem}[see for instance \cite{Car67}]\label{thm-implicit}
Let $\phi$ be as above and $(x_1^0,x_2^0) \in \mathcal{O}$ be such that 
\begin{equation*}
\phi(x_1^0,x_2^0)=0.
\end{equation*}
Let us also assume that $\phi \in C^1(\mathcal{O})$ with 
\begin{equation}
\label{cond-isom}
\dif_{x_2} \phi (x_1^0,x_2^0) \in {\rm Isom}(E_2,F).
\end{equation}
Then:
\begin{enumerate}[label={{\rm (\roman*)}}]
\item\label{formula-implicit} There are open sets $U \subseteq E_1$ and $V \subseteq E_2$ and a function $\varphi:U \to V$ such that $(x_1^0,x_2^0) \in U \times V \subseteq \mathcal{O}$ and  
$$
\Big[ \phi(x_1,x_2)=0 \Leftrightarrow x_2=\varphi(x_1) \Big] \quad \forall (x_1,x_2) \in U \times V.
$$
\item\label{formula-byproduct} Moreover $\varphi \in C^1(U)$ and 
\begin{equation*}
\dif \varphi(x_2^0)=-\left[\dif_{x_2} \phi(x_1^0,x_2^0)\right]^{-1} \circ \dif_{x_1} \phi(x_1^0,x_2^0).
\end{equation*}
\end{enumerate}
\end{theorem}
In the sequel we shall use \ref{formula-byproduct} to compute $c'(\lambda)$. Before we need to give a few lemmas. Let us first precise the Banach spaces and the function $\phi$ which we will take in the frame of our problem. To this end, we consider $\left(\lambda,(c,h)\right)$ as free variables living in the following Banach spaces:
\begin{equation*}
E_{1}:=\mathbb{R}_{\lambda } \quad \mbox{and} \quad E_{2}:=\mathbb{R}_{c}\times E,
\end{equation*}
where
\begin{equation}\label{correction-def-E}
E:=\left\{h\in W_\diese^{1,\infty }(\mathbb{R}_{z}) \text{ such that }
\int_{0}^{1}h(z) \, \dif z=0\right\}.
\end{equation}
(For the sake of clarity, we have added some subscripts to the real space $\R$ in order to remember which variable is considered.) 
We then define the function $\phi$ as:
\begin{equation}
\label{def-phi}
\begin{array}{ccccl}
\phi  & : & E_{1}\times E_{2} & \to  & F \\ 
&  & \left(\lambda ,(c,h)\right) & \mapsto  & \phi\left(\lambda ,(c,h)\right):z \mapsto -c+\mathscr{R}
(z)\sqrt{1+h^{2}(z)}-\lambda \, \frac{h_{z}(z)}{1+h^{2}(z)},
\end{array}
\end{equation}
with the arrival space
\begin{equation*}
F:=L_\diese^{\infty }(\mathbb{R}_{z}).
\end{equation*}
All these spaces are endowed with their usual norms: The absolute value for $E_1$, the $\|\cdot\|_{W^{1,\infty}}$-norm for $E$, etc. Our first lemma is a computation of the partial differential of $\phi$ with respect to the $(c,h)$-variable.
%
%
%
%
%

\begin{lemma}
\label{Lemma2-C-monotone}
Assume \eqref{hyp-geometric}. 
Then the function $\phi$ as defined in \eqref{def-phi} is $C^\infty$. Moreover for all  $\lambda \in E_1$, $(c,h) \in E_{2}$ and $(\mathscr{C},\mathscr{H}) \in E_{2},$ 
\begin{equation*}
\dif_{c,h} \phi \left(\lambda,(c,h)\right) \cdot (\mathscr{C},\mathscr{H})=-\mathscr{C}+\widetilde{\mathscr{R}} \, \widetilde{\mathscr{H}}%
-\lambda \, \widetilde{\mathscr{H}}_{z},
\end{equation*}
where $\widetilde{\mathscr{R}}=\mathscr{R} \, h \sqrt{1+h^{2}}$ and $
\widetilde{\mathscr{H}}=\frac{\mathscr{H}}{1+h^{2}}.$
\end{lemma}

\begin{proof}
The function $\phi$ is smooth as composition of smooth functions---notice that $h \in W^{1,\infty} \mapsto h_z \in L^\infty$ is smooth as it is a bounded and linear map. 
To differentiate $\phi$, we consider the useful formula:
\begin{equation*}
\dif_{c,h}\phi \left(\lambda,(c,h)\right) \cdot (\mathscr{C},\mathscr{H})=\lim_{t \downarrow 0}\frac{\phi (\lambda
,(c+t \, \mathscr{C},h+t \, \mathscr{H}))-\phi (\lambda ,(c,h))}{t}.
\end{equation*}
Recall that here: $c$ and $\mathscr{C}$ are reals; $h$ and $\mathscr{H}$ are functions in $W^{1,\infty}_\diese(\R_z)$; and the limit has to be taken in $F$ that is to say strongly in $L^\infty_\diese(\R_z)$. Let us set $I=I(z)$ equal to the quotient above and let us compute its limit in $L^\infty_\diese(\R_z)$. Setting 
\begin{equation*}
F(r):=\sqrt{1+r^{2}} \quad \text{and} \quad G(r):=\frac{1}{1+r^{2}}
\end{equation*}
for any real $r$,
we can write 
$$
\phi \left(\lambda,(c,h)\right)=-c+\mathscr{R} \, F(h)-\lambda \, G(h) \, h_z.
$$
Injecting this in $I$ and rearranging the terms, we get
\begin{eqnarray*}
I =-\mathscr{C}+\mathscr{R} \, \frac{F(h+t \, \mathscr{H})-F(h)}{t}-\lambda \, \frac{G(h+t \, \mathscr{H})(h+t \, \mathscr{H})_{z}-G(h) \, h_{z}}{t}.
\end{eqnarray*}
Now it is easy to see that the quotient in $F$ goes to $F'(h) \, \mathscr{H}=\frac{h}{\sqrt{1+h^{2}}} \, \mathscr{H}$ as $t \downarrow 0$ and in $L_\diese^\infty(\R_z)$. As concerning the
quotient in $G,$ by rewriting it as 
\begin{equation*}
h_{z} \, \frac{G(h+t \, \mathscr{H})-G(h)}{t}+G(h+t \, \mathscr{H}) \, \mathscr{H}_{z},
\end{equation*}
we see that its limit is $h_{z} \, G'(h) \, \mathscr{H}+G(h) \, \mathscr{H}_{z}=\left\{G(h) \, \mathscr{H}\right\}_z$. Finally in the limit $t \downarrow 0$, we get 
\begin{equation*}
\dif_{c,h}\phi \left(\lambda,(c,h)\right) \cdot (\mathscr{C},\mathscr{H})=-\mathscr{C}+\frac{\mathscr{R} \, h}{\sqrt{1+h^{2}}} \, \mathscr{H}-\lambda \left\{
\frac{\mathscr{H}}{1+h^{2}}\right\}_z.
\end{equation*}
This is the desired formula with $\widetilde{\mathscr{R}}$ and $\widetilde{\mathscr{H}}$ defined as in the lemma.
\end{proof}

Our next lemma will serve to verify Condition \eqref{cond-isom} of Theorem \ref{thm-implicit}.
%

\begin{lemma}
\label{Lemma3-C-monotone}
Assume \eqref{hyp-geometric} and let $\lambda_0 >0$, $c_0=c(\lambda_0)$ and $h_0(z)=h(z,\lambda_0)$ be as defined by \eqref{correction-def-c}. Then $(c_0,h_0) \in E_2$ and for any $f \in F$, there is a unique $(\mathscr{C},\mathscr{H}) \in E_2$ such that
\begin{equation*}
\dif_{c,h} \phi \left(c_0,(\lambda_0,h_0) \right) \cdot (\mathscr{C},\mathscr{H})=f.
\end{equation*}
Moreover $\mathscr{C}$ is given by: 
\begin{equation}
\mathscr{C}=-\frac{\int_{0}^{1} f(z ) \exp \left(\frac{1}{\lambda_0}\int_{z}^{1}
\widetilde{\mathscr{R}}_0(t) \, \dif t\right) \dif z}{
\int_{0}^{1}\exp \left(\frac{1}{\lambda_0}\int_{z}^{1}\widetilde{\mathscr{R}}_0
(t) \, \dif t \right) \dif z}  \label{EqndeC-Cmonotone}
\end{equation}
where $\widetilde{\mathscr{R}}_0=\mathscr{R} \, h_0 \sqrt{1+h_0^{2}}$.
\end{lemma}

Note that one can also give an explicit formula for $\mathscr{H}$ even if we will not need it in our case. 

\begin{proof}
It is clear that $(c_0,h_0) \in E_2$ as it satisfies \eqref{Eqn1-Cmonotone}. Let now $f \in F=L^\infty_\diese(\R_z)$ and consider the problem of finding $(\mathscr{C},\mathscr{H}) \in E_2=\R_c \times E$ such that 
$$
\dif_{c,h} \phi \left(c_0,(\lambda_0,h_0) \right) \cdot (\mathscr{C},\mathscr{H})=f.
$$ 

Assume first that such a pair exists and let us show that it is entirely determined by some explicit formulas. Recall that $\mathscr{H}$ will belong to $E$ (see \eqref{correction-def-E}) so that $\mathscr{H} \in W^{1,\infty}_\diese(\R_z)$ and
\begin{equation}\label{cond-mean}
\int_0^1 \mathscr{H}=0.
\end{equation}
By Lemma \ref{Lemma2-C-monotone} we also have that
\begin{equation*}
-\mathscr{C}+\widetilde{\mathscr{R}}_0 \, \widetilde{\mathscr{H}}-\lambda_0 \, \widetilde{\mathscr{H}}_{z}=f \quad \text{almost everywhere in} \quad \R_z,
\end{equation*}
where $\widetilde{\mathscr{R}}_0=\mathscr{R} \, h_0 \sqrt{1+h_0^{2}}$ and $\widetilde{\mathscr{H}}=\frac{\mathscr{H}}{1+h_0^2}$. Note that $\widetilde{\mathscr{H}}$ is Lipschitz because so is $\mathscr{H}$. By the variation of the constant method, $\widetilde{\mathscr{H}}$ is necessarily of the form:   
\begin{equation} \label{Eqn3-Cmonotone}
\widetilde{\mathscr{H}}(z)=C \exp \left(\frac{1}{\lambda_0 }\int_{0}^{z}\widetilde{\mathscr{R}}_0
(t) \, \dif t\right)-\frac{1}{\lambda_0 }\int_{0}^{z} \left(f(t)+\mathscr{C}\right) \exp \left(\frac{1}{\lambda_0}\int_{t}^{z}\widetilde{\mathscr{R}}_0(s) \, \dif s\right) \dif t,
\end{equation}
for some constant $C$. Since $\mathscr{H}$ is $1$-periodic, so is $\widetilde{\mathscr{H}}$ and $\widetilde{\mathscr{H}}(0)=\widetilde{\mathscr{H}}(1)$. This leads to   
\begin{equation}\label{last-eq}
C=C \exp \left(\frac{1}{\lambda_0} \int_{0}^{1} \widetilde{\mathscr{R}}_0(t) \, \dif t\right)-\frac{1}{\lambda_0} \int_{0}^{1} \left(f(t)+\mathscr{C}\right) \exp \left(
\frac{1}{\lambda_0}\int_{t}^{1}\widetilde{\mathscr{R}}_0(s) \, \dif s \right) \dif t.  
\end{equation}
To continue we claim that:
\begin{lemma}\label{lemme-tech-monotonie}
We have
$
\int_0^1 \widetilde{\mathscr{R}}_0 =0.
$ 
\end{lemma}
\noindent Indeed $\widetilde{\mathscr{R}}_0=\mathscr{R} \, h_0 \sqrt{1+h_0^{2}}=c_0 \, h_0+\lambda_0 \frac{h_0}{1+h_0^{2}} \, (h_0)_{z}$
by the ODE in \eqref{Eqn1-Cmonotone}, so that 
\begin{equation*}
\int_{0}^{1} \widetilde{\mathscr{R}}_0=c_0 \int_{0}^{1}  h_0+\frac{\lambda_0}{2} \int_{0}^{1}\left\{\ln (1+h_0^{2}) \right\}_z=0
\end{equation*}
(due to the two last conditions in \eqref{Eqn1-Cmonotone}). This completes the proof of the intermediate lemma. Injecting it into \eqref{last-eq} therefore leads to 
\begin{equation*}
\mathscr{C}\int_{0}^{1} \exp \left(
\frac{1}{\lambda_0}\int_{t}^{1}\widetilde{\mathscr{R}}_0(s) \, \dif s \right) \dif t=-\int_{0}^{1} f(t) \exp \left(
\frac{1}{\lambda_0}\int_{t}^{1}\widetilde{\mathscr{R}}_0(s) \, \dif s \right) \dif t.
\end{equation*}
This implies that $\mathscr{C}$ is uniquely determined by the desired formula \eqref{EqndeC-Cmonotone}. To get the uniqueness of $\mathscr{H}$, we rewrite \eqref{cond-mean} as $\int_0^1 \widetilde{\mathscr{H}} \, (1+h_0^2)=0$ and inject this into \eqref{Eqn3-Cmonotone}. We find that $C$ is uniquely determined and so are therefore $\widetilde{\mathscr{H}}$ and $\mathscr{H}=\widetilde{\mathscr{H}} \, (1+h_0^2)$. This completes the proof of the uniqueness of $(\mathscr{C},\mathscr{H})$. 

Conversely, if we take $(\mathscr{C},\mathscr{H})$ defined by the preceding formulas, the same arguments allow to show that $(\mathscr{C},\mathscr{H}) \in E_2$ and
$\dif_{c,h} \phi \left(c_0,(\lambda_0,h_0) \right) \cdot (\mathscr{C},\mathscr{H})=f.
$ This proves the existence of the pair $(\mathscr{C},\mathscr{H})$ and completes the proof of the lemma.
\end{proof}   

We can now apply the implicit function theorem to get the result below.

\begin{lemma}
Assume \eqref{hyp-geometric}. Then the map $\lambda \mapsto c(\lambda)$ is $C^\infty$ in $(0,+\infty)$. Moreover for all $\lambda _{0}>0,$%
\begin{equation}
c'(\lambda _{0})=-\frac{\int_{0}^{1}\frac{(h_{0})_{z}}{1+h_{0}^{2}
}(z) \exp \left(\frac{1}{\lambda_0}\int_{z}^{1} \widetilde{\mathscr{R}}_{0}
(t) \, \dif t \right) \dif z}{\int_{0}^{1} \exp \left(\frac{1}{\lambda_0} \int_{z}^{1} \widetilde{\mathscr{R}}_{0}(t) \, \dif t\right) \dif z}  \label{Eqn-Cprime-Cmonotone}
\end{equation}
where $h_{0}(z)=h(z,\lambda _{0})$ and $\widetilde{\mathscr{R}}_{0}=\mathscr{R} \, h_{0} \sqrt{1+h_{0}^{2}}.$
\end{lemma}

\begin{proof}
Let us consider the open set $\mathcal{O}:=\mathbb{R}_{\lambda }^{+} \times E_2 \subset E_{1}\times E_{2}$, $c_0=c(\lambda_0)$, and let us apply Theorem \ref{thm-implicit} to $\phi:\mathcal{O} \subset E_1 \times E_2 \to F$ at the point $\left(\lambda_0,(c_0,h_0)\right) \in \mathcal{O}$. Let us recall that $(c_0,h_0)$ satisfies \eqref{Eqn1-Cmonotone}. Hence $\phi \left(\lambda_0,(c_0,h_0) \right)=0$ in $F=L^\infty_\diese(\R_z)$ since by the choice of $\phi$ in \eqref{def-phi}, this equality is equivalent to the ODE in \eqref{Eqn1-Cmonotone}. We also know that $\dif_{c,h} \phi \left(\lambda_0,(c_0,h_0) \right) \in \mbox{Isom} (E_2,F)$ by Lemma \ref{Lemma3-C-monotone}. Theorem \ref{thm-implicit} then gives us the existence of the implicit function $\varphi:U \to V$, with $U$ an open set containing $\lambda_0$. In our case, $\varphi$ is $C^\infty$ because so is $\phi$, see for instance \cite{Car67}.

To continue, we claim that for all $\lambda \in U$, $\varphi(\lambda)=\left(c(\lambda),h(\cdot,\lambda)\right)$. This is a consequence of Theorem \ref{thm-implicit}\ref{formula-implicit}. Indeed, the pair $(c,h):=\varphi(\lambda) \in V \subset E_1 \times E_2$ satisfies $\phi \left(\lambda,(c,h)\right)=0$ in $F$. This means that $h$ solves the ODE in \eqref{Eqn1-Cmonotone}. Moreover, since $h$ belongs to $E$ ($E$ defined by \eqref{correction-def-E}) we have that: $h \in W^{1,\infty}(\R_z)$, $h$ is $1$-periodic and $\int_0^1 h=0$. Hence $(c,h)$ satisfies all the conditions in \eqref{Eqn1-Cmonotone}, which completes the proof of the claim by the uniqueness of such a pair.

Having verified the preceding claim, we can define the speed $c(\lambda)$ through the implicit function $\varphi$. To do so, we consider the projection
\begin{equation*}
\Pi _{c}:(c,h)\in E_{2}=\R_c \times E \mapsto c\in \mathbb{R},
\end{equation*}
which gives us that $c(\lambda )=\Pi _{c} \left(\varphi (\lambda)\right)$ for all $\lambda \in V$. This projection is $C^\infty$ as a bounded and linear map. Thus the composition $c(\cdot)=\left(\Pi_c \circ \varphi\right)(\cdot)$ is $C^{\infty}$ in $U$. We have shown that $\lambda \mapsto c(\lambda)$ is $C^\infty$ in some neighborhood $U$ of $\lambda_0$. The regularity then holds on all $\R^+$, because $\lambda _{0}>0$ is arbitrarily taken. 

It remains to show \eqref{Eqn-Cprime-Cmonotone}. By the chain rule for differentials, see \cite{Car67}, we have that $\dif \left(\Pi_c \circ \varphi\right)(\lambda_0)=\dif \Pi_c \left(\varphi(\lambda_0)\right) \circ \dif \varphi(\lambda_0) \in \mathcal{L}(\R_c,\R_\lambda)$. Since the variable $\lambda$ is real, $\dif \left(\Pi_c \circ \varphi\right)(\lambda_0) \cdot 1=\left(\Pi_c \circ \varphi\right)'(\lambda_0)$ and we get
\begin{equation*}
c'(\lambda _{0})=\left(\Pi_c \circ \varphi\right)'(\lambda_0)=\dif \Pi _{c} \left(\varphi (\lambda
_{0})\right) \cdot \left(d\varphi (\lambda _{0}) \cdot 1\right).
\end{equation*}
As $\Pi _{c}$ is linear continuous, $\dif \Pi_c=\Pi_c$ everywhere in $E_2$. Hence
\begin{equation*}
c'(\lambda _{0})=\Pi _{c} \left(\dif \varphi (\lambda _{0}) \cdot 1\right)
\end{equation*}
which, thanks toTheorem \ref{thm-implicit}\ref{formula-byproduct}, gives 
\begin{equation*}
c'(\lambda _{0})=-\Pi _{c}\left\{\left[ \dif_{c,h}\phi \left(\lambda
_{0},(c_{0},h_{0})\right)\right]^{-1}\cdot \left(\dif_{\lambda } \phi \left(\lambda
_{0},(c_{0},h_{0}) \right) \cdot 1\right) \right\}.
\end{equation*}
If we denote by $f$ the function $\dif_{\lambda } \phi \left(\lambda
_{0},(c_{0},h_{0}) \right) \cdot 1 \in F$, the above formula means that $c'(\lambda_0)=-\mathscr{C}$ where $\mathscr{C}$ is the speed of the unique pair $(\mathscr{C},\mathscr{H}) \in E_2=\R_c \times E$ solution of
$$
\dif_{c,h}\phi \left(\lambda
_{0},(c_{0},h_{0})\right) \cdot (\mathscr{C},\mathscr{H})=f.
$$ 
By \eqref{EqndeC-Cmonotone}, we then deduce that
$$
c'(\lambda_0)=\frac{\int_{0}^{1} f(z ) \exp \left(\frac{1}{\lambda }\int_{z}^{1}
\widetilde{\mathscr{R}}_0(t) \, \dif t\right) \dif z}{
\int_{0}^{1}\exp \left(\frac{1}{\lambda }\int_{z}^{1}\widetilde{\mathscr{R}}_0
(t) \, \dif t \right) \dif z}.
$$
It only remains to compute $f=\dif_{\lambda } \phi \left(\lambda
_{0},(c_{0},h_{0}) \right) \cdot 1$. Here again the variable $\lambda$ is real and $f$ is the usual partial derivative: 
$$
f=\phi_\lambda' \left(\lambda
_{0},(c_{0},h_{0}) \right)=\lim_{\lambda \to \lambda_0} \frac{\phi \left(\lambda,(c_{0},h_{0}) \right)-\phi \left(\lambda
_{0},(c_{0},h_{0}) \right)}{\lambda-\lambda_0}
$$
(the limit being in $F=L_\diese^\infty(\R_z)$). Recalling the definition of $\phi$ given by \eqref{def-phi}, we end up with
$
f=-\frac{(h_{0})_{z}}{1+h_{0}^{2}}
$
%
%
%
%
%
and the proof is complete.
\end{proof}

\begin{remark}\label{rem-continuum-corrector}
We claim that the mapping
$$
\lambda \in \R^+ \mapsto h(\cdot,\lambda) \in W^{1,\infty}_\diese(\R_z)
$$
is also $C^\infty$. Indeed we have seen above that $\varphi(\lambda)=\left(c(\lambda),h(\cdot,\lambda)\right)$, so that we can copy the arguments used to get the regularity of the speed by considering this time the projection $
\Pi_h:(c,h) \mapsto h$.
\end{remark}

We are now ready to prove Theorem \ref{Theorem-C-monotone}.

\begin{proof}[Proof of Theorem \ref{Theorem-C-monotone}]
It remains to prove that $c'(\lambda _{0})$ is nonpositive, for any $\lambda_0>0$, and that $\lim_{\lambda \to +\infty} c(\lambda)=\int_0^1 \mathscr{R} (z) \, \dif z$. 

Let us start with the first claim. By \eqref{Eqn-Cprime-Cmonotone}, $c'(\lambda_0)$ has the same sign as the numerator term
\begin{equation*}
J:=-\int_{0}^{1}\frac{(h_{0})_{z}}{1+h_{0}^{2}
}(z) \exp \left(\frac{1}{\lambda_0}\int_{z}^{1} \widetilde{\mathscr{R}_{0}}
(t) \, \dif t \right) \dif z.
\end{equation*}
An integration by parts gives
\begin{equation*}
\begin{split}
J & = \frac{1}{\lambda _{0}}\int_{0}^{1}\arctan (h_{0}) \, \widetilde{\mathscr{R}_{0}}(z) \exp \left(\frac{1}{\lambda _{0}}\int_{z}^{1}\widetilde{\mathscr{R}_{0}}(t) \, \dif t \right) \dif z  \\
& \quad +\left[ \arctan(h_{0}) \exp \left(\frac{1}{\lambda _{0}}\int_{z}^{1}\widetilde{\mathscr{R}_{0}}(t) \, \dif t \right)\right] _{z
=0}^{z=1}.
\end{split}
\end{equation*}
The second term vanishes since $h_0$ is $1$-periodic and
$\int_{0}^{1}\widetilde{\mathscr{R}_{0}} =0$ by Lemma \ref{lemme-tech-monotonie}. Recalling that $\widetilde{\mathscr{R}_{0}}=\mathscr{R} \, h_0 \sqrt{1+h_0^2}$, we obtain
\begin{equation*}
J=\frac{1}{\lambda _{0}} \int_{0}^{1} \left(\mathscr{R} \, h_{0} \arctan (h_{0})\sqrt{1+h_{0}^{2}}\right)(z)\exp \left(\frac{1}{\lambda_0}\int_{z}^{1}
\widetilde{\mathcal{R}}_{0}(t) \, \dif t \right) \dif z.
\end{equation*}%
Now since $\arctan (h_{0})$ has the same sign as $h_{0}$ and $\mathscr{R} \geq 0$ by \eqref{hyp-geometric}, we have $J\geq 0.$ This proves that $\lambda \mapsto c(\lambda)$ is nonincreasing. 

Let us immediately show that this map is decreasing if $\mathscr{R}$ is not constant. In that case, the ODE in \eqref{Eqn1-Cmonotone} implies that $h_0$ is not identically equal to zero. Hence $J>0$---since $\mathscr{R}>0$ by \eqref{hyp-geometric}---and this completes the proof that $c'(\lambda_0)<0$ if $\mathscr{R}$ is not trivial. 

Let us end with the limit of $c(\lambda)$ as $\lambda \to +\infty$. Recall that 
\begin{equation}\label{formula-speed-h}
c(\lambda)=\int_0^1 \mathscr{R}(z) \sqrt{1+h^2(z,\lambda)} \, \dif z,
\end{equation}
after integrating the equation in \eqref{Eqn1-Cmonotone}. Recall now that by  \eqref{hyp-geometric}, we have $\mathscr{R}_M \geq \mathscr{R} \geq \mathscr{R}_m>0$ for some constants. Then
$
c \leq \mathscr{R}_M
$
and
$$
\max_{z \in \R} |h(z,\lambda)| \leq {\sqrt{\frac{\mathscr{R}_M^2}{\mathscr{R}_m^2}-1}}<+\infty
$$
by Theorem \ref{front-well-posed}. Using again the ODE in \eqref{Eqn1-Cmonotone},
$$
\max_{z \in \R} |h_z(z,\lambda)| \leq \frac{C}{\lambda}
$$
for some $C$ not depending on $\lambda$. Since $h(\cdot,\lambda)$ is $1$-periodic with $\int_0^1 h(z,\lambda) \, \dif z=0$, an integration gives 
$$
\max_{z \in \R} |h(z,\lambda)| \leq \frac{C}{\lambda}.
$$
This proves that $h(\cdot,\lambda) $ uniformly converges towars zero on $\R$, as $\lambda \to +\infty$. Then the fact that $\lim_{\lambda \to +\infty} c(\lambda)=\int_0^1 \mathscr{R}(z) \, \dif z$ is obvious from \eqref{formula-speed-h}. The proof of the theorem is now complete.
\end{proof}

\begin{remark}\label{continuum-corrector-limit}
We have also proved that $h(\cdot,\lambda) \to 0$ in $W^{1,\infty}(\R)$ as $\lambda \to +\infty$.
\end{remark}

\section{Asymptotic expansion of the front's profile}\label{sec-expansion} 
  
From Section \ref{sec-homogenization}, we know that at a macroscopic level, the profile of the front,~$v^\varepsilon$, behaves like $v^0$ which is a constant (normalized to zero). In this section, we propose to analyze its microscopic oscillations by looking at the corrector $w=w(z)$ given by Theorem \ref{Theorem-homogenization-lambda}\ref{corrector}. So let us consider for any $\lambda \in (0,+\infty)$, the unique 
$w \in W^{2,\infty}(\R)$ satisfying
\begin{equation}\label{corrector-fixed}
\begin{cases}
-c^0+R \left(z,\frac{\overline{g}}{\overline{b}}\right) \sqrt{1+w_z^{2}}=\lambda \, \frac{w_{zz}}{1+w_{z}^{2}},\\[1ex]
w(z)=w(z+1),\\[1ex] 
\int_{0}^{1} w(t) \, \dif t=0,
\end{cases}
\end{equation}
for almost every $z \in \R$.
%

\begin{theorem}\label{thm-expansion}
Let the assumptions of Theorem \ref{Theorem-homogenization-lambda} hold. Then for all $y \in \R$ and $\varepsilon >0$, 
\begin{equation*}
v^\varepsilon(y)=
\begin{cases}
\co(\varepsilon) & \quad \mbox{if} \quad \lambda=+\infty,\\
\varepsilon \, w\left(\frac{y}{\varepsilon}\right)+\co (\varepsilon) & \quad \mbox{if} \quad \lambda \in (0,+\infty),
\end{cases}
\end{equation*}
where $\frac{\co(\varepsilon)}{\varepsilon} \to 0$ in $L^{\infty}(\R)$ as $\varepsilon \downarrow 0$ and with $w$ given by \eqref{corrector-fixed}.
\end{theorem}

\begin{remark}
If $\lambda<+\infty$ and if $R \left(\cdot,\frac{\overline{g}}{\overline{b}} \right)$ is not constant, then $w$ is not trivial. Note also that the more difficult case $\lambda=0$ will be discussed in Section \ref{conclusion}. 
\end{remark}

\begin{proof}
Let us begin with the proof for $\lambda=+\infty$. By \eqref{error-estimate}, 
$$
\|v^\varepsilon_y\|_\infty=\cO \left(\frac{\varepsilon}{\mu(\varepsilon)}\right)={\co}_\varepsilon(1)
$$
(since $\frac{\mu(\varepsilon)}{\varepsilon} \to \lambda=+\infty$). Since $v^\varepsilon$ is $\varepsilon$-periodic with a zero mean value, 
$$
\|v^\varepsilon\|_\infty \leq \varepsilon \, \|v^\varepsilon_y\|_\infty=\co (\varepsilon). 
$$
This completes the proof in that case.

Let us now consider the case $\lambda \in (0,+\infty)$ and set  $w^\varepsilon(z)=\frac{v^\varepsilon(\varepsilon \, z)}{\varepsilon}$ (as in \eqref{ansatz}). We claim that $w^\varepsilon$ converges to $w$ as $\varepsilon \downarrow 0$ stronlgly in $W^{1,\infty}(\R)$. We have in fact established this claim during the proof of Theorem \ref{Theorem-homogenization-lambda} but only along some sequence $\varepsilon_n \to 0$. But, applying this reasoning to any such sequence, as before, we get again the convergence for the whole family, because here also the limit is always the same, that is $w$ defined by \eqref{corrector-fixed}. We thus have:
$$
\max_{y} \Big| \frac{\dif}{\dif y} \left\{v^{\varepsilon }(y)- \varepsilon \, w \left(\frac{y}{
\varepsilon}\right) \right\}\Big|=\max_{y} \Big|v_y^{\varepsilon }(y)- w_z \left(\frac{y}{
\varepsilon}\right)\Big| ={\co}_\varepsilon(1).
$$
Using as above the periodicity of $v^\varepsilon$ and $w$ and the fact that they have zero mean values, 
$$
\max_{y} \Big| v^{\varepsilon }(y)- \varepsilon \, w \left(\frac{y}{
\varepsilon }\right) \Big| \leq \varepsilon \, \max_{y} \Big|  \frac{\dif}{\dif y} \left\{v^{\varepsilon }(y)- \varepsilon \, w \left(\frac{y}{
\varepsilon}\right) \right\} \Big|=\co (\varepsilon)
$$
and the proof is complete.
\end{proof}

Just as for the speeds, here too we can see that there is a smooth one-to-one correspondence between the correctors $w$ and the curvature regimes $\lambda$. More precisely, let us consider the mapping 
$$
\lambda \in \R^+ \mapsto w(\cdot)=w(\cdot,\lambda) \in W^{2,\infty}_\diese(\R_z),
$$ 
defined by \eqref{corrector-fixed}. Then:

\begin{theorem}\label{continuum-corrector}
Under the assumptions of Theorem \ref{Theorem-homogenization-lambda}, $\lambda \mapsto w(\cdot,\lambda)$ is $C^\infty$, injective, and satisfies: 
$$
\lim_{\lambda \to +\infty} w(\cdot,\lambda)=0 \quad \mbox{in} \quad W^{2,\infty}(\R).
$$
\end{theorem}

\begin{remark}
This is a version of Theorem \ref{Theorem-C-monotone} for the correctors. The limit at $\lambda=0$ will be discussed in the concluding remarks.
\end{remark}

\begin{proof}
Recall that $w_z(\cdot,\lambda)=h(\cdot,\lambda)$ in \eqref{Eqn1-Cmonotone} and use Remark \ref{rem-continuum-corrector} to get the $C^\infty$ regularity. For the injectivity, use that $\lambda \mapsto c^0=c^0(\lambda)$ is injective by Theorem \ref{Theorem-C-monotone}. For the limit at $+\infty$, use Remark \ref{continuum-corrector-limit}.
\end{proof}

Our last result considers a fixed $\mu$ as in Theorem \ref{Theorem-homogenization}.  In that case, we can identify the second order term of the expansion. To get an expansion in $L^\infty$, we will need to assume in addition that:
\begin{equation}\label{hyp-arrh-6}
\mbox{The function $T \in \R^+ \mapsto R \left(\cdot,T\right) \in L^\infty(\R)$ is continuous at $T=\frac{\overline{g}}{\overline{b}}$.}
\end{equation}

\begin{theorem}\label{thm-expansion-bis}
Let the assumptions of Theorem \ref{Theorem-homogenization} hold. Then for all $y \in \R$ and  $\varepsilon >0$, we have
\begin{equation*}
v^\varepsilon(y)=\varepsilon^2 \, Q \left(\frac{y}{\varepsilon}\right) +\co (\varepsilon^2),
\end{equation*}
where $Q \in W^{2,\infty}(\R_z)$ is defined as
$$
\begin{cases}
Q_{zz}(z)=\frac{1}{\mu} \left\{R\left(z,\frac{\overline{g}}{\overline{b}}\right)-\int_0^1 R\left(t,\frac{\overline{g}}{\overline{b}}\right) \dif t\right\},\\
Q(z+1)=Q(z),\\
\mbox{and } \int_0^1 Q=0,
\end{cases}
$$
and where $\lim_{\varepsilon \downarrow 0} \frac{\co (\varepsilon^2)}{\varepsilon^2}=0$ in $L_{\loc}^p(\R)$ for any $p \in [1,+\infty)$. If in addition \eqref{hyp-arrh-6} holds then the latter limit holds in $L^\infty(\R)$.
\end{theorem}

\begin{remark}
\begin{enumerate}[label={{\rm (\roman*)}}]
\item The limit in $L^p_\loc(\R)$ has to be understood as follows: For any fixed $p \in [1,+\infty)$ and $r>0$,
$$
\frac{1}{\varepsilon^2} \left(\int_{-r}^r |{\co}(\varepsilon^2)|^p \, \dif y\right)^{\frac{1}{p}} \to 0
$$
(the limit being uniform neither in $p$ nor in $r$).
\item The additional assumption \eqref{hyp-arrh-6} is satisfied by the combustion rate in \eqref{combustion-type} (provided that $A$ and $E$ are bounded).
\item The profile $Q$ is not trivial if $R\left(\cdot,\frac{\overline{g}}{\overline{b}}\right)$ is not a constant.
\end{enumerate}
\end{remark}

\begin{proof}
Note that $Q$ is well-defined, because the function
$$
z \mapsto R\left(z,\frac{\overline{g}}{\overline{b}}\right)-\int_0^1 R\left(t,\frac{\overline{g}}{\overline{b}}\right) \dif t
$$ 
is $1$-periodic with a zero mean value. Let us divide the rest of the proof in two cases.

\medskip

{\bf 1.} {\it Expansion in $L^p_{\loc}$ for $p \neq +\infty$.} Let us use again $w^{\varepsilon}(z)=\frac{v^{\varepsilon}(\varepsilon \, z)}{\varepsilon}$ (as in \eqref{ansatz}). We can rewrite its equation \eqref{Homog-eqn-weps} as
$$
\frac{\mu}{\varepsilon} \, w_{zz}^{\varepsilon} =  F_\varepsilon(z):=-c^\varepsilon \left(1+(w_{z}^{\varepsilon})^2\right) +R\left(z,u^{\varepsilon}(v^{\varepsilon}(\varepsilon
\, z),\varepsilon \, z)\right) \left(1+(w_{z}^{\varepsilon})^2\right)^{\frac{3}{2}}.
$$
We claim that for any $p \in [0,+\infty)$
\begin{equation}\label{claim-expansion}
F_\varepsilon(\cdot) \to  R\left(\cdot,\frac{\overline{g}}{\overline{b}}\right)-\int_0^1 R\left(t,\frac{\overline{g}}{\overline{b}}\right) \dif t \quad \mbox{in} \quad L^p(0,1),
\end{equation}
as $\varepsilon \downarrow 0$.
To show this claim, recall that $c^\varepsilon \to \int_0^1 R\left(t,\frac{\overline{g}}{\overline{b}}\right) \dif t$ and $\|w^\varepsilon_z\|_\infty \to 0$, by Theorem \ref{Theorem-homogenization} and \eqref{error-estimate}.
The case $p=1$ is thus a consequence of Lemma \ref{lem-corrector}\ref{conv-combustion-corrector} (whose result holds for the whole family $\{\varepsilon\}$, thanks to the same arguments). Since $R$ is bounded by $R_M$, this also implies the convergence for any $p \neq +\infty$ by interpolation.

With \eqref{claim-expansion} in hands, we can write that
\begin{equation*}
\begin{split}
& \frac{\dif^2}{\dif y^2} \left\{v^\varepsilon(y) - \varepsilon^2 \, Q \left(\frac{y}{\varepsilon}\right) \right\} \\
& = \frac{w^\varepsilon_{zz}\left(\frac{y}{\varepsilon}\right)}{\varepsilon}-Q_{zz} \left(\frac{y}{\varepsilon}\right)\\
& = \frac{1}{\mu} \, F_\varepsilon\left(\frac{y}{\varepsilon}\right)-\frac{1}{\mu} \left\{ R\left(\frac{y}{\varepsilon},\frac{\overline{g}}{\overline{b}}\right)-\int_0^1 R \left(t,\frac{\overline{g}}{\overline{b}}\right) \dif t \right\}
\end{split}
\end{equation*}
and conclude that
$$
\frac{1}{\varepsilon} \int_0^{\varepsilon} \left|\frac{\dif^2}{\dif y^2} \left\{v^\varepsilon(y) - \varepsilon^2 \, Q \left(\frac{y}{\varepsilon}\right)\right\}\right|^p \dif y={\co}_\varepsilon(1)
$$
(with ${\co}_\varepsilon(1)$ depending on $p$). Let us now apply the Poincar\'e-Wirtinger's inequality which states that $\int_0^\varepsilon |f|^p \leq C \, \varepsilon^p \, \int_0^\varepsilon |f_y|^p$ for any $\varepsilon$-periodic $f \in W^{1,p}(\R)$ with a zero mean value (and for some $C=C(p)$). Applying this two times, we get that
$$
\frac{1}{\varepsilon} \int_0^{\varepsilon} \left|v^\varepsilon(y) - \varepsilon^2 \, Q \left(\frac{y}{\varepsilon}\right)\right|^p \dif y=\varepsilon^{2 \, p} \, {\co}_\varepsilon(1).
$$
Given then $r>0$, we take $r_\varepsilon=\varepsilon \, \ulcorner \frac{r}{\varepsilon} \urcorner \geq r$ (this symbol denoting the upper integer part). We can then rewrite the mean value above over the larger period $2 \, r_\varepsilon$ as follows:
$
\frac{1}{\varepsilon} \int_0^{\varepsilon} \left|\dots\right|^p \dif y=\frac{1}{2 \, r_\varepsilon}  \int_{-r_\varepsilon}^{r_\varepsilon} \left|\dots\right|^p \dif y.
$
This easily implies the desired result.

\medskip

{\bf 2.} {\it Expansion in $L^\infty$.} Let us prove that \eqref{claim-expansion} holds in $L^\infty(\R)$ if in addition \eqref{hyp-arrh-6} holds. Set
$$
T_\varepsilon:=\min_t u^{\varepsilon}(v^{\varepsilon}(\varepsilon
\, t),\varepsilon \, t) \quad \mbox{and} \quad T^\varepsilon:=\max_t u^{\varepsilon}(v^{\varepsilon}(\varepsilon
\, t),\varepsilon \, t).
$$
Assumption \ref{hyp-arrh-1} then gives
$$
R(z,T_\varepsilon) \leq  R\left(z,u^{\varepsilon}(v^{\varepsilon}(\varepsilon
\, z),\varepsilon \, z)\right) \leq R(z,T^\varepsilon)
$$
for almost every $z$.  Now recalling that $T_\varepsilon,T^\varepsilon \to \frac{\overline{g}}{\overline{b}}$ (again by Theorem \ref{Theorem-homogenization}), we have by the assumption \eqref{hyp-arrh-6} 
$$
R\left(\cdot,u^{\varepsilon}(v^{\varepsilon}(\varepsilon
\, \cdot),\varepsilon \, \cdot)\right) \to R \left(\cdot,\frac{\overline{g}}{\overline{b}}\right) \quad \mbox{in} \quad L^\infty(\R).
$$
This implies \eqref{claim-expansion} in $L^\infty(\R)$ (since $c^\varepsilon \to \int_0^1 R\left(t,\frac{\overline{g}}{\overline{b}}\right) \dif t$ and $\|w^\varepsilon_z\|_\infty \to 0$) and the rest of the proof is the same (applying Poincar\'e-Wirtinger for $p=+\infty$).
\end{proof}

\section{Concluding remarks}\label{conclusion}

Let us conclude by a synthesis on the propagation governed by the typical Arrhenius law in \eqref{combustion-type}, 
$$
R(y,T)=A(y) \, e^{-\frac{E(y)}{T}},
$$ 
with $R$ the combustion rate, $T$ the temperature, $A$ a prefactor and $E$ related to the activation energy. We have established the existence of a travelling wave solution ``speed-front-temperature'' provided that the period of the medium is small enough, see Theorem \ref{existence-small-Y}. For large periods, such waves still exist up to slightly modifying $R$ at some neighborhood of $T=0$,
see Theorem \ref{existence-tw}. During the homogenization of these waves as the period $\varepsilon$ tends to zero, we have allowed the curvature coefficient $\mu=\mu(\varepsilon)$ to depend on $\varepsilon$ too (this parameter being related to the surface tension effects). 
Then the limiting speed of propagation $c^0$ is entirely determined by the value of the curvature regime $\lambda=\lim_{\varepsilon \downarrow 0} \frac{\mu(\varepsilon)}{\varepsilon}$, see Theorem \ref{Theorem-homogenization-lambda} and Remark \ref{rem-adherence-value}.  This speed $c^0=c^0(\lambda)$ is decreasing in $\lambda$ and takes the following minimal and maximal values:
$$
c^0(+\infty)=\overline{A(\cdot) \, e^{-E(\cdot) \, \frac{\overline{b}}{\overline{g}}}} \quad \mbox{and} \quad c^0(0)=\esssup_z A(z) \, \, e^{-E(z) \, \frac{\overline{b}}{\overline{g}}},
$$
with $b$ the heat capacity, $g$ the heat release, $\overline{b}$ and $\overline{g}$ their mean values, etc., see Theorem \ref{Theorem-C-monotone} and Remark \ref{rem-adherence-value}. Here the constant $\frac{\overline{g}}{\overline{b}}>0$ is the limiting temperature at the front. Finally  the profile $v^\varepsilon$ of the heterogeneous flame front  satisfies:
\begin{equation}\label{expansion-last}
v^\varepsilon(y)=\varepsilon \, w \left(\frac{y}{\varepsilon}\right) +\co(\varepsilon),
\end{equation}
where $w=w(z)$ is a corrector entirely determined by $\lambda$, see Theorem \ref{thm-expansion}. At the macroscopic level, the front's profile is a straight line (normalized to zero without loss of generality). At the microscopic level, its oscillations are entirely described by $w$ which solves the geometric equation
\begin{equation}\label{corrector-last}
-c_0+A(z) \, e^{-E(z) \, \frac{\overline{b}}{\overline{g}}} \sqrt{1+w_z^2}=\lambda \, \frac{w_{zz}}{1+w_z^2}.
\end{equation}
If $\mu$ is fixed, then the profile $w$ is also a straight line and 
$$
v^\varepsilon(y)=\varepsilon^2 \, Q \left(\frac{y}{\varepsilon}\right) +\co(\varepsilon^2),
$$ 
where the profile $Q$ is given by:
$$
Q(z)=\frac{1}{\mu} \left\{P(z)-\overline{P}\right\}
$$ 
with
\begin{equation*}
P(z):=\int_0^z \int_0^{t} \left(A(s) \, e^{-E(s) \, \frac{\overline{b}}{\overline{g}}}-c^0(+\infty)\right)  \dif s \, \dif t,
\end{equation*}
see Theorem \ref{thm-expansion-bis}.

We end up with an open question. So far we have been able to give ansatz of the front's profile for all values of $\lambda$, including $+\infty$, but not for $\lambda=0$ where the expansion \eqref{expansion-last} is not clear. This question is related to the passage to the limit in \eqref{corrector-last} as $\lambda \downarrow 0$. By the result of \cite{ChNa97}, it is known that some sequence converges towards a viscosity solution of the Hamilton-Jacobi equation
\begin{equation*}
-c+A(z) \, e^{-E(z) \, \frac{\overline{b}}{\overline{g}}} \sqrt{1+w_z^{2}}=0
\end{equation*}
(with $c=c^0(\lambda=0)$). Unfortunately, the convergence of the whole family is not clear because this solution is not unique (even up to the addition of a constant). The problem is that $w_z^2$ is unique but not $w_z$. The identification of the first-order term in \eqref{expansion-last} for $\lambda=0$ is thus open and probably difficult (to the best of our knowledge).  

\appendix

\section{Technical proofs}

\subsection{Proof of the claim in Remark~\ref{var-equiv-dis}} \label{tech-app}

\begin{proof}[Proof that ``\eqref{var-eqn} $\Rightarrow$ \eqref{dis-eqn}'']
Let $\varphi \in C^1_c(\R^2)$ and define a partition of unity
$\{\theta_{i}\}_{i \in\mathbb{Z} }$ such that for all~$i \in\mathbb{Z} $,
\begin{equation*}
\begin{cases}
0 \leq\theta_{i} \in C_{c}^1(\mathbb{R} ),\\
\supp \theta_{i} \subset\left( \frac{i}{2},\frac{i+2}{2}\right) ,\\
\theta_{i+1}(\cdot)=\theta_{i}\left( \cdot-\frac{1}{2}\right) ,\\
\sum_{i \in\mathbb{Z} } \theta_{i}=1.
\end{cases}
\end{equation*}
Note that the sum at any value is locally taken only on two consecutive indices. Let us define
\begin{equation*}
\varphi_{i}(x,y):=\varphi(x,y) \, \theta_{i}(y),
\end{equation*}
so that~$\varphi=\sum_{i \in\mathbb{Z} } \varphi_{i}$. By linearity, it
suffices to prove~\eqref{dis-eqn} for each~$\varphi_{i}$. So let $i
\in\mathbb{Z}$ and consider $w \in C^1(\mathbb{R} ^{2})$ such that
$w=\varphi_{i}$ on~$\left\{ \frac{i}{2}<y<\frac{i+2}{2} \right\} $ and
extended to $y \in\mathbb{R} $ by periodicity. It is clear that~$w_{|_{\Omega
}} \in H^{1}_{\diese}(\Omega)$. In particular, it can be
chosen in~\eqref{var-eqn}, which exactly gives~\eqref{dis-eqn}.
\end{proof}

\begin{proof}[Proof that ``\eqref{dis-eqn} $\Rightarrow$ \eqref{var-eqn}'']
Conversely, let us consider~$w \in H^{1}_{\diese}(\Omega)$.
Extending it by reflexion if necessary, we can consider that~$w \in
H^{1}_{\diese}(\mathbb{R} ^{2})$. Let us define~$w_{i} \in
H^{1}(\mathbb{R} ^{2})$ by~
\begin{equation*}
w_{i}(x,y):=w(x,y) \, \theta_{i}(y).
\end{equation*}
By density of~$C^1_{c}(\mathbb{R} ^{2})$ in~$H^{1}(\mathbb{R} ^{2}%
)$,~\eqref{dis-eqn} holds true for each~$w_{i}$ whenever it is so for test
functions. Since $\supp w_{i} \subset\left( \frac{i}{2},\frac{i+2}{2}\right)
$, we get:
\begin{equation}
\label{dis-tech}%
\begin{split}
&  \int_{\Omega\cap\left\{ 0<y<\frac{3}{2} \right\} } \left( c \, b \, u_{x}
\, (w_{0}+w_{1}) + a \, \nabla u \, \nabla(w_{0}+w_{1}) \right) \\
&  =\int_{\Gamma\cap\left\{ 0<y<\frac{3}{2} \right\} } \frac{c \, g}%
{\sqrt{1+v_{y}^{2}}} \, (w_{0}+w_{1}).
\end{split}
\end{equation}
Let us rewrite these terms as integrals over $\left\{ 0<y<1\right\} $. If
$\frac{1}{2}<y<1$, then we use that $w=\sum_{i \in\mathbb{Z} } w_{i}%
=w_{0}+w_{1}$ (for such $y$). We get:
\begin{equation}
\label{dis-tech-bis}%
\begin{split}
&  \int_{\Omega\cap\left\{ \frac{1}{2}<y<1 \right\} } \left( c \, b \, u_{x}
\, (w_{0}+w_{1}) + a \, \nabla u \, \nabla(w_{0}+w_{1}) \right) \\
&  = \int_{\Omega\cap\left\{ \frac{1}{2}<y<1 \right\} } \left( c \, b \, u_{x}
\, w + a \, \nabla u \, \nabla w \right) .
\end{split}
\end{equation}
For the remaining $y$, we use that $u$ is $Y$-periodic and $w_{1}%
(\cdot)=w_{-1}\left( \cdot-1\right) $ to show that
\begin{equation*}%
\begin{split}
\int_{\Omega\cap\left\{ 1<y<\frac{3}{2} \right\} } \left( c \, b \, u_{x} \,
w_{1} + a \, \nabla u \, \nabla w_{1} \right) = \int_{\Omega\cap\left\{
0<y<\frac{1}{2} \right\} } \left( c \, b \, u_{x} \, w_{-1} + a \, \nabla u \,
\nabla w_{-1} \right) .
\end{split}
\end{equation*}
Using in addition that $w_{0}=0$ on $\left\{ 1<y<\frac{3}{2} \right\} $ and
$w_{1}=0$ on $\left\{ 0<y<\frac{1}{2} \right\} $, we deduce that
\begin{equation}
\label{dis-tech-bis-bis}%
\begin{split}
&  \int_{\Omega\cap\left( \left\{ 0<y<\frac{1}{2} \right\}  \cup\left\{
1<y<\frac{3}{2} \right\} \right) } \left( c \, b \, u_{x} \, (w_{0}+w_{1}) + a
\, \nabla u \, \nabla(w_{0}+w_{1}) \right) \\
&  = \int_{\Omega\cap\left\{ 0<y<\frac{1}{2} \right\} } \left( c \, b \, u_{x}
\, (w_{0}+w_{-1}) + a \, \nabla u \, \nabla(w_{0}+w_{-1}) \right) \\
&  = \int_{\Omega\cap\left\{ 0<y<\frac{1}{2} \right\} } \left( c \, b \, u_{x}
\, w + a \, \nabla u \, \nabla w \right) .
\end{split}
\end{equation}
The last line is obtained from similar arguments as for \eqref{dis-tech-bis}.
Adding \eqref{dis-tech-bis} and \eqref{dis-tech-bis-bis}, we conclude that
\begin{equation*}%
\begin{split}
&  \int_{\Omega\cap\left\{ 0<y<\frac{3}{2} \right\} } \left( c \, b \, u_{x}
\, (w_{0}+w_{1}) + a \, \nabla u \, \nabla(w_{0}+w_{1}) \right) \\
&  = \int_{\Omega\cap\left\{ 0<y<1 \right\} } \left( c \, b \, u_{x} \, w + a
\, \nabla u \, \nabla w \right) .
\end{split}
\end{equation*}
We show in the same way that
\begin{equation*}%
\int_{\Gamma\cap\left\{ 0<y<\frac{3}{2} \right\} } \frac{c \, g}%
{\sqrt{1+v_{y}^{2}}} \, (w_{0}+w_{1})
= \int_{\Gamma\cap\left\{ 0<y<1 \right\} } \frac{c \, g}{\sqrt{1+v_{y}^{2}%
}} \, w.
\end{equation*}
We thus complete the proof of~\eqref{var-eqn} from \eqref{dis-tech}.
\end{proof}

\subsection{Proof of Theorem~\ref{T-well-posed}: Well-posedness of the temperature}\label{app-T-well-posed}

We will apply
standard techniques from variational analysis \cite{Bre83}. The main difficulty is the lack of a Poincar\'e's inequality, which will be compensated by the explicit bounds in \eqref{Linfty-esti}. Let us give some details for completeness sake. 

Given $\alpha>0$, we introduce the following auxiliary
problem:
\begin{equation}
\label{pert-pb}
\begin{cases}
\mbox{find $u_\alpha \in H_{\diese}^1(\Omega)$ such that}\\
a_{\alpha}(u_{\alpha},w)=l(w) \quad\forall w \in H_{\diese}^{1}(\Omega),
\end{cases}
\end{equation}
where
\begin{align*}
a_{\alpha}(u_{\alpha},w)  &  :=  \frac{\alpha}{Y} \int_{\Omega_{\diese}}
u_{\alpha}\,w+\frac{1}{Y} \int_{\Omega_{\diese}} \left( c \, b \,
(u_{\alpha})_{x} \, w + a \, \nabla u_{\alpha}\, \nabla w \right) ,\\
l(w)  &  :=  \frac{1}{Y} \int_{\Gamma_{\diese}} \frac{c \, g}{\sqrt
{1+v_{y}^{2}}} \, w.
\end{align*}
Recall that the $L^2$-norm and $H^1$-semi-norm for periodic spaces have been defined by 
\begin{equation*}
\|w\|_{L_{\diese}^{2}(\Omega)}=\left(  \frac{1}{Y}%
\int_{\Omega_{\diese}}w^{2}\right)  ^{\frac{1}{2}}%
\quad \mbox{and} \quad |w|_{H_{\diese}^{1}(\Omega)}=\left(  \frac{1}{Y}%
\int_{\Omega_{\diese}}|\nabla w|^{2}\right)  ^{\frac{1}{2}%
}.
\end{equation*}
The proof of Theorem \ref{T-well-posed} will follow from the passage to the limit as $\alpha\downarrow0$. Let us first give a few lemmas. We start by a well-posedness result.

\begin{lemma}
\label{pert-pb-well-posed} 
Under the assumptions of Theorem \ref{T-well-posed}, Problem \eqref{pert-pb} admits
a unique solution, for each $\alpha>0$.
\end{lemma}

\begin{proof}
To use the Lax-Milgram's theorem, it suffices to check
the coercivity of $a_{\alpha}$ which follows from:
\begin{equation}
\label{coer-esti}\frac{1}{Y} \int_{\Omega_{\diese}} \left( c \, b \, w_{x}
\, w + a \, |\nabla w|^{2} \right)  = \underbrace{\frac{1}{Y} \int_{\Gamma
_{\diese}} c \, b \, \frac{w^{2}}{2}}_{\geq0} + \underbrace
{\frac{1}{Y} \int_{\Omega_{\diese}} a \, |\nabla w|^{2}}_{\geq a_{m} \,
|w|_{H^{1}_{\diese}(\Omega)}^2} \quad \forall w \in H^{1}_{\diese}(\Omega).
\qedhere
\end{equation}
\end{proof}


We proceed by a nonnegativity lemma.

\begin{lemma}
\label{nonnegativity} 
Assume the hypotheses of Theorem \ref{T-well-posed} and
let $\alpha
\geq0$ and $u$ be such that
\begin{equation*}%
\begin{cases}
u \in H_{\diese}^{1}(\Omega),\\
a_{\alpha}(u,w) \geq0 \quad\forall w \in H_{\diese}%
^{1}(\Omega) \mbox{ with } w \geq0.
\end{cases}
\end{equation*}
Then~$u \geq0$.
\end{lemma}

\begin{proof}
Let us take $w=u^{-}:=-\min\{u,0\}$ and show that $u^- = 0$. Since $\nabla u^{-}=-\mathbf{1}_{u<0} \, \nabla u$, we have:
\begin{equation*}
-\alpha\, \|u^{-}\|^{2}_{L^{2}_{\diese}(\Omega)}
-\frac{1}{Y}\int_{\Omega_{\diese}} a \, |\nabla u^{-}|^{2}\geq\underbrace{\frac{1}{Y}\int_{\Omega_{\diese}} c
\, b \, u^{-}_{x} \, u^{-}}_{=\frac{1}{Y}\int_{\Gamma_{\diese}} c \, b \,
\frac{(u^{-})^{2}}{2} \geq0}.
\end{equation*}
This completes the proof even for $\alpha=0$ since $u^-$ is (square) integrable as $x$ approaches $-\infty$.
\end{proof}

We finally give some estimates on $u_{\alpha}$.

\begin{lemma}
\label{esti-pert-pb} Let $\alpha>0$ and $u_{\alpha}$ be given by Lemma
\ref{pert-pb-well-posed}. Then
\begin{equation*}
|u_{\alpha}|_{H^{1}_{\diese}(\Omega)}^{2} \leq\frac{2 \, c \,
g_{M}^{2}}{a_{m} \, b_{m}}%
\end{equation*}
and for almost every $(x,y) \in\Omega$,
\begin{equation*}
\frac{g_{m} \, a_{m}}{a_{M} \, b_{M}}\, e^{c \, \frac{b_{M}}{a_{m}} \left(
x-\|v\|_{\infty}\right) }\leq u_{\alpha}(x,y) \leq\frac{g_{M} \, a_{M}}{a_{m}
\, b_{m}} \, e^{c \, \frac{b_{m}}{a_{M}} \left( x+\|v\|_{\infty}\right) }.
\end{equation*}

\end{lemma}

\begin{proof}
If we take~$w=u_{\alpha}$ in \eqref{coer-esti} and use the inequality
\begin{equation*}
l(u_{\alpha})=a_{\alpha}(u_{\alpha},u_{\alpha}) \geq \frac{1}{Y}\int_{\Omega
_{\diese}} \left( c \, b \, (u_{\alpha})_{x} \, u_{\alpha}+ a
\, |\nabla u_{\alpha}|^{2} \right) ,
\end{equation*}
we infer that
\begin{multline*}
\frac{c \, b_{m}}{2} \, \frac{1}{Y} \int_{0}^{Y} u_{\alpha}^{2}(v(y),y) \, \mathrm{d} y\leq
l(u_{\alpha})\\
\leq \frac{c \, g_{M}}{Y} \int_{0}^{Y}u_{\alpha}(v(y),y) \, \mathrm{d} y \leq c \, g_{M}
\left(\frac{1}{Y} \int_{0}^{Y} u_{\alpha}^{2}(v(y),y) \, \mathrm{d} y\right) ^{\frac{1}{2}},
\end{multline*}
so that $\left(\frac{1}{Y} \int_{0}^{Y} u_{\alpha}^{2}(v(y),y) \, \mathrm{d} y\right)
^{\frac{1}{2}} \leq\frac{2 \, g_{M}}{b_{m}}$ and $l(u_{\alpha}) \leq\frac{2 \,
c \, g_{M}^{2}}{b_{m}}$. Using again \eqref{coer-esti}, we get the first
desired estimate $|u_{\alpha}|_{H^{1}_{\diese}(\Omega)}^{2}
\leq\frac{a(u_{\alpha},u_{\alpha})}{a_{m}} = \frac{l(u_{\alpha})}{a_{m}}
\leq\frac{2 \, c \, g_{M}^{2}}{a_{m} \, b_{m}}.
$

For the remaining estimate, we look for sub- and supersolutions of the form
$\tilde{u}(x,y):=C_{1} \, e^{C_{2} \, x}$. Let us prove the upper bound (the
proof for the lower bound being the same). Taking~$C_{2}=\frac{c \, b_{m}%
}{a_{M}} \geq0$,~$C_{1}=\frac{c \, g_{M}}{a_{m}\, C_{2}} \, e^{C_{2} \,
\|v\|_{\infty}} \geq0$, easy computations show that
\begin{equation*}%
\begin{cases}
\alpha\, \tilde{u}+c \, b \, \tilde{u}_{x}-\Div \left(  a \, \nabla\tilde
{u}\right)  \geq0 & \quad\mbox{in} \quad \Omega,\\[1ex]
a \, \frac{\partial\tilde{u}}{\partial \nu} \geq\frac{c \, g}{\sqrt{1+v_{y}^{2}%
}} & \quad\mbox{on}  \quad \Gamma.
\end{cases}
\end{equation*}
In particular, $\tilde{u} \in C^{\infty}(\mathbb{R} ^{2}) \cap H^{1}%
_{\diese}(\Omega)$ satisfies the following variational
inequalities:
\begin{equation*}
a_{\alpha}(\tilde{u},w) \geq l(w) \quad\forall w \in H^{1}%
_{\diese}(\Omega) \mbox{ such that } w \geq0.
\end{equation*}
We complete the proof by applying Lemma~\ref{nonnegativity} to $u=\tilde
{u}-u_{\alpha}$.
\end{proof}

We can now prove the theorem.

\begin{proof}
[Proof of Theorem~\ref{T-well-posed}]The uniqueness is an immediate
consequence of Lemma~\ref{nonnegativity} (valid for $\alpha=0$). Moreover, by
Lemma \ref{esti-pert-pb}, $\{u_{\alpha}\}_{\alpha>0}$ is bounded in
$H^{1}_{\diese}(\Omega) $. Hence, by the weak-compactness
theorem, $u_{\alpha}$ weakly converges to some $u$ in $H^{1}%
_{\diese}(\Omega)$ as $\alpha \downarrow 0$ (and up to a subsequence). It is then standard
to pass the limit in \eqref{pert-pb} and get a solution of \eqref{T-eqn}. This
solution satisfies the desired estimates \eqref{H1-esti} and
\eqref{Linfty-esti}, since so does $u_{\alpha}$ (by Lemma \ref{esti-pert-pb}).
\end{proof}

\subsection[Proof of Theorem \ref{Holder-regularity}]{Proof of Theorem \ref{Holder-regularity}: H\"{o}lder regularity of the temperature}\label{app-Holder}

The proof is organized in several lemmas. The first one is a synthesis of the a priori estimates derived in Section \ref{sec-existence}. The rest is an application of the regularity theory for elliptic PDEs, see \cite{GiTr01,Nit11} and the references therein. Since we need an estimate independent of $\mu$, we will investigate the H\"older regularity for the Neuman problem \eqref{T-eqn} by assuming the boundary of the domain to be only Lipschitz. We can then use the results of \cite{Nit11}. However the latter reference, which deals with a fixed domain, does not mention how the estimates depend on the boundary. This point is unfortunately crucial for our free boundary problem. For that reason, we propose a complete proof following the arguments of \cite{Nit11} for the sake of completeness. 

Here are preliminary a priori estimates.

\begin{lemma}
\label{estimates-tw} 
Let the assumptions of Theorem \ref{Holder-regularity} hold. Then:
\begin{equation*}
c_m\leq c\leq c_M, \quad 
\|v_{y}\|_{\infty} \leq C \quad \mbox{and} \quad \|u\|_{\infty} \leq C,
\end{equation*}
for some positive constant $c_m$, $c_M$ and $C$ depending only on $Y_0$, $\lambda_0$, $R_0$, and the bounds $a_{m}$, $b_{m}$, $g_{m}$, $a_{M}$, $b_{M}$, $g_{M}$ and $R_{M}$.
\end{lemma}


\begin{proof}
We first claim that:
\begin{equation}\label{claimed-lower-bound-precise}
u \geq u_m \quad \mbox{at the front} \quad \Gamma=\{y=v(x)\},
\end{equation}
for some constant $u_m>0$ having the desired dependences (as stated in the lemma). The important assumption to show this claim is \eqref{hyp-both}. If the first condition holds, that is $\frac{\mu}{Y} \geq \lambda_0 >\frac{4 \, R_M}{\pi}$, we can apply Lemma \ref{new-lower-bound}. This gives us \eqref{claimed-lower-bound-precise} with $u_m$ depending only on $Y_0$, $\lambda_0$, and the bounds $a_m$, etc. In the case where the second condition holds in \eqref{hyp-both}, we have in particular that
$$
\lim_{T \downarrow 0} \left\{|\ln T| \, \essinf_y R(y,T) \right\} =+\infty.
$$
Applying then Lemma \ref{lower-bound}, we get also \eqref{claimed-lower-bound-precise} with now $u_m$ depending only on $Y_0$, $R_0$, and $a_m$, etc. This completes the proof of \eqref{claimed-lower-bound-precise}. 

Now we can use the assumptions \ref{hyp-arrh-1}, \ref{hyp-arrh-3} and \ref{hyp-arrh-4} to show that the combustion rate $y \mapsto R(y,u(v(y),y))$ in Problem \eqref{T-eqn}--\eqref{front-eqn} is positively bounded from below and above (by some constants having the desired dependences). The end of the proof is then an easy application of Theorem \ref{front-well-posed}\eqref{front-esti} followed by Theorem \ref{T-well-posed}\eqref{Linfty-esti}.
\end{proof}

Following \cite{Nit11}, the rest of the proof of Theorem \ref{Holder-regularity} consists in identifying an elliptic PDE for the extension of $u$ in order to apply the (local) De Giorgi-Nash-Moser's theorem. 

The following lemma identifies the equation satisfied by the extension of $u$ to the whole space.

\begin{lemma}\label{identification}
Let $(c,v,u)$ be as in the preceding lemma.
Let us define, for almost every $(x,y) \in \R^2$,
\begin{equation*}
\begin{split}
\tilde{b}(x,y) & :=
 b(y) \times
\begin{cases}
c, & \quad x<v(y),\\
-c, & \quad x>v(y),
\end{cases}\\
A(x,y) & :=a(y)  \times
\begin{cases}
\left(\begin{array}{cc}
1 & 0 \\
0 & 1
\end{array} \right), & \quad x<v(y),\\[3ex]
 \left( 
\begin{array}{cc}
1+4 \, v_y^2 & 2 \, v_y \\
2 \, v_y & 1
\end{array} 
\right), & \quad x>v(y),\\
\end{cases}\\
\tilde{g}(x,y) & := g(y) \times
\begin{cases}
-c, & \quad x<v(y),\\
c , & \quad x>v(y).
\end{cases}
\end{split}
\end{equation*}
Then $u \in H_{\loc}^1(\R^2)$ satisfies:
\begin{equation*}
\int_{\R^2} \left(\tilde{b} \, u_x \, \varphi+ \langle A \, \nabla u , \nabla \varphi \rangle \right)=-\int_{\R^2} \tilde{g} \, \varphi_x \quad \forall \varphi \in C^1_c(\R^2),
\end{equation*}
where $\langle \cdot,\cdot \rangle$ is the inner product.
\end{lemma}

\begin{remark}
In other words, the extension of $u$ is a variationnal solution of 
\begin{equation}\label{ext-eqn}
 \tilde{b} \, u_x-\Div (A \, \nabla u)= \tilde{g}_x \quad \mbox{in} \quad \R^2,
\end{equation}
with measurable and bounded coefficients $ \tilde{b}$, $A$ and $\tilde{g}$.
\end{remark}

\begin{proof}
Let $\phi:\R^2 \to \R^2$ be defined by
$$
\phi(x,y):=(2 \, v(y)-x,y).
$$
Then $\phi$ is a $C^1$ bijection, since $v$ is $W^{2,\infty}$, and $\phi^{-1}=\phi$. Moreover
\begin{equation*}
u \circ \phi=u
\end{equation*}
($u$ being extended to $\R^2$ by \eqref{extension}). Taking the gradient, it follows that
\begin{equation*}
\left(\textrm{Jac} \, \phi\right)^{\transpose} \nabla u \circ \phi=\nabla u, 
\end{equation*}
where 
$$
\textrm{Jac} \, \phi:=\left( 
\begin{array}{cc}
-1 & 2 \, v_y \\
0 & 1
\end{array}
\right)
$$ 
is the Jacobian matrix of $\phi$ and $\left(\textrm{Jac} \, \phi\right)^{\transpose}$ its transpose. 
But $\left(\textrm{Jac} \, \phi\right)^{-1}=\textrm{Jac} \, \phi$, since $\phi^{-1}=\phi$, and thus 
\begin{equation}\label{formula-gradient}
\nabla u \circ \phi =\left(\textrm{Jac} \, \phi\right)^{\transpose} \nabla u \quad \mbox{so that} \quad u_x \circ \phi=-u_x.
\end{equation}
Let us introduce a new test function
\begin{equation*}
\psi:=\varphi \circ \phi.
\end{equation*}
By the same computations as above, 
\begin{equation}\label{formula-gradient-bis}
\nabla \psi \circ \phi=\left(\textrm{Jac} \, \phi\right)^{\transpose} \nabla \varphi. 
\end{equation}
But $\psi \in C^1_c(\R^2)$ can be used as a test function in \eqref{dis-eqn}, that is
\begin{equation*}
\int_{x<v(y)} \left(c \, b \, u_x \, \psi+a \, \nabla u \, \nabla \psi \right) \underbrace{(x,y)}_{=\phi(\phi^{-1}(x,y))} \dif x \, \dif y=\int_{\R} c \, g(y) \, \psi(v(y),y) \, \dif y. 
\end{equation*}
Let us change the variable by $\phi^{-1}(x,y) \mapsto (x,y)$ 
in the first integral. We have the new element of integration $\left|\textrm{Jac} \, \phi\right| \dif x \, \dif y=\dif x \, \dif y$, since $\left|\textrm{Jac} \, \phi\right|=1$, and the new domain 
$
\phi^{-1} \left(\{x<v(y) \} \right)=\{x>v(y)\}.
$
Hence, 
\begin{equation*}
\int_{x>v(y)} \left(c \, b \, u_x+a \, \nabla u \, \nabla \psi \right)\left(\phi(x,y)\right)  \dif x \, \dif y=\int_\R c \, g(y) \, \varphi(v(y),y) \, \dif y,
\end{equation*}
where we have used in addition that $\psi=\varphi$ at the front $\{x=v(y)\}$ to rewrite the second integral. By \eqref{formula-gradient} and \eqref{formula-gradient-bis}, we conclude that
\begin{multline*}
\int_{x>v(y)} \left(- c \, b \, u_x+\langle a \left(\textrm{Jac} \, \phi\right)^{\transpose} \nabla u, \left(\textrm{Jac} \, \phi\right)^{\transpose} \nabla \varphi \rangle \right) (x,y) \, \dif x \, \dif y\\
=\int_\R c \, g(y) \, \varphi(v(y),y) \, \dif y=-\int_{x>v(y)} c \, g(y) \, \varphi_x(x,y) \, \dif x \, \dif y.
\end{multline*}
Using that $\left(\textrm{Jac} \, \phi \right) \left(\textrm{Jac} \, \phi\right)^{\transpose}=\left( 
\begin{array}{cc}
1+4 \, v_y^2 & 2 \, v_y \\
2 \, v_y & 1
\end{array} 
\right)$, we obtain:
$$
\int_{x>v(y)} \left( \tilde{b} \, u_x \, \varphi+ \langle A \, \nabla u , \nabla \varphi \rangle \right)=-\int_{x>v(y)} \tilde{g} \, \varphi_x,
$$
where $\tilde{b}$, $A$ and $\tilde{g}$ are defined as in the lemma. To complete the proof, it suffices to choose  $\varphi$ in \eqref{dis-eqn}, thus getting the remaining equality for $x<v(y)$, and then sum up the result with the equality above.
\end{proof}

With the preceding lemma in hands, the H\trema{o}lder estimate stated in Theorem \ref{Holder-regularity} will be a consequence of \cite[Theorem 8.24]{GiTr01}. We first need to check that the coefficient $\tilde{b}$, $A$ and $\tilde{g}$ satisfy the assumption of this theorem, see \cite[Pages 177--178]{GiTr01}. 

\begin{lemma}\label{checking-assumptions} Let $(c,v,u)$, $\tilde{b}$ and $A$ be as in the preceding lemmas. Then for almost every $(x,y) \in \R^2$ and all $\xi \in \R^2$,
\begin{enumerate}[label={{\rm (\roman*)}}]
\item $\langle A(x,y) \, \xi,\xi \rangle \geq \lambda \, |\xi|^2$,\label{ellipticity}
\item $|A(x,y)| \leq \Lambda$,
\item $\lambda^{-1} \, |\tilde{b}(x,y)| \leq \nu$,
\end{enumerate}
for some positive constants $\lambda$, $\Lambda$ and $\nu$ depending only on $Y_0$, $\lambda_0$, $R_0$, and the bounds $a_{m}$, $b_{m}$, $g_{m}$, $a_{M}$, $b_{M}$, $g_{M}$ and $R_{M}$.
\end{lemma}

\begin{proof}
All the items are easy to prove excepted may be \ref{ellipticity}. To show this property, it suffices to find a lower bound $\lambda>0$ of the eigenvalues of $A=A(x,y)$. If $x<v(y)$, we have $A=a(y) \left( 
\begin{array}{cc}
1 & 0 \\
0 & 1
\end{array} 
\right)$ and any constant $\lambda \leq a_m$ will work. Let us thus focus on the other case where
$
A= \left( 
\begin{array}{cc}
1+4 \, v_y^2 & 2 \, v_y \\
2 \, v_y & 1
\end{array} 
\right)
$. Simple computations show that:
\begin{itemize}
\item[$\bullet$] If $v_y=0$, then $A$ has a double eigenvalue
$$
\lambda_0=1+2 \, v_y^2,
$$
\item[$\bullet$] and if $v_y \neq 0$, then it has two single eigenvalues
$$
\lambda_1=1+2 \, v_y^2+2\sqrt{v_y^2 \, (1+ v_y^2)} \quad \mbox{and} \quad \lambda_2=\lambda_1^{-1}.
$$
\end{itemize}  
The existence of $\lambda$ then follows from the bound on $ v_y$ in Lemma \ref{estimates-tw}. 
\end{proof}

We can now prove the theorem.

\begin{proof}[Proof of Theorem \ref{Holder-regularity}]
Let $(x_0,y_0) \in \R^2$ be arbitrary and consider the balls
$$
\Omega:=B\left((x_0,y_0),1\right) \quad \mbox{and} \quad \Omega':=B\left((x_0,y_0),1/2\right).
$$
Let us apply \cite[Theorem 8.24]{GiTr01} to Equation \eqref{ext-eqn} satisfied by $u \in H^1(\Omega)$ as stated in Lemma \ref{identification}; note that we are using the notation of \cite{GiTr01} for simplicity, so that $\Omega$ is not just the fresh region here. We choose $q=+\infty$, since $G$ is bounded, and note that the distance $d'=1/2$ between $\Omega'$ and the boundary of $\Omega$ does not depend on $(x_0,y_0)$. The estimate stated in \cite[Theorem 8.24]{GiTr01} then reads:
$$
|u(x,y)-u(\tilde{x},\tilde{y})| \leq C \times \left(\|u\|_{L^2(\Omega)}+\lambda^{-1} \, \|G\|_{\infty} \right) \times \left(|x-\tilde{x}|^\alpha+|y-\tilde{y}|^\alpha \right)
$$
for all $x,\tilde{x},y,\tilde{y} \in \overline{B ((x_0,y_0),1/2)}$ and
for some constants $C=C \left(\Lambda/\lambda,\nu\right) \geq 0$ and $\alpha=\alpha\left(\Lambda/\lambda,\nu\right)>0$, where $\lambda$, $\Lambda$ and $\nu$ come from Lemma \ref{checking-assumptions}. Let us recall that these three constants have the desired dependences (as stated in Theorem~\ref{Holder-regularity}). Moreover, by the definition of $G$ in Lemma \ref{identification} and Lemma \ref{estimates-tw}, it is clear that $\|u\|_{L^2(\Omega)}+\lambda^{-1} \, \|G\|_{\infty} \leq \tilde{C}$ for another constant having the same dependences. This completes the proof since all the constants above have the desired dependences and in particular do not depend on the arbitrary $(x_0,y_0)$.
\end{proof}

\section{Strong convergence of the gradients}\label{app-strong-conv}

In this appendix we establish the strong convergence of $\nabla u^\varepsilon$ during the homogenization of Problem \eqref{eps-T-eqn-bis}--\eqref{eps-front-eqn-bis}. The ideas of the proof are standard, see \cite{LuNgWa02}. Some details are given for completeness. This appendix is independent of the rest of the paper.  

\begin{proposition}
\label{Theorem-strongCVofU} 
Let the assumptions of Theorem \ref{Theorem-homogenization} or Theorem \ref{Theorem-homogenization-lambda} hold. We then have:
\begin{equation*}
\begin{cases}
\lim_{\varepsilon \downarrow 0} u_x^{\varepsilon }=u_x^{0} & \quad \mbox{(strongly) in} \quad L_{\loc}^2(\R_y;L^2(\R_x)),\\
\lim_{\varepsilon \downarrow 0} u_y^{\varepsilon }=0 & \quad \mbox{(strongly) in} \quad L_{\loc}^2(\R_y;L^2(\R_x)).\\
\end{cases}
\end{equation*}
\end{proposition}

\begin{proof}[Proof of Proposition \ref{Theorem-strongCVofU}]
We use the notation $\nabla u^0:=(u^0_x,0)$. By $\varepsilon$-periodicity of $u^\varepsilon$, it suffices to prove that
$$
\lim_{\varepsilon \downarrow 0}\int_{0}^{l_{\varepsilon }}\int_{-\infty
}^{v^{\varepsilon }(y)} \left|\nabla u^\varepsilon-\nabla u^{0}\right|^{2} \, \dif x \, \dif y=0
$$
(with $l_{\varepsilon }=\llcorner \frac{1}{\varepsilon }\lrcorner  \, \varepsilon $, $\llcorner 
\cdot \lrcorner$ lower integer part).
In fact, as $a^{\varepsilon }(y)\geq a_{m}>0,$ it suffices to prove that the
quantity 
\begin{equation*}
I:=\int_{0}^{l_{\varepsilon }}\int_{-\infty }^{v^{\varepsilon }(y)}a\left(\frac{y}{\varepsilon}\right) \left| \nabla u^{\varepsilon }-\nabla u^{0}\right|
^{2} \, \dif x \, \dif y
\end{equation*}
goes to zero. A simple expansion of $I$ gives 
$$
I=I_{1}+I_{2}+I_{3},
$$
where
\begin{eqnarray*}
I_{1} & := & \iint a\left(\frac{y}{\varepsilon }\right) \left| \nabla u^{\varepsilon
}\right|^{2},\\
I_{2} & := & \iint a\left(\frac{y}{\varepsilon }\right) \left| \nabla u^{0}\right|^{2},\\
I_{3} & := & -2\iint a\left(\frac{y}{\varepsilon }\right) \nabla u^{\varepsilon} \, \nabla u^{0}
\end{eqnarray*}
(with the same domain of integration as $I$).
We will pass in the limit $\varepsilon \downarrow 0$ term by term. In what follows, we will only need the convergences stated in Theorems \ref{Theorem-homogenization} and \ref{Theorem-homogenization-lambda}, that is to say: $c^{\varepsilon } \to c^{0},$ $v^{\varepsilon } \to 0$ uniformly on $\R$, and $u^\varepsilon \to u^0$ uniformly on $\R^2$. Recall that we have moreover $l_\varepsilon \to 1$ and $a \left(\frac{\cdot}{\varepsilon }\right) \to \overline{a}$ in $L^\infty(\R)$ weak-$\star$ (the same kind of result
holding also for the other periodic data $b$ and $g$). By taking $w=u^{\varepsilon }$ in Definition \ref{defn-eps-existence},
\begin{equation*}
I_{1} =\underbrace{-\int_{0}^{l_{\varepsilon }}\int_{-\infty }^{v^{\varepsilon
}(y)} \, c^{\varepsilon} \, b\left(\frac{y}{\varepsilon }\right) \left(\frac{(u^{\varepsilon })^{2}}{2}\right)_{x} \dif x \, \dif y}_{=\int_{0}^{l_{\varepsilon }} c^{\varepsilon} \, b\left(\frac{y}{\varepsilon }\right) \left(\frac{(u^{\varepsilon})^{2}}{2}\right) \left(v^{\varepsilon}(y),y\right) \, \dif y \mbox{ after integration}}+\int_{0}^{l_{\varepsilon }} c^{\varepsilon} \, g\left(\frac{y}{\varepsilon}\right) u^{\varepsilon }\left(v^{\varepsilon }(y),y\right)  \dif y
\end{equation*}%
so that in the limit $\varepsilon \downarrow 0,$ we have
\begin{equation*}
\lim_{\varepsilon \downarrow 0} I_{1}=-c^{0}\,\overline{b}\,\frac{\left(u^{0}(0)\right)^{2}}{2}+c^{0}\,\overline{g}\,u^{0}(0).
\end{equation*}%
and which can be rewritten as 
\begin{equation*}
\lim_{\varepsilon \rightarrow 0}I_{1}=-\int_{-\infty }^{0}c^{0}\,\overline{b}
\,u_{x}^{0} \, u^{0} \, \dif x+c^{0}\,\overline{g}\,u^{0}(0).
\end{equation*}%
As concerning $I_{2,}$ 
\begin{equation*}
\lim_{\varepsilon \downarrow 0}I_{2}=\lim_{\varepsilon \downarrow
0}\int_{0}^{l_{\varepsilon }} \int_{-\infty }^{v^{\varepsilon }(y)} a\left(\frac{y}{\varepsilon }\right) \left(u_{x}^{0}\right)^{2} \dif x \, \dif y=\int_{-\infty }^{0} \overline{a} \left(u_{x}^{0}\right)^{2} \dif x.
\end{equation*}%
Now we proceed with 
\begin{equation*}
I_{3}=-2\int_{0}^{l_{\varepsilon }}\int_{-\infty }^{v^{\varepsilon
}(y)} a\left(\frac{y}{\varepsilon }\right) u_{x}^{\varepsilon } \, u_{x}^{0} \, \dif x \, \dif y
\end{equation*}
and where the passage to the limit is not immediate because we have two weak
convergences. We thus use the regularity of $u^{0}$ and integrate by parts to obtain
\begin{equation*}
I_{3}=J_{1}+J_{2},
\end{equation*}%
with 
\begin{eqnarray*}
J_{1} & := & 2\int_{0}^{l_{\varepsilon }}\int_{-\infty }^{v^{\varepsilon }(y)} a\left(\frac{y}{\varepsilon }\right) u^{\varepsilon} \, u_{xx}^{0} \, \dif x \, \dif y,\\
J_{2} & := & -2 \int_{0}^{l_{\varepsilon }} a\left(\frac{y}{\varepsilon }\right)
u^{\varepsilon}\left(v^{\varepsilon }(y),y\right) u_{x}^{0}\left(v^{\varepsilon }(y)\right) \dif y.
\end{eqnarray*}
In the limit $\varepsilon \rightarrow 0,$ we then have%
\begin{eqnarray*}
\lim_{\varepsilon \downarrow 0}J_{1} =2\int_{-\infty }^{0}\overline{a} \,
u^{0} \, u_{xx}^{0} \, \dif x =-2\int_{-\infty }^{0}\overline{a} \left(u_{x}^{0}\right)^{2} \dif x+2\, \overline{a} \, u^{0}(0) \, u_{x}^{0}(0),
\end{eqnarray*}%
and $\lim_{\varepsilon \downarrow 0}J_{2}=-2 \, \overline{a} \, u^{0}(0) \, u_{x}^{0}(0).$
Finally by adding all the limits, we obtain%
\begin{equation*}
\lim_{\varepsilon \downarrow 0} I=-\int_{-\infty }^{0} \left\{c^{0}\,\overline{b} \, u_{x}^{0} \, u^{0}+\overline{a} \left(u_{x}^{0}\right)^{2}\right\} \dif x+c^{0}\,\overline{g}\,u^{0}(0). 
\end{equation*}
On the right-hand side, we recognize the weak formulation of Problem \eqref{homogenized-T-eqn}, with $u^0$ itself as a test function. Standard computations, consisting in multiplying the equation in \eqref{homogenized-T-eqn}, integrating by parts, and using the boundary condition, then show that this right-hand side equals zero and the proof is complete.
\end{proof}

\section{Main notations}\label{app-notations}

(Essentially by order of the first occurence in the paper.)

\begin{longtable}{p{110pt} p{230pt}}
$\R^+$ & set of positive reals (excluding $0$)\\
\smallskip
$Y$ & period used for the existence of travelling waves\\
\smallskip
$\Omega$ & fresh region $\{x<v(y)\}$ (for a given $Y$-periodic $v$)\\
\smallskip
$\Gamma$ & position of front $\{x=v(y)\}$\\
\smallskip
$\Omega_\diese$ & $\Omega \cap \{0<y<Y\}$\\
\smallskip
$\Gamma_\diese$ & $\Gamma \cap \{0<y<Y\}$\\
\smallskip
$L^p_\diese$, $H^1_\diese$, etc. & spaces of functions $Y$-periodic in $y$\\ 
\smallskip
$\overline{f}$ & mean value $\frac{1}{Y} \int_0^Y f$ of a $Y$-periodic $f=f(y)$\\
\smallskip
$\varepsilon$ & period used for the homogenization process\\
\smallskip
$\Omega^\varepsilon$ & fresh region $\{x<v^\varepsilon(y)\}$\\
\smallskip
$\Gamma^\varepsilon$ & position of the front $\{x=v^\varepsilon(y)\}$\\
\smallskip
$\Omega_\per^\varepsilon$ & $\Omega^\varepsilon \cap \{0<y<\varepsilon\}$\\
\smallskip
$\Gamma_\per^\varepsilon$ & $\Gamma^\varepsilon \cap \{0<y<\varepsilon\}$\\	
\smallskip
$L_\per^p$, $H_\per^1$, etc. & spaces of functions $\varepsilon$-periodic in $y$\\ 
\smallskip
$\overline{f}$ & mean value of an $\varepsilon$-periodic $f=f(y)$\\
$\llcorner \cdot \lrcorner$, $\ulcorner \cdot \urcorner$ & lower and upper integer parts\\
\smallskip
$w^\varepsilon(z)=\frac{v^\varepsilon(\varepsilon \, z)}{\varepsilon}$ & rescaled front\\
\smallskip
$c(\lambda)$ & shorthand notation of $c^0(\lambda)$ in Section \ref{monotonicity}\\
\smallskip
$h(z)=w_z(z)=h(z,\lambda)$ & solution of Equation \eqref{Eqn1-Cmonotone}\\ 
\smallskip
$\mathscr{R}(z)=R\left(z,\frac{\overline{g}}{\overline{b}} \right)$ & combustion rate of Equation \eqref{Eqn1-Cmonotone}\\
\smallskip
$\dif_{x_i} \phi$ & partial differential in Banach spaces\\
\smallskip
$\dif_{x_i} \phi(x_1,x_2) \cdot h_i$ & partial differential at $(x_1,x_2)$ in the direction $h_i$\\
\smallskip
$\mathcal{L}(E,F)$ & space of bounded and linear maps\\
\smallskip
$\mbox{Isom} (E,F)$ & space of isomorphisms of Banach spaces\\
\smallskip
${\co}_\varepsilon(1)$, $\co(\varepsilon)$, $\cO (\varepsilon)$, etc. & usual Landau notations
\end{longtable} 


\end{document}